\documentclass[11pt,reqno]{amsart}
\usepackage{amsmath,amssymb,amsthm,geometry,graphicx,mathrsfs,cite}
\usepackage[dvipsnames]{xcolor}
\usepackage{hyperref}

\hypersetup{
    linkcolor = violet,
    citecolor = NavyBlue,
}

\newtheorem{theorem}{Theorem}[section]

\newtheorem{conjecture}[theorem]{Conjecture}
\newtheorem{example}[theorem]{Example}
\newtheorem{lemma}[theorem]{Lemma}
\newtheorem{proposition}[theorem]{Proposition}
\newtheorem{corollary}[theorem]{Corollary}

\newcommand{\Seven}{S_{\mathrm{even}}}
\newcommand{\Sodd}{S_{\mathrm{odd}}}

\theoremstyle{remark}
\newtheorem*{remark}{Remark}

\newcommand{\psibar}{\overline{\psi}}
\newcommand{\xibar}{\overline{\xi}}
\newcommand{\exactdivide}{\mid\!\mid}
\newcommand{\qr}[2]{\genfrac{(}{)}{}{}{#1}{#2}}
\newcommand{\Q}{\mathbb{Q}}

\newcommand{\G}{\mathcal{G}}

\renewcommand{\mod}[1]{{\ifmmode\text{\rm\ (mod~$#1$)}\else\discretionary{}{}{\hbox{ }}\rm(mod~$#1$)\fi}}

\numberwithin{equation}{section}

\newcommand{\abs}[1]{\left\lvert#1\right\rvert}

\author{J.~W.~Bober, Z.~Ge, M.~B.~Milinovich}

\date{\today}

\title[Biases in gaps between zeros of Dirichlet $L$-functions]{Biases in gaps between zeros of Dirichlet $L$-functions}

\begin{document}

\begin{abstract}
%We describe a thin family of Dirichlet $L$-functions that provably have an unusual value distribution. Experimentally, we find that the $L$-functions in this family have a significant and previously undetected bias in the distribution of gaps between their zeros. These phenomena appear to have an arithmetic explanation that corresponds to the nonvanishing of a certain Gauss-type sum. We give a complete classification of the characters for which this sum is nonzero and count the number of corresponding characters. It turns out that this Gauss-type sum vanishes for 100\% of primitive Dirichlet characters, so $L$-functions in our newly discovered family are rare (a set of density zero among all primitive characters). The nonvanishing of this Gauss-type sum implies that there is a main term when one integrates the $Z$-function associated to a Dirichlet $L$-function over a long interval. We prove that the $L$-functions in this thin family are the only GL$(1)/\mathbb{Q}$ or GL$(2)/\mathbb{Q}$ cuspidal automorphic $L$-functions whose $Z$-functions have a ``large'' mean value when integrated.
We describe a new family of Dirichlet $L$-functions that provably have an unusual value distribution. Experimentally, we find that the $L$-functions in this family have a significant and previously undetected bias in the distribution of gaps between their zeros. These phenomena appear to have an arithmetic explanation that corresponds to the nonvanishing of a certain Gauss-type sum. We give a complete classification of the characters for which this sum is nonzero and count them. This sum vanishes for 100\% of primitive Dirichlet characters, so $L$-functions in our newly discovered family are rare (a set of density zero among all primitive characters). The nonvanishing of this sum implies that there is a main term when one integrates the $Z$-function associated to a Dirichlet $L$-function over a long interval. We prove that the $L$-functions in this family are the only GL$(1)/\mathbb{Q}$ or GL$(2)/\mathbb{Q}$ cuspidal automorphic $L$-functions whose $Z$-functions have a ``large'' mean value when integrated. We discuss analogous results for the Riemann zeta-function.
\end{abstract}

%A.~Ivi\'c conjectured that the integral of Hardy's $Z$-function from $0$ to $T$ is $O(T^{1/4})$ and $\Omega_{\pm}(T^{1/4})$ as $T\to\infty$. This conjecture has proved recently by M.~A.~Korolev and M.~Jutila independently. In this paper, we study the analogous problem for $Z$-functions associated to general $L$-functions. We show that the analogous conjecture is false for the $Z$-functions of certain ``special" Dirichlet $L$-functions. In particular, we show that the integral of the $Z$-function of a Dirichlet $L$-functions from $0$ to $T$ is asymptotic to $C_\chi T^{3/4}+O(T^{1/4+\varepsilon})$ and we classify precisely when the constant $C_\chi$ is nonzero. Somewhat surprisingly, numerical evidence seems to suggest that the unexpectedly large mean value is caused by a currently unexplained parity bias in the gaps between the zeros of these ``special" Dirichlet $L$-functions.

\maketitle

\section{Introduction}

Let $\chi$ (mod $q$) be a primitive Dirichlet character. In this paper, we investigate what appears to be a deep connection between the distribution of the nontrivial zeros of the Dirichlet $L$-function $L(s,\chi)$, and the Gauss-type sum
\begin{equation}\label{Schi}
S(\chi)=\sum_{n=1}^{2q}\chi(n)\,e\!\left(\frac{-n^2}{2q}\right),
\end{equation}
where as usual we let $e(x)=e^{2\pi i x}$. Notice that $\chi$ is periodic (mod $q$) and the additive character is periodic (mod $2q$), so $S(\chi)$ is a quadratic Gauss sum (mod $2q$) twisted by an imprimitive character. In a natural sense, we prove that $S(\chi)=0$ for 100\% of primitive characters. However, we prove that there is an infinite (but zero-density) family of characters such that $S(\chi)\ne 0$. Proposition \ref{prop:S1} gives a complete classification of this thin family of Dirichlet characters. 

\smallskip

In Theorem \ref{thm:GL1}, we show that the integral of the $Z$-function,
$Z_\chi(t)$, associated to the character $\chi$ has a main term when
$S(\chi)\ne 0$ that is not present if $S(\chi) = 0$. %For this reason, if $S(\chi)\ne 0$, then we say that the character $\chi$ is {\it $Z$-biased}.
Moreover, when $S(\chi)\ne 0$ it appears that there is a significant and
previously unobserved bias in the gaps between the nontrivial zeros of
$L(s,\chi)$. Experimentally, we give numerous examples of this newly discovered
bias. % in Section XXX.

\subsection{\texorpdfstring{The Hardy $Z$-function}{The Hardy Z-function}}

The sum $S(\chi)$ naturally arises when considering the average of
$Z_\chi(t)$. To describe
this result, let us first recall how to define the Hardy $Z$-function. The functional
equation for the Riemann zeta-function, $\zeta(s)$, is given by \[
\Gamma_\mathbb{R}(s) \, \zeta(s) = \Gamma_\mathbb{R}(1\!-\!s) \, \zeta(1\!-\!s),
\]
where  $\Gamma_\mathbb{R}(s)=\pi^{-{s/2}} \, \Gamma(s/2)$. The Hardy $Z$-function is defined by
\[
Z(t) = 
\left( 
\frac{\Gamma_\mathbb{R}(\frac{1}{2}\!+\!it)}
{\Gamma_\mathbb{R}(\frac{1}{2}\!-\!it)}\right)^{1/2} \zeta(\tfrac{1}{2}\!+\!it).
\]
By the functional equation, we observe that $\overline{Z(t)} = Z(t)$, so
$Z(t)$ is real-valued for real $t$. Moreover,  $|Z(t)| =
|\zeta(\tfrac{1}{2}+it)|$, so the zeros of $Z(t)$ correspond to zeros of
$\zeta(s)$ on the critical line, $\mathrm{Re}(s)=1/2$, and $Z(t)$ changes sign
when $t$ is an ordinate of a zero of $\zeta(s)$ on the critical line of odd
multiplicity. This makes $Z(t)$ useful for understanding the distribution of
the nontrivial zeros of $\zeta(s)$. For example, if we define the integral
\begin{equation}\label{int:Z}
F(T) = \int_0^{T}Z(t)\, \mathrm{d}t,
\end{equation}
then the estimate $F(T) = o(T),$ as $T\to \infty$, implies Hardy's theorem that
infinitely many nontrivial zeros of the zeta function are on the critical line
(see \cite[Section 10.5]{titchmarsh}).
%{(\color{red} Insert a picture of $Z(t)$? We probably don't need this)}

\smallskip

One might expect a great deal of cancellation in the integral $F(T)$, as
$T\to\infty$, since $Z(t)$ oscillates in sign. Indeed,
Titchmarsh \cite[Eq.~10.5.1]{titchmarsh} proved that $F(T)=O(T^{7/8})$.
Ivi\'{c} \cite{ivic} dramatically improved this by showing that
$F(T)=O_\varepsilon(T^{1/4+\varepsilon})$ for all $\varepsilon>0$. In the same
paper, Ivi\'c  conjectured that in fact  
\begin{equation}\label{ivic}
F(T)=O(T^{1/4}) \quad \text{and}  \quad F(T)=\Omega_\pm(T^{1/4}).
\end{equation} 
Ivi\'c's conjecture has been independently established by Korolev \cite{korolev1,korolev2} and Jutila \cite{jutila1, jutila2} in a rather precise way. Let
$\theta_0=\min\big(|\{\sqrt{T/2\pi}\}-\tfrac14|,\,|\{\sqrt{T/2\pi}\}-\tfrac34|\big)$.
Jutila \cite{jutila2} proved that if $\theta_0\neq 0$, then
\[
\int_{0}^{T}Z(t)\,\mathrm{d}t = \left(\frac{T}{2\pi}\right)^{1/4} (-1)^{\lfloor\sqrt{T/2\pi}\rfloor} \, K\!\left(\Big\{\sqrt{T/2\pi}\Big\}\right) + O\big(T^{1/6}\log T\big) + O\big(\min\{T^{1/4},\,T^{1/8}\theta_0^{-3/4}\}\big),
\]
where $K(x)=2\pi$ for $\frac14<x<\frac34$, and $K(x)=0$ for $0\le x<\frac14$ and for $\frac34<x<1$ (see Figure \ref{Fplot}). He also proved that if $\theta_0=0$, then the same formula holds with error term $O(T^{1/6}\log T)$, where $K(\frac14)=\frac{4\pi}{3}$ and $K(\frac34)=\frac{2\pi}{3}$.

\begin{figure}
        \centering
        \includegraphics[scale=0.8]{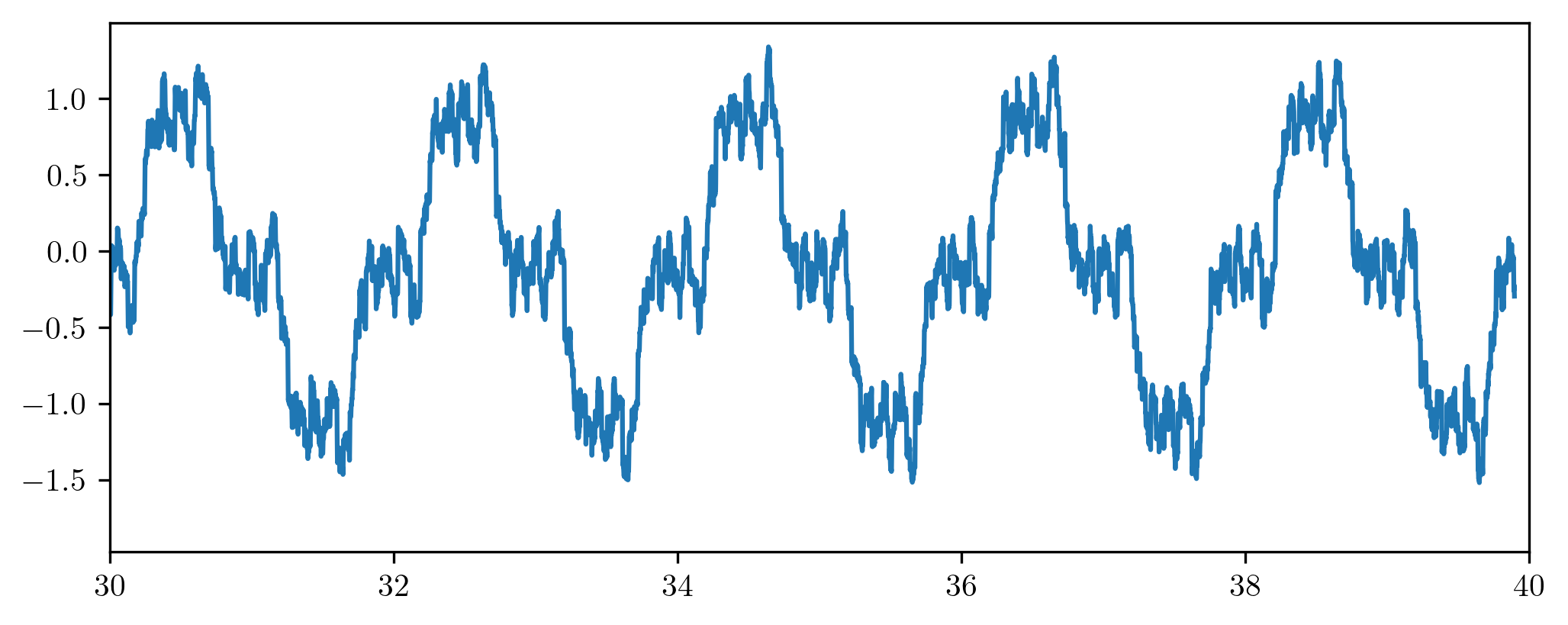} 
        \caption{A plot of $\displaystyle \frac{1}{2\pi \sqrt{T}}F\big( 2\pi T^2 \big)$ for the function $F(T)$ defined in \eqref{int:Z}.}
        \label{Fplot}
\end{figure}

\subsection{\texorpdfstring{The $Z$-function associated to a Dirichlet $L$-function}{The Z-function associated to a Dirichlet L-function}}
Let $\chi$ (mod $q$) be a primitive Dirichlet character.  Then $L(s,\chi)$
satisfies the  functional equation, in asymmetric form,
\[
L(s,\chi) = 
\mathcal{G}_\chi(s) \, L(1\!-\!s,\overline{\chi}),
\quad \textup{where}
\quad
\G_\chi(s) 
= \varepsilon_\chi \,q^{1/2-s} 
\frac{\Gamma_{\mathbb{R}}(1\!-\!s\!+\!\mathfrak{a})}{\Gamma_{\mathbb{R}}(s\!+\!\mathfrak{a})}.
%=\frac{\tau(\chi)}{q^{1/2}}q^{1/2-s}\frac{\pi^{(1-s)/2}\Gamma((1-s)/2)}{\pi^{s/2}\Gamma(s/2)}.
\]
Here, $\mathfrak{a}\in\{0,1\}$ is such that $\chi(-1)=(-1)^{\mathfrak{a}}$ and $\varepsilon_\chi =i^{-\mathfrak{a}} \tau(\chi)/\sqrt{q}$,
where
\[
\tau(\chi) = \sum_{n=1}^q \chi(n) \, e\Big(\frac{n}{q}\Big)
\]
is the classical Gauss sum satisfying $|\tau(\chi)|=\sqrt{q}$, so that $|\varepsilon_\chi|=1$. In analogy with the definition of $Z(t)$ for the Riemann zeta-function, we define
\[
Z_\chi(t) = \G_\chi(\tfrac{1}{2}\!+\!it)^{-1/2}L(\tfrac{1}{2}\!+\!it,\chi).
\]
By the functional equation, it again follows that $\overline{Z_\chi(t)} =
Z_\chi(t)$, so $Z_\chi(t)$ is real-valued for real $t$. Moreover,  $|Z_\chi(t)|
= |L(\tfrac{1}{2}+it,\chi)|$, so the zeros of $Z_\chi(t)$ correspond to zeros
of $L(s,\chi)$ on the critical line and $Z_\chi(t)$ changes sign when $t$ is an
ordinate of a zero of $L(s,\chi)$ on the critical line of odd multiplicity. 
A well-known conjecture of Chowla \cite{chowla} is that $L(\frac{1}{2},\chi) \ne
0$ when $\chi$ is a primitive quadratic character. It is now generally believed that $L(\frac{1}{2},\chi) \ne
0$ for all primitive characters $\chi$, which would imply that $Z_\chi(0)\ne 0$. When $Z_\chi(0) \ne 0$, we choose
a branch of $\G_\chi(s)^{-1/2}$ so that $Z_\chi(0) > 0$. For quadratic characters, this means that $Z_\chi(0) = L(1/2,\chi) >0$ assuming Chowla's conjecture and GRH. The usual square root for the Hardy $Z$-function is chosen so that $Z(0) = \zeta(1/2)<0$. In both cases, the value of the $Z$-function at $0$ is chosen to agree with the central value of the $L$-function.

%{\color{blue}An application of Waldspurger's formula (see \cite[Proposition 2.2]{jls}) implies $L(1/2,\chi)\geq 0$ for all primitive quadratic $\chi$. Moreover, the classical conjecture of Chowla \cite{chowla} asserts that $L(\frac{1}{2},\chi) \ne 0$ for all primitive quadratic character $\chi$. Thus, a natural choice of branch of $\tau(\chi)^{1/2}$ (hence $\G_\chi(s)^{-1/2}$), for primitive quadratic $\chi$, would satisfy $Z_\chi(0)=L(1/2,\chi)>0$. Throughout the paper, for general primitive character $\chi$ with $Z_\chi(0)\neq 0$, we choose the branch of $\tau(\chi)^{1/2}$ so that $Z_\chi(0) > 0$.}

\smallskip

Work of Rudnick and Sarnak \cite{rs} suggests that $n$-level correlations of the ordinates of the nontrivial zeros of an $L$-function associated to {\it any} cuspidal automorphic representation of GL$(m)$ over $\mathbb{Q}$ should correspond to $n$-level correlations of eigenvalues of random matrices in the Gaussian unitary ensemble (GUE). This led Keating and Snaith \cite{ks1, ks2} to suggest that the value distribution of any principal automorphic $L$-function on its critical line is related to that of the characteristic polynomials of random unitary matrices. An explicit description of this correspondence in terms of $Z$-functions is given by Conrey, Farmer, Keating, Rubinstein, and Snaith in \cite[Section 1]{cfkrs}. Hence, using the random matrix model for $L$-functions, one might expect an analogue of Ivi\'c's conjecture for $F(T)$ in \eqref{ivic} to hold for the integral average of $Z_\chi(t)$ for any primitive character $\chi$. In other words, is it true that
\[
\int_0^T Z_\chi(t) \, \mathrm{d}t = O_q(T^{1/4})  \quad \text{and} \quad  \int_0^T Z_\chi(t) \, \mathrm{d}t = \Omega_{\pm}(T^{1/4})?
\]
When $S(\chi) \ne 0$, the first estimate fails (and the absolute value of the integral is $\gg_q T^{3/4}$).

\begin{theorem}\label{thm:GL1}
Let $\chi$ be a primitive Dirichlet character \textup{(mod }$q$\textup{)}. As
$T \rightarrow \infty$,  we have
\begin{equation} \label{integralZ}
\int_{0}^{T}Z_{\chi}(t) \, \mathrm{d}t
=
C_\chi \,
  T^{\frac{3}{4}} + O_q(T^{\frac{1}{4}} \log T),
\end{equation}
where 
\begin{equation}\label{C}
C_\chi=\frac{2^{\frac{3}{4}} \pi^{\frac{1}{4}}
  e\!\left(\frac{1}{16}\right)}{3\,  q^{\frac{1}{2}} \tau(\chi)^{\frac{1}{2}}} \, S(\chi)
\end{equation}
and $S(\chi)$ is defined in \eqref{Schi}. 
Moreover, there are infinitely many $\chi$ for which $S(\chi)\ne 0$.
\end{theorem}

In Theorem \ref{thm:GL1}, $\tau(\chi)^{1/2}$ denotes the square root determined by the branch of $\mathcal{G}_\chi(s)^{-1/2}$ used to define $Z_\chi(t)$. In particular, if $L(\frac12,\chi)\ne0$, then it is the square root for which $Z_\chi(0)>0$.

\smallskip

%\begin{figure}
%        \centering
%        \includegraphics[scale=0.35]{mod8} 
%        \caption{A plot of $\displaystyle \int_0^T Z_\chi(t) \, \mathrm{d}t \,$ for the two primitive characters (mod $8$).}
%        \label{mod8plots}
%\end{figure}

%\begin{figure}
%        \centering
%        \includegraphics[scale=0.35]{Z8normalized} 
%        \caption{A plot of $\displaystyle \frac{1}{T^{3/4}}\int_0^T Z_\chi(t) \, \mathrm{d}t \,$ for the even primitive character (mod $8$). Note that $\displaystyle{\frac{2^{3/2} \pi^{1/4}}{3}=1.255195\ldots}$}
%         \label{mod8plots2}
%\end{figure}

\begin{figure}
    \includegraphics[width=\textwidth]{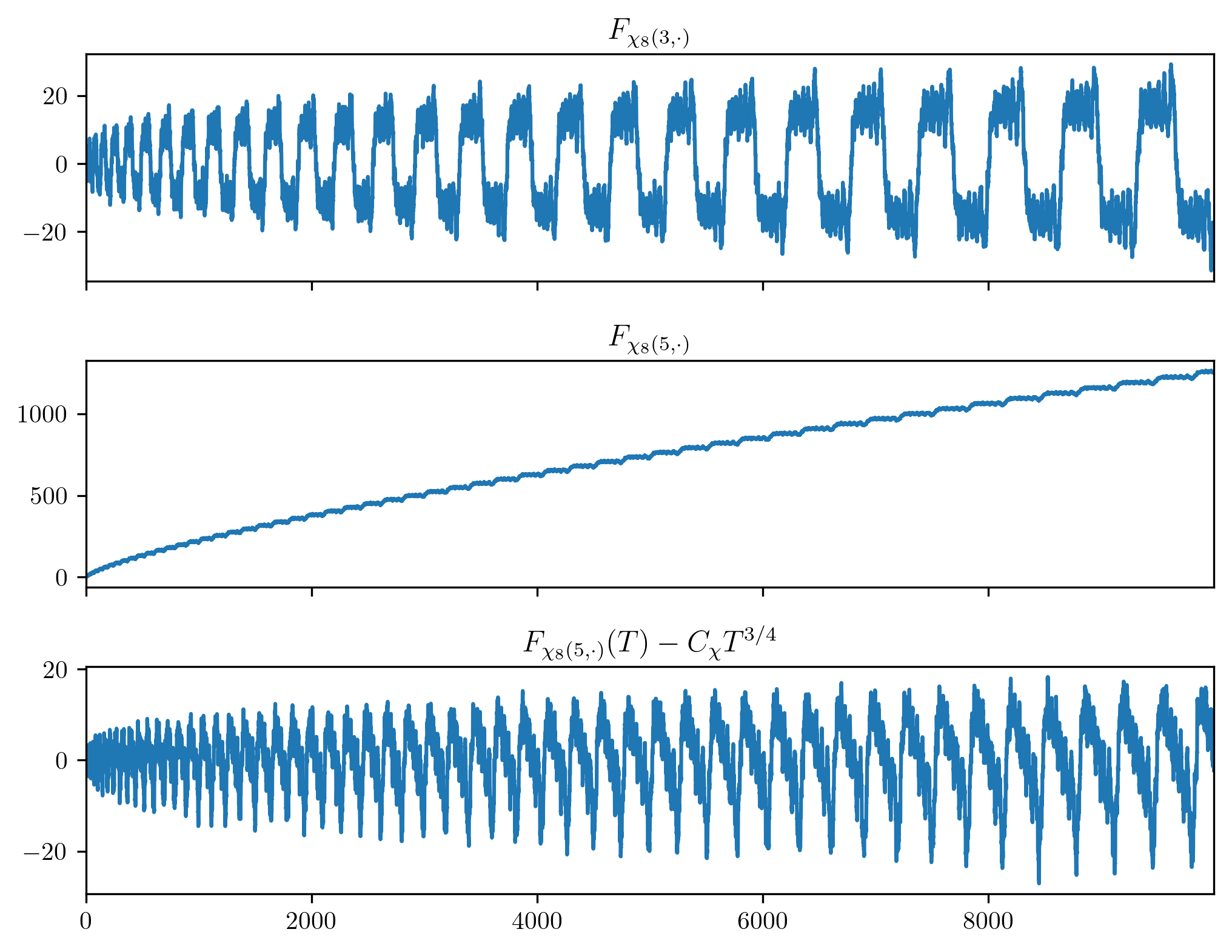}
    \caption{$F_\chi(T) := \int_0^T Z_{\chi}(t)\,\mathrm{d}t$ for the odd (top)
        and even (middle) primitive characters mod $8$, and the even primitive
        character (bottom) with the main term subtracted.
        Here $C_{\chi_8(5, \cdot)} = \displaystyle{\frac{2^{3/2} \pi^{1/4}}{3}=1.255195\ldots}$.
        In this and other examples we use the LMFDB notation for Dirichlet Characters.
        See \cite{lmfdb:character-labels}.}
        \label{mod8plots}
\end{figure}

{\sc Remarks}.

\smallskip

($i$) Since $Z_{\chi}(t)$ is real-valued, we see that the constant $C_\chi$ defined in Theorem \ref{thm:GL1} is real. This is not immediate from the definition of $C_\chi$ and we give a direct proof of this fact in Appendix \ref{appendix}.

\smallskip

($ii$) Notice that $C_\chi \ne 0$ if and only if $S(\chi)\ne 0$. In Section
\ref{section:distribution}, we give another evaluation of $C_\chi$ in the case
where it is nonzero, and speculate on its distribution.

\smallskip
 
($iii$) We classify the primitive characters $\chi$ for which $S(\chi)\ne 0$ in Proposition \ref{prop:S1} and then use this classification to count the number of such $\chi$ in Theorem \ref{thm:counting_q} and Corollary \ref{sd}. 

\smallskip

%($iii$) In the case where $S(\chi)$ is nonzero there will always be a primitive
%character $\psi$ mod $2q$ for which $\psi^2 = \chi$. Suppose that
%$q = 2^a q_1 q_2$, where $q_1$ and $q_2$ are odd and coprime, $q_1$ is square-full
%and $q_2$ is square-free, and that $\psi = \prod_{p|q}\psi_p$, where each $\psi_p$
%is a character mod $p^{v_p(q)}$. Then
%\[
%C_\chi = \pm \frac{2^{\frac{3}{4} + a' + \omega(q_1q_2) } \pi^{\frac{1}{4}}}{3\, q^{\frac{1}{4}}}
%    \prod_{p|q_2}
%        \left(\mathrm{Re}\left(\frac{J(\psi_p,\epsilon_p)}{\tau(\epsilon_p)}\right) \pm 1\right),
%\]
%where $\epsilon_p$ is the quadratic character mod $p$ and $J(\psi_p, \epsilon_p)$
%is the Jacobi sum.
 
($iv$) We have not tried to optimize the error term in \eqref{integralZ}. Instead we keep our proof relatively simple and we focus our attention on classifying the characters for which $S(\chi)\ne 0$. With more work, we can show that the error term in \eqref{integralZ} is $O_q(T^{1/4})$ and $\Omega_\pm(T^{1/4})$, as discussed at the end of this subsection.

\smallskip
 
($v$) There are other situations where averages of $L$-functions over families have lower-order main terms that have not been predicted using random matrix theory models. For example, the work of Diaconu \cite{diaconu} and Diaconu and Whitehead \cite{dw} has shown that there is a lower-order main term in the third moment of quadratic Dirichlet $L$-functions over both number fields and function fields. Also, Florea \cite{florea} has shown that there is a lower-order main term in the first moment of quadratic Dirichlet $L$-functions over function fields (see also the recent work of \v{C}ech \cite{cech} and of Gao and Zhao \cite{gz}). Similar lower-order main terms appear in averages of Hecke $L$-functions attached to $\ell$th order residue symbols. For example, see Hamdar \cite[Remark 1.3 and Conjecture 1.5]{hamdar}. 

\smallskip
 
\begin{example}\label{ex1}
Let $\chi_8$ and $\chi_{-8}$ be the even and odd primitive characters \textup{(mod }$8$\textup{)}, respectively. Then $S(\chi_8)= 8 \, e(-1/16)$ and $S(\chi_{-8})=0$. The character $\chi_8$ \textup{(mod }$8$\textup{)} is the character of smallest modulus for which $S(\chi)\ne 0$. For these characters, Theorem \ref{thm:GL1} gives
\[
\int_{0}^{T}Z_{\chi_8}(t) \, \mathrm{d}t = \frac{2^{3/2}\pi^{1/4}}{3} T^{3/4} + O(T^{1/4}\log T) \quad \text{and} \quad  \int_{0}^{T}Z_{\chi_{-8}}(t) \, \mathrm{d}t = O(T^{1/4}\log T).
\]
These integrals are plotted, as a function of $T$, in Figure \ref{mod8plots}. 
\end{example}

\begin{example}\label{ex2}
    Let $\chi_1$ be the even primitive character \textup{(mod }$16$\textup{)} with
    $\chi_1(5)=-i$, and let $\chi_2, \chi_3,$ and $\chi_4$ denote the other 3
    primitive characters  \textup{(mod }$16$\textup{)}. Then $S(\chi_1) = 16 \,
    e(-1/32)$ and $S(\chi_i)=0$ for $i=2,3,4$. For these characters, Theorem
    \ref{thm:GL1} gives
\[
\int_{0}^{T}Z_{\chi_1}(t) \, \mathrm{d}t = \frac{2^{7/4}\pi^{1/4}}{3} T^{3/4} + O(T^{1/4}\log T) \quad \text{and} \quad  \int_{0}^{T}Z_{\chi_i}(t) \, \mathrm{d}t = O(T^{1/4}\log T), 
\]
for $i=2,3,4$. These integrals are plotted, as a function of $T$, in Figure \ref{mod16plots}.
\end{example}

The four primitive characters (mod 16) come in two conjugate pairs. This example shows that there are primitive characters $\chi$ for which $S(\chi)\ne 0$ and $S(\overline{\chi}) = 0$. By the functional equation, we have 
\[
\int_0^T Z_{\chi}(t) \, \mathrm{d}t = \int_{-T}^0 Z_{\overline{\chi}}(t) \, \mathrm{d}t.
\]
Therefore, for the even character $\chi_1$ (mod 16) with $\chi_1(5)=-i$ in Example \ref{ex2}, we have
\[
\int_{0}^{T}Z_{\chi_1}(t) \, \mathrm{d}t = \frac{2^{7/4}\pi^{1/4}}{3} T^{3/4} + O(T^{1/4}\log T) \quad \text{and} \quad  \int_{-T}^{0}Z_{\chi_1}(t) \, \mathrm{d}t = O(T^{1/4}\log T). 
\]
So the asymptotic behavior of these integrals, as functions of $T$, is not symmetric about $T=0$.

\begin{example}
    There are 9 primitive even characters \textup{(mod 88)}. These are the quadratic
    character and 8 characters of order 10, which can be distinguished by
    the values of $\chi(13)$. There is one Galois orbit for which
    $\chi(13) = e(a/5)$, for $a = 1,2,3,4$ and another Galois orbit
    for which $\chi(13) = e(a/10)$ for $a = 1,3,7,9$. In the first
    orbit, $C_\chi$ is 0. In the second orbit, $C_\chi$ takes the
    values $1.076\ldots, -0.2616\ldots, 1.353\ldots, 0.861\ldots$
    for $a = 1,3,7,9$. This is the smallest modulus for which there is a character $\chi$ with $C_\chi<0$. \textup{(}See Figure \ref{mod16plots}.\textup{)}
%1.07619437869409
%-0.261622704460426
%1.35340148620804
%0.861363849732835
\end{example}

 \smallskip

 We can prove an appropriately modified version of Ivi\'c's conjecture in \eqref{ivic} and show
that
\[
\int_{0}^{T}Z_{\chi}(t) \, \mathrm{d}t - C_\chi \, T^{3/4} = \begin{cases}
           O_q(T^{\frac{1}{4}} ), \quad \text{and}
           \\
          \Omega_{\pm}(T^{1/4}),
       \end{cases}
\]
where $C_\chi$ is the constant defined in \eqref{C}. Such a result is quite technical and is not
the focus of this paper. For example, for $\chi_8$ and $\chi_{-8}$, the even and odd characters \textup{(mod }$8$\textup{)} of Example \ref{ex1}, define
\[
\mathcal K_{\chi_8}(u)=-\pi\Big(\{u\}-\frac12\Big)
\qquad\text{and}\qquad
\mathcal K_{\chi_{-8}}(u)=-\frac\pi2\,(-1)^{\lfloor u+\frac12\rfloor}.
\]
Then we can prove that, for any fixed $\delta>0$,
\[
\int_0^TZ_{\chi_8}(t)\,\mathrm dt=\frac{2^{3/2}\pi^{1/4}}{3}T^{3/4}
+\Big(\frac{4T}{\pi}\Big)^{1/4}\mathcal K_{\chi_{8}}\Big(\sqrt{T/\pi}\Big)
+O_\delta\big(T^{1/6}\log T\big)
\]
when $\sqrt{T/\pi}$ is at distance at least $\delta$ from the integers, and
\[
\int_0^TZ_{\chi_{-8}}(t)\,\mathrm dt=\Big(\frac{4T}{\pi}\Big)^{1/4}
\mathcal K_{\chi_{-8}}\Big(\sqrt{T/\pi}\Big)+O_\delta\big(T^{1/6}\log T\big)
\]
when $\sqrt{T/\pi}$ is at distance at least $\delta$ from $\mathbb{Z}+\frac12$. Normalized versions of these integrals are plotted in Figure \ref{mod8fun}.
Similar results, for all primitive characters, will appear in a companion paper.

\subsection{A bias in the gaps between zeros of \texorpdfstring{$L(s,\chi)$}{L(s,chi)} on the critical line when \texorpdfstring{$S(\chi) \ne 0$}{S(chi) is not 0}} 
If $S(\chi) \ne 0$, then Theorem \ref{thm:GL1} gives
\[
\int_{0}^{T}Z_{\chi}(t) \, \mathrm{d}t=C_\chi \, T^{\frac{3}{4}} + O_q(T^{\frac{1}{4}} \log T)
\]
with $C_\chi \ne 0$. The integral on the left-hand side is real, so $C_\chi$ is real and nonzero. If $C_\chi > 0$ (resp.~$C_\chi < 0$), then $Z_{\chi}(t)$ has a bias towards positive (resp.~negative) values.  There are a number of reasons that this might happen. For instance, $Z_{\chi}(t)$ might get larger in the positive direction than it does in the negative direction. Or it might be the case that the measure of the set where $Z_{\chi}(t)>0$ is larger than the  measure of the set where $Z_{\chi}(t)<0$. Though the first possibility is quite interesting, we experimentally investigated the second possibility. Since $Z_\chi(t)$ conjecturally changes sign at every zero, it is natural to study the spacings between consecutive zeros.

%\begin{figure}

%        \includegraphics[scale=0.35]{mod16} % second figure itself
%        \caption{A plot of $\displaystyle \int_0^T Z_\chi(t) \, \mathrm{d}t \,$ for the four primitive characters (mod $16$).}
%	 \label{mod16plots}
%\end{figure}

\begin{figure}
    \includegraphics[width=\textwidth]{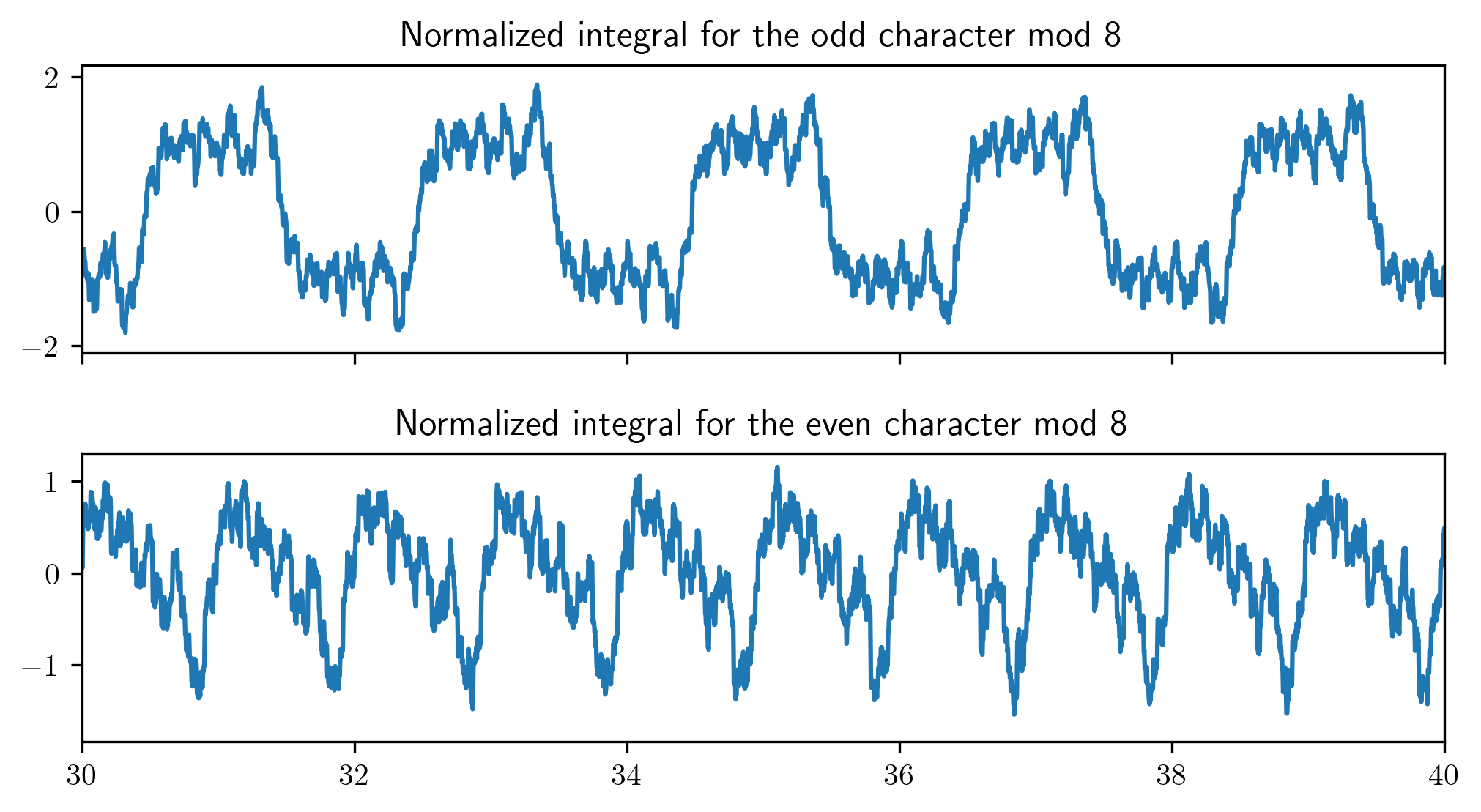}
    \caption{$\frac{\sqrt{2}}{\pi T^{1/2}}\left(F_\chi(\pi T^2) - C_\chi (\pi T^2)^{3/4}\right)$
        for the odd (top) and even primitive characters mod 8, where
        $F_\chi(T) := \int_0^T Z_{\chi}(t)\,\mathrm{d}t$. For the odd character
        $C_\chi = 0$, while for the even character, $C_{\chi_8(5, \cdot)} =
        \displaystyle{\tfrac{2^{3/2} \pi^{1/4}}{3}=1.255195\ldots}$.
    } 
    \label{mod8fun}
\end{figure}

\begin{figure}
    \includegraphics[width=\textwidth]{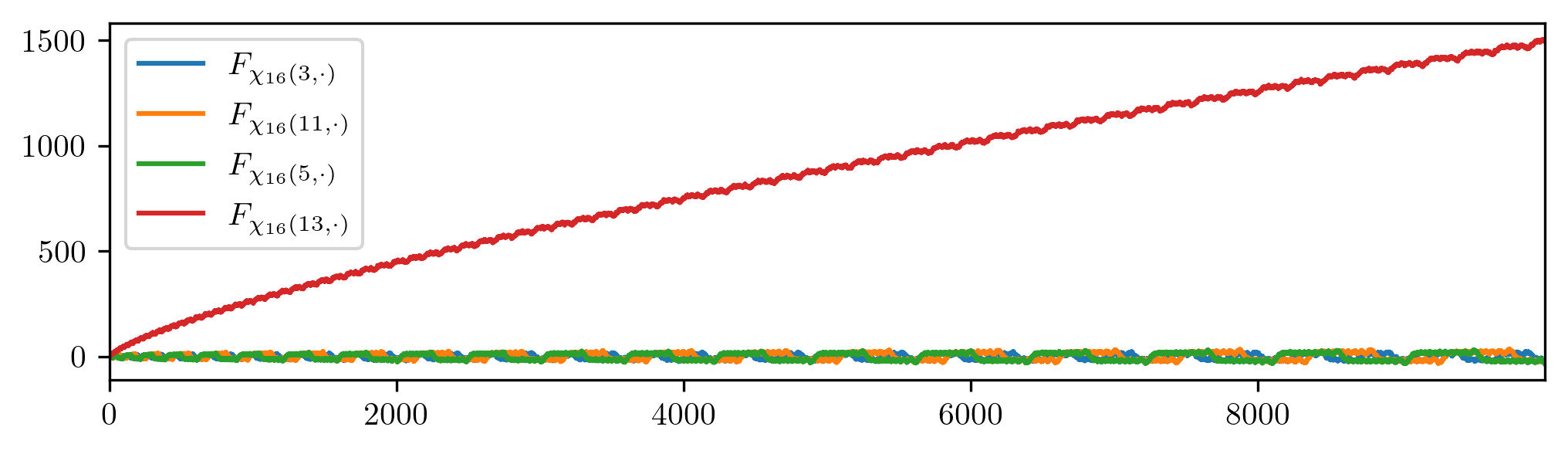}
    \includegraphics[width=\textwidth]{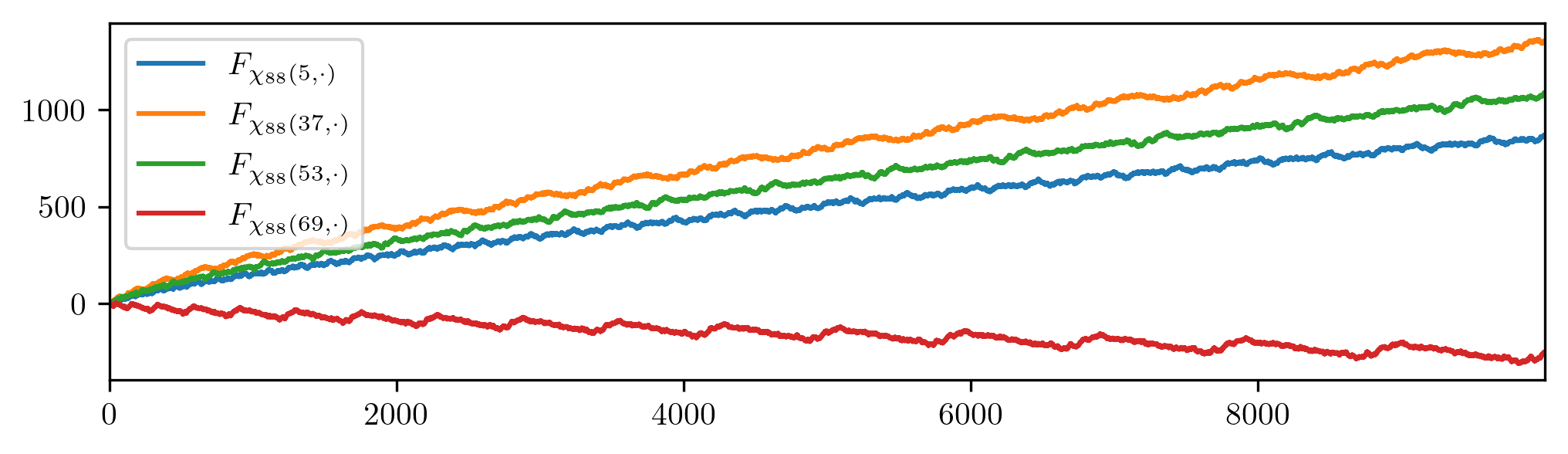}
    \caption{$F_\chi(T) := \int_0^T Z_{\chi}(t)\,\mathrm{d}t$ 
        for the four primitive characters (mod $16$) and the four primitive
        characters (mod $88$) for which $C_\chi \ne 0$.
        Here $\chi_{16}(3, \cdot)$ and $\chi_{16}(11, \cdot)$ are the conjugate
        pair of odd characters ($\chi_{16}(3,5) = -i$);
        characters $\chi_{16}(5, \cdot)$ and $\chi_{16}(13, \cdot)$ are the even characters ($\chi_{16}(5, 5) = i$).
        The characters (mod $88$) can be distinguished by
        $\chi_{88}(5,13) = e(9/10),
        \chi_{88}(37,13) = e(7/10),
        \chi_{88}(53,13) = e(1/10), \text{ and }
    \chi_{88}(69,13) = e(3/10)$.}
    \label{mod16plots}
\end{figure}

\smallskip

Since it is conjectured that $L(\tfrac{1}{2},\chi)\ne 0$ for all primitive characters $\chi$, we may assume that $Z_\chi(0)>0$. We further assume the generalized Riemann hypothesis  (GRH) for $L(s,\chi)$ and that the zeros of $L(s,\chi)$ are all simple. Then the zeros of $Z_\chi(t)$ correspond precisely to the imaginary parts of the zeros of $L(s,\chi)$ on the critical line and $Z_\chi(t)$ changes sign at every zero. 

\smallskip

With these assumptions, we let $0<\gamma_1<\gamma_2<\gamma_3 < \cdots$ denote the sequence of zeros of $Z_\chi(t)$ listed in increasing order, and let $g_n = \gamma_{n+1}-\gamma_n$ denote the size of the $n^{\text{th}}$ gap in this sequence. Then $Z_\chi(t)<0$ when $t \in (\gamma_{2n-1},\gamma_{2n})$ and $Z_\chi(t)>0$ when $t \in (\gamma_{2n},\gamma_{2n+1})$. So the size of the gap $g_n$ between consecutive zeros corresponds to the measure of the set where $Z_\chi(t)<0$ when $n$ is odd and to the measure of the set where $Z_\chi(t)>0$ when $n$ is even. If $C_\chi>0$ (resp.~$C_\chi<0$), then $Z_\chi(t)$ has a positive (resp.~negative) bias and we want to know to what extent that is reflected in the sizes of the gaps $g_n$ where $n$ is even (resp.~odd).

\smallskip

Here is a striking way to see a bias towards even or odd gaps in the zeros of
certain $L$-functions. Define the normalized gaps
$\tilde g_n = \gamma_{n+1} \log(\gamma_{n+1}) - \gamma_{n} \log(\gamma_{n})$,
so that $\tilde g_n$ is (asymptotically) constant on average. Take some large
enough parameter $X$ and consider the two averages
\begin{equation}
    \label{eq:gap-averages}
    A_0(n) = \frac{1}{X}\sum_{0\le k < X} \tilde g_{2n + 2k}\ \ \ \textrm{ and }\ \ \ 
    A_1(n) = \frac{1}{X}\sum_{0\le k < X} \tilde g_{2n + 2k + 1}
\end{equation}
as functions of $n$. Furthermore, to define something that averages to $1$,
define
\[
    \tilde A_0(n) = \frac{2 A_0(n)}{A_0(n) + A_1(n)}\ \ \ \textrm{ and }\ \ \
    \tilde A_1(n) = \frac{2 A_1(n)}{A_0(n) + A_1(n)}.
\]
A priori, we should expect these averages to be very close to the mean value
of the gaps over the whole range, so significant deviation from the mean
is perhaps already a curious feature. In Figure \ref{RunningMeanMod8}, we
see (for $X = 25000$) that the first of these averages is significantly
larger than the second for $n$ up to 1,500,000 (that is, over the first 3
million zeros) for the even primitive character (mod 8); for the odd character,
there is no such bias.

%In Figure \ref{RunningMeanMod8}, we plot the average even and odd gaps ({\color{red} Bober can you fix the wording? We want to include images of the ``running mean" plots}) for the zeros of the $L$-functions associated to the two primitive characters (mod 8) discussed in Example \ref{ex1}. For the even character $\psi_1$, which has $S(\psi_1)\ne 0$, there is a dramatic bias in the average size of the gaps $g_n$ depending on whether $n$ is even or odd. For the odd character $\psi_2$ which has $S(\psi_1)= 0$, there is no apparent bias the average size of the gaps $g_n$.

\smallskip

In Figure \ref{RunningMeanMod16}, we plot the same averages for the zeros
of the $L$-functions associated to the even primitive characters (mod 16). Let
$\chi_1$ (mod 16) be the character discussed in Example \ref{ex2}, and let $\overline{\chi_1}$ be its conjugate. In this case, $S(\chi_1)\ne 0$ and
$S(\overline{\chi_1})= 0$, and we observe that there appears to be a dramatic
bias in the average size of the gaps $g_n$ depending on whether $n$ is even or
odd for the zeros of $L(s,\chi_1)$ but not for the zeros of $L(s,\overline{\chi_1})$.
%{\color{red} M: We want to show four pictures of the ``running mean" of even
%and odd gaps for the characters (mod 8) and the character (mod 16) and its
%conjugate discussed above. I included the (mod 16) images we currently have.
%Did you want re-do these normalizing over each interval instead of over the
%whole range? Personally, I think these ``running mean" plots are by far the
%most interesting picture that we have}.

%\begin{figure}
%        \centering
%        \includegraphics[scale=0.25]{running_mean_16_13_0} 
%         \includegraphics[scale=0.25]{running_mean_16_13_1} 
%        \caption{{\color{red}This is a placeholder. Can we plot the running means for the two characters (mod 8)?}}
%        \label{RunningMeanMod8}
%\end{figure}

\begin{figure}
        \centering
        \includegraphics[width=\textwidth]{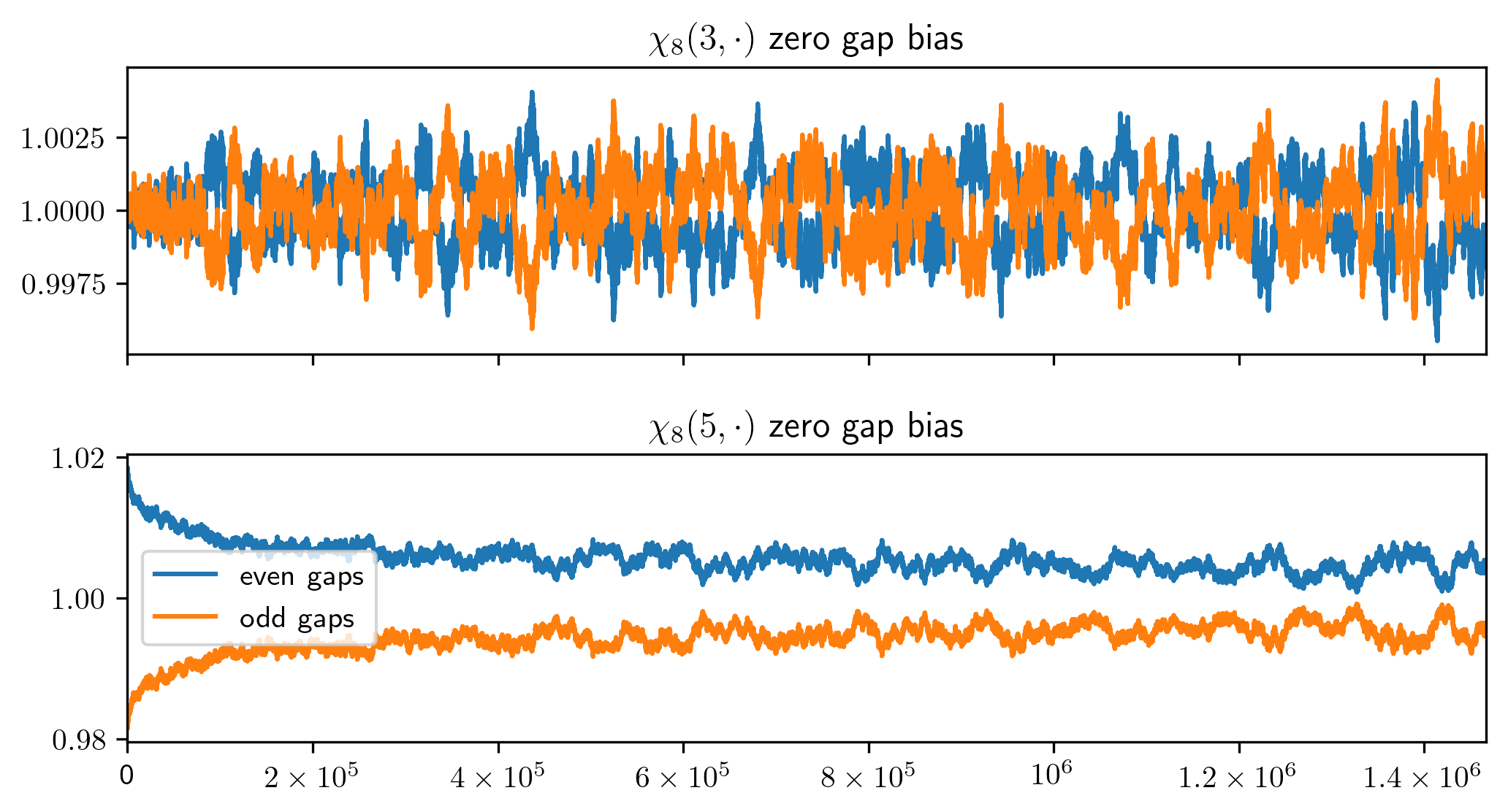} 
        \caption{Running averages of $25000$ consecutive normalized zero gaps for
the odd (top) and even (bottom) primitive characters (mod 8). See the
        discussion around Equation \eqref{eq:gap-averages}.}
        \label{RunningMeanMod8}
\end{figure}

\begin{figure}
        \centering
        \includegraphics[width=\textwidth]{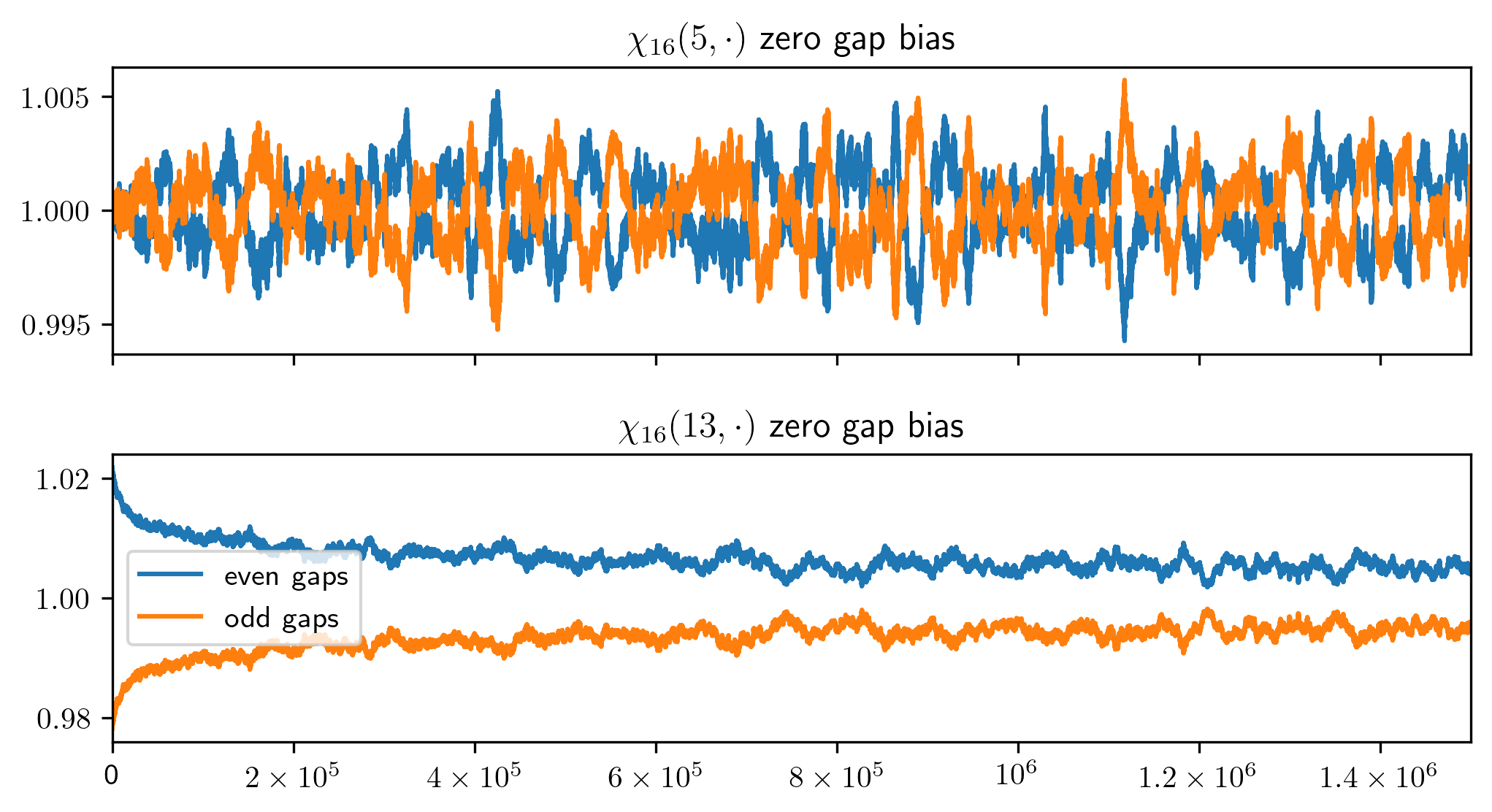}
        \caption{
        Running averages of $25000$ consecutive normalized zero gaps for
        the complex conjugate pair of primitive even characters (mod 16) of
        Example \ref{ex2}. Top: $\chi_{16}(5,3) = -i$. Bottom: $\chi_{16}(13,3) = i$.
        See the discussion around Equation \eqref{eq:gap-averages} for a
        precise description.
        }
        \label{RunningMeanMod16}
\end{figure}

\smallskip

%{\color{red} I might like to come back to a brief discussion the the choice of parameters for these figures in Section XXX yet to be written. e.g if we take $X$ as an increasing function of $n$, and then take a limit, do we actually get two different answers?}

The choice of parameters for Figures \ref{RunningMeanMod8} and
\ref{RunningMeanMod16} is somewhat arbitrary; however, studying bias in gaps is
the same as studying a bias in measure in the following sense. Following Gonek
and Ivi\'c \cite{gi}, who studied the distribution of positive and negative
values of the Hardy $Z$-function $Z(t)$, we define
\begin{equation}\label{eq:Mplusminus}
M^{\pm}_\chi(T) = \mathrm{meas}\Big\{ t\in [0,T]: \pm Z_\chi(t) >0 \Big\}.
\end{equation}
Then, we have
\[
\begin{split}
M_\chi^+(T) - M_\chi^-(T) &= \sum_{\substack{ n \ge 1 \\ 0<\gamma_{2n+1} \le T}} \big( g_{2n}-g_{2n-1}  \big) + O(1).
\end{split}
\] 
In Figure \ref{MeasureMod8}, we plot $M_{\chi_8}^+(T) - M_{\chi_8}^-(T)$ and $M_{\chi_{-8}}^+(T) - M_{\chi_{-8}}^-(T)$ for the two primitive characters (mod 8) discussed in Example \ref{ex1}. For $\chi_8$, which has $S(\chi_8)\ne 0$, there is a clear bias in measure for positive values of $Z_{\chi_8}(t)$ which is not present for the character $\chi_{-8}$, where $S(\chi_{-8})=0$.
In Figure \ref{MeasureMod16}, we plot $M_{\chi_1}^+(T) - M_{\chi_1}^-(T)$ and $M_{\overline{\chi_1}}^+(T) - M_{\overline{\chi_1}}^-(T)$ for the even character $\chi_1$ (mod 16) discussed in Example \ref{ex2}. For $\chi_1$, which has $S(\chi_1)\ne 0$, there is a clear bias in measure for positive values of $Z_{\chi_1}(t)$ which is not present for its conjugate character $\overline{\chi_1}$, where $S(\overline{\chi_1})=0$.

\smallskip

%In Section \ref{section:distribution}, we show more examples illustrating the connection between the nonvanishing of sum $S(\chi)$, the bias in gaps between zeros of $L(s,\chi)$, and the bias in the measure of the intervals for which $Z_\chi(t)$ is positive or negative indicated by plotting $M_\chi^+(T) - M_\chi^-(T)$. 

%\smallskip

%Letting $M^{\pm}(T) = \mathrm{meas}\big\{ t\in [0,T]: \pm Z(t) >0 \big\}$, Gonek and Ivi\'c \cite{gi} have conjectured that $M^{\pm}(T) \sim \frac{T}{2}$ as $T\to \infty$ and they proved some interesting partial results toward this conjecture. Likewise, we conjecture that $M^{\pm}_\chi(T) \sim \frac{T}{2}$ as $T\to \infty$ for every primitive character $\chi$, but when $S(\chi) \ne 0$, based on our numerical data for the distribution of $M_\chi^+(T) - M_\chi^-(T)$, we expect that there is a lower-order term in this asymptotic formula (with the same sign as $C_\chi$). This is discussed more in Section \ref{section:distribution}.

Letting $M^{\pm}(T) = \mathrm{meas}\big\{ t\in [0,T]: \pm Z(t) >0 \big\}$, Gonek and Ivi\'c \cite{gi} have conjectured that $M^{\pm}(T) \sim \frac{T}{2}$ as $T\to \infty$ and they proved some interesting partial results toward this conjecture. Likewise, we conjecture that $M^{\pm}_\chi(T) \sim \frac{T}{2}$ as $T\to \infty$ for every primitive character $\chi$. However, when $S(\chi) \ne 0$, our numerical data suggest that $M_\chi^+(T) - M_\chi^-(T)$ tends to $+\infty$ if $C_\chi>0$ and to $-\infty$ if $C_\chi<0$. This is discussed more in Section \ref{section:distribution}.

%When $S(\chi) \ne 0$, as $T\to \infty$, perhaps an asymptotic formula such as
%\[
%M^{\pm}_\chi(T) = \frac{T}{2} + \mathrm{sgn}( C_\chi ) \, T^{\theta_\chi} \, P_\chi(\log T) + o(T^{\theta_\chi} )
%\]
%holds, where $0< \theta_\chi <1$ and $P_\chi(\cdot)$ is a polynomial with positive leading coefficient.  

\smallskip

%(old) Experimentally, we have found another thin family of Dirichlet characters $\chi$ for which $S(\chi)=0$ but there still seems to be a bias in measure in the sense that $M_\chi^+(T) - M_\chi^-(T)$ appears to have an asymptotic growth rate as of at least  $T^{\theta_\chi}$ for some $0< \theta_\chi <1$. We have an alternate theoretical explanation for this family in terms of the different Gauss-type sum that arises in the main-term in asymptotic formula for the discrete mean-values
%\[
%\sum_{0<\gamma_\chi\le T} Z_\chi'(\gamma_\chi) \quad \text{and} \quad \sum_{0<\lambda_\chi\le T} Z_\chi(\lambda_\chi),
%\]
%where $\lambda_\chi$ are the roots of $Z_\chi'(t)$, which we estimate assuming GRH for $L(s,\chi)$. We will present these results in a companion paper. %{\color{red} M: give an example, or leave this mysterious?}

Experimentally, we have found another thin family of Dirichlet characters $\chi$ with \emph{odd} modulus for which $S(\chi)=0$ but there still seems to be a bias in gaps and a bias in measure as described above.
It seems these biases can be explained via discrete moments. For instance, assuming GRH for Dirichlet $L$-functions, we have proved asymptotic formulae for the discrete moments
\[
\sum_{0<\gamma_\chi\le T} Z_\chi'(\gamma_\chi)
\qquad\text{and}\qquad
\sum_{0<\lambda_\chi\le T} Z_\chi(\lambda_\chi),
\]
where $\gamma_\chi$ and $\lambda_\chi$ run over the zeros of
$Z_\chi(t)$ and $Z_\chi'(t)$, respectively.
A new Gauss-type sum arises in the main terms of these asymptotic formulae; it vanishes for primitive characters of even modulus and can be nonzero for even characters of odd modulus.

\smallskip

Perhaps surprisingly, the Riemann zeta-function is in this family. For the usual Hardy $Z$-function, $Z(t)$, assuming RH, for every $\varepsilon>0$, our results give
\begin{align*}
\sum_{0<\gamma\le T} Z'(\gamma)
&=-\frac{2^{3/2}\log 2\,\sin(\pi/8)}{3}
  \left(\frac{T}{2\pi}\right)^{3/4}
  \left(\log\frac{T}{2\pi}
        +2\gamma_0-5\log 2-\frac43\right) +O_\varepsilon\!\left(T^{1/2+\varepsilon}\right),
\end{align*}
where $\gamma_0$ is Euler's constant, and
\[
\sum_{0<\lambda\le T} Z(\lambda)
=
-\frac{2^{5/2}(e-2)\log 2\,\cos(\pi/8)}{3}
 \left(\frac{T}{2\pi}\right)^{3/4}
+O\!\left(\frac{T^{3/4}}{\log T}\right),
\]
where $\gamma$ and $\lambda$ run over the zeros of
$Z(t)$ and $Z'(t)$, respectively.
Both discrete averages have negative leading
coefficients and this is reflected in the experimental bias in gaps and bias in measure for the Hardy $Z$-function (see Figures \ref{Z_gaps} and \ref{Z_meas}). For the second sum, we obtain an asymptotic expansion to arbitrarily many orders in inverse powers of $\log T$. We will present results for all primitive characters, together with their proofs, in a companion paper.

\smallskip

Random matrix models have played an important role in the study of $L$-functions for many years, so a natural question is how to construct a random model for these $Z$-functions that simultaneously accounts for their continuous and discrete moments. Such a model should reflect the arithmetic dependence of the main terms, including the different Gauss-type sums that arise in these asymptotic formulae. We plan to investigate this  in our forthcoming paper on discrete moments.

%{\color{red} TO CHANGE: Experimentally, we have found another thin family of Dirichlet characters $\chi$ with {\it odd} modulus (which includes $Z(t)$) for which $S(\chi)=0$ but there still seems to be a bias in measure for $Z_\chi(t)$ in the sense that $M_\chi^+(T) - M_\chi^-(T)$ appears to have an asymptotic growth rate of at least  $T^{\theta_\chi}$ for some $0< \theta_\chi <1$.  We have an alternate theoretical explanation for the bias in measure of this family in terms of the nonvanishing of different Gauss-type sum that arises in the main-term of the asymptotic formula for the discrete mean-values
%\[
%\sum_{0<\gamma_\chi\le T} Z_\chi'(\gamma_\chi) \quad \text{and} \quad \sum_{0<\lambda_\chi\le T} Z_\chi(\lambda_\chi),
%\]
%where $\lambda_\chi$ are the roots of $Z_\chi'(t)$, which we have estimated assuming GRH for $L(s,\chi)$. We will present these results in a companion paper.}

\begin{figure}
        \centering
        \includegraphics[width=\textwidth]{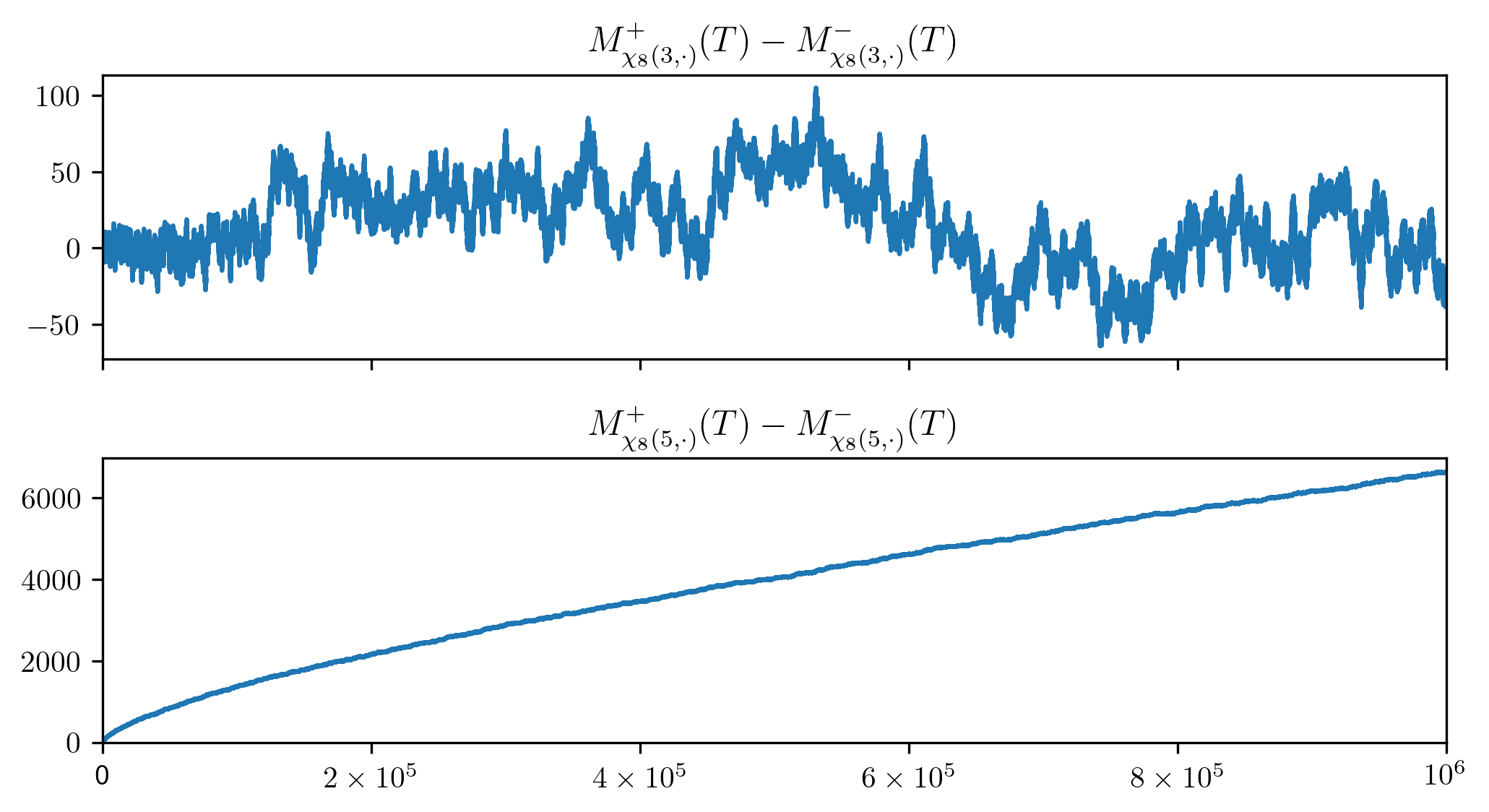}
        %\caption{{\color{red} This is a placeholder.} We want one plot containing both $M_{\psi_1}^+(T) - M_{\psi_1}^-(T)$ and $M_{\psi_2}^+(T) - M_{\psi_2}^-(T)$ for the two primitive characters (mod 8) discussed in Example \ref{ex1}.}
        \caption{The difference between the amount of time that $Z_\chi(t)$ spends
            positive and negative for the odd (top) and even (bottom) primitive
        characters (mod 8).}
        \label{MeasureMod8}
\end{figure}

\begin{figure}
        \centering
        \includegraphics[width=\textwidth]{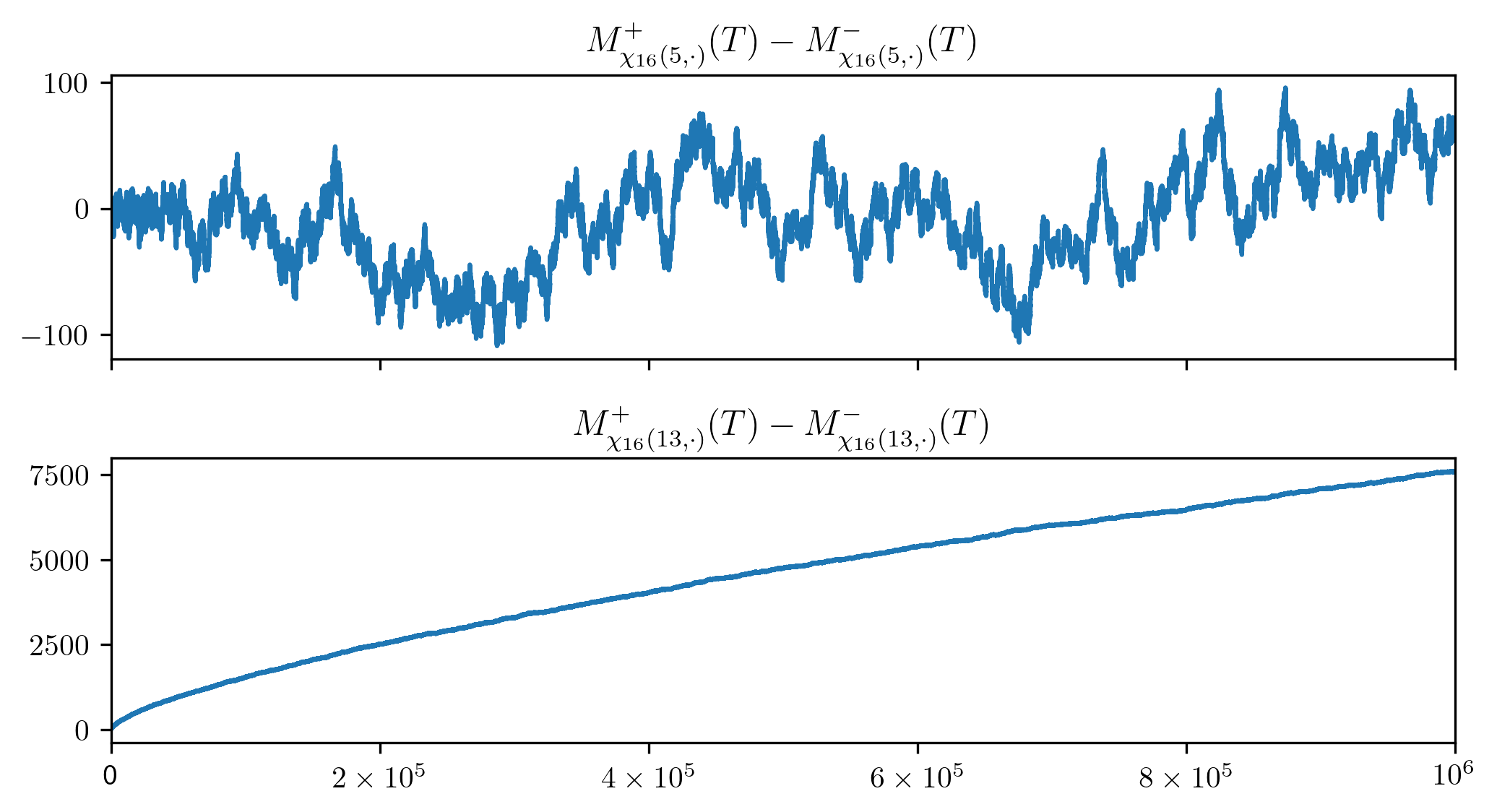}
        %\caption{This is a plot of $M_{\chi_1}^+(T) - M_{\chi_1}^-(T)$ (in orange) and $M_{\overline{\chi_1}}^+(T) - M_{\overline{\chi_1}}^-(T)$ (in blue) for the even character $\chi_1$ (mod 16) discussed in Example \ref{ex2}.}
        \caption{The difference between the amount of time that $Z_\chi(t)$
            spends positive and negative, for the even character (mod 16) and its
            complex conjugate, discussed in Example \ref{ex2}. Top: $\chi_{16}(5,3)
        = -i$. Bottom: $\chi_{16}(13,3) = i$.}
        \label{MeasureMod16}
\end{figure}

\begin{figure}
        \centering
        \includegraphics[width=\textwidth]{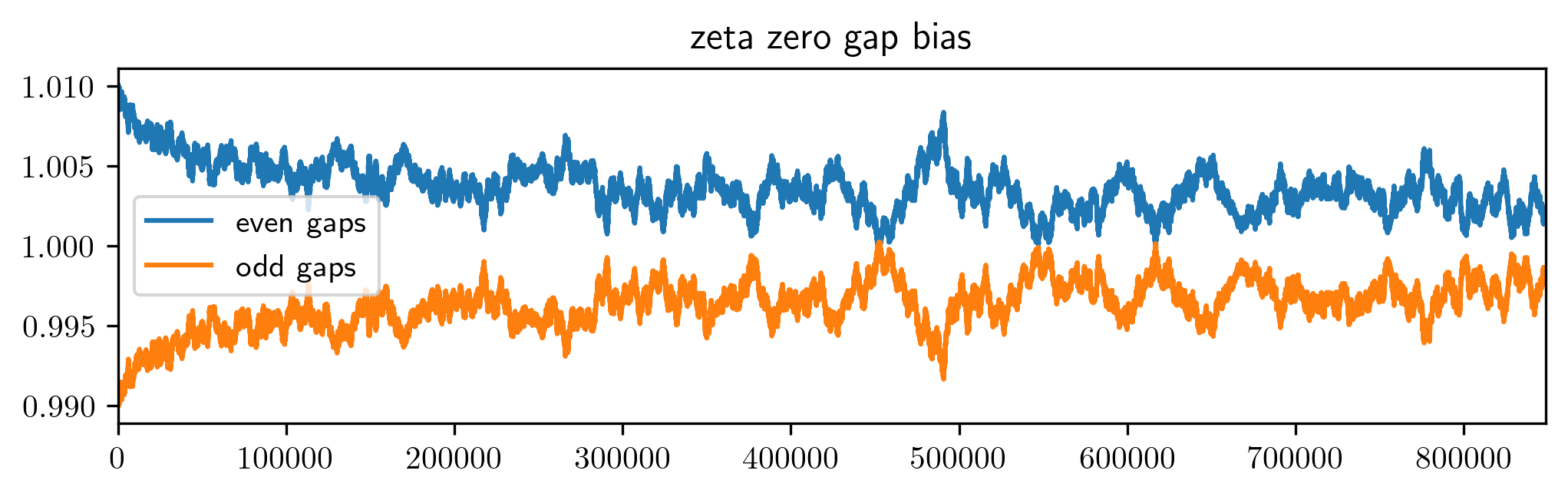} 
        \caption{Running averages of $25000$ consecutive normalized gaps of ordinates of zeros of $\zeta(s)$, the Riemann zeta-function. See the discussion around Equation \eqref{eq:gap-averages}. The blue curve  corresponds to the intervals where $Z(t)<0$.}
        \label{Z_gaps}
\end{figure}

\begin{figure}
    \centering
    \includegraphics[width=\textwidth]{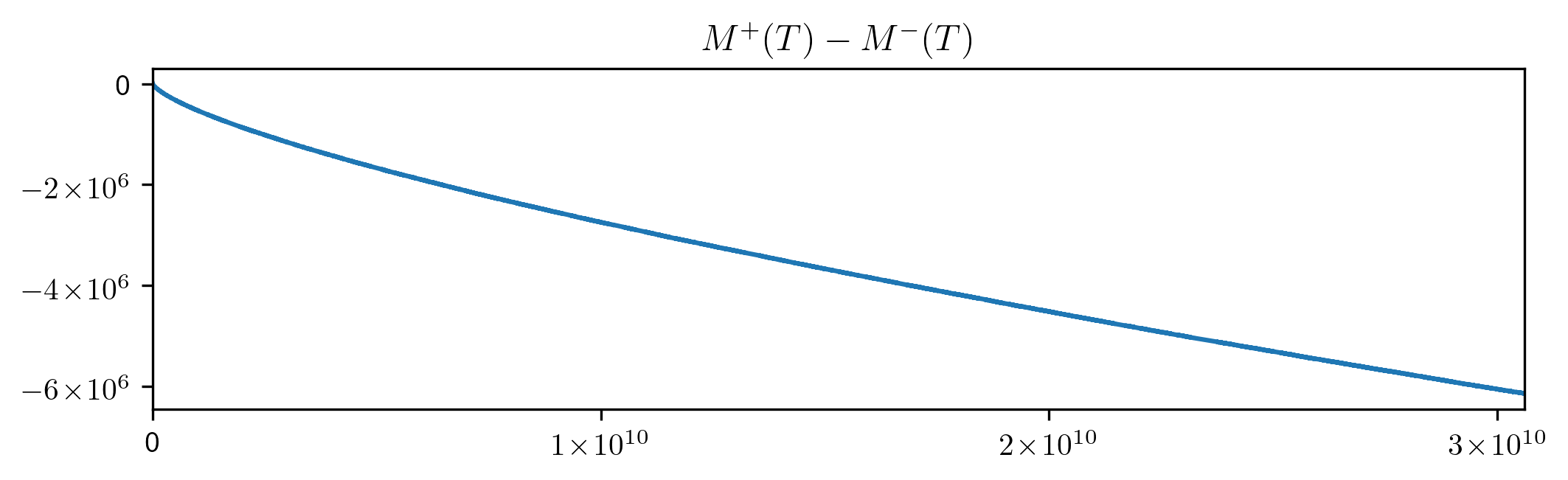}
    \caption{Plot of $M^+(T) \!-\! M^-(T)$,
        where $M^{\pm}(T)=\textrm{meas}\{ t \in [0,T] : \pm Z(t) > 0\}$
        and $Z(t)$ is the Hardy $Z$-function. This plot was made using approximately
    100 billion zeros of $\zeta(s)$, computed by D.~Platt \cite{platt}.}
    \label{Z_meas}
\end{figure}

\subsection{Analysis of \texorpdfstring{$S(\chi)$}{S(chi)}}
In Proposition \ref{prop:S1} of this subsection, we give a complete classification of the primitive $\chi$ (mod $q$) for which
\[
S(\chi)=\sum_{n=1}^{2q}\chi(n)\,e\!\left(\frac{-n^2}{2q}\right) \ne 0.
\]
Note the following simple observations. 
Replacing $n$ by $-n$ in the sum, we see that $S(\chi) = \chi(-1)\, S(\chi)$. So $S(\chi)=0$ when $\chi$ is odd. Also,
\[
\begin{split}
S(\chi) &= \left\{ \sum_{n=1}^{q}+\sum_{n=q+1}^{2q} \right\} \chi(n)\,e\!\left(\frac{-n^2}{2q}\right) = \left\{ 1+ e\!\left(-\frac{q}{2}\right)  \right\} \sum_{n=1}^q \chi(n)\,e\!\left(\frac{-n^2}{2q}\right).
\end{split}
\]
So $S(\chi)=0$ when $q$ is odd. These two observations show that $S(\chi)=0$ for at least 75\% of primitive characters. 
In general, we let
\begin{equation}\label{eq:counting_function}
N(q) := \# \big\{ \chi \, (\text{mod }q): \chi \text{ primitive and } S(\chi) \neq 0 \big\} 
\end{equation}
denote the number of primitive characters mod $q$ for which $S(\chi)$ is nonzero. A consequence of Proposition \ref{prop:S1} is the following exact formula for $N(q)$ in terms of the prime factorization of $q$.

\begin{theorem}\label{thm:counting_q}
Let $q \in \mathbb{N}$. If $8$ does not divide $q$, then $N(q)=0$. If $8|q$, write
\[\displaystyle{q=2^{m} \Big(\prod_i p_i \Big) \Big(\prod_j r_j^{a_j}\Big)},
\]where $p_i, r_j$ are distinct odd primes, $m\ge 3$, and $a_j\geq 2$. Then, we have
\begin{equation}\label{eq:counting_q}
N(q) = \big\lceil 2^{m-5} \big\rceil \, \prod_i \left( \frac{p_i-3}{2}\right) \prod_j \left(\frac{1}{4}r_j^{a_j-2}\left( r_j-1\right)^2\right).
\end{equation}
\end{theorem}

Using the Selberg--Delange method, we can compute the average of $N(q)$, as $q$ varies. 

\begin{corollary}\label{sd} As $Q\to \infty$, we have
\begin{equation}\label{eq:counting_Q}
\sum_{q\leq Q}N(q) = \frac{\mathscr{A} Q^2}{\sqrt{\log Q}} + O\!\left( \frac{Q^{2}}{(\log Q)^{3/2}}\right),
\end{equation}
where $\mathscr{A}>0$ is a constant given explicitly in \eqref{Avalue}. 
\end{corollary}

Since there are $\asymp Q^2$ primitive characters with modulus not exceeding $Q$, Corollary \ref{sd} implies that $S(\chi)=0$ for $100\%$ of primitive characters. Thus, the constant in Theorem \ref{thm:GL1} satisfies $C_\chi=0$ for $100\%$ of primitive characters and therefore we expect the analogue of Ivi\'c's conjecture in \eqref{ivic} to hold for $100\%$ of primitive characters.

\smallskip

We now state our classification of the primitive $\chi$ for which $S(\chi) \ne 0$. As seen in Corollary \ref{sd}, these characters form a set of density zero among all primitive characters. 

\begin{proposition}[Main Proposition]\label{prop:S1}
    Let $\chi$ be a primitive character modulo $q$, where
    $q = \prod_{p} p^{e_p}$ is the factorization of $q$ into prime powers and
    $\chi = \prod_p \chi_p$ is the factorization of $\chi$ into characters
    of prime-power modulus. Set $\alpha_2$ to be an integer such
    that $\chi_2(5) =
    e(\alpha_2/2^{e_2 - 2})
    $, and for each other prime $p$
    set $\alpha_p$ to be an integer such that $\chi_p(p + 1) = e(\alpha_p/p^{e_p - 1})$.
    Then $S(\chi) \ne 0$ if and only if all of the following hold:
    \begin{enumerate}
        \item[(i)] $\chi_p$ is even for each $p$;
        \item[(ii)] For each $p$ with $e_p > 1$, the residue $\alpha_p q/p^{e_p}$ is a square modulo $p$;
        \item[(iii)] $e_2 > 2$, and $\alpha_2 \equiv 3q/ 2^{e_2} \pmod{2^{e_2-2}}$ if $3 \le e_2 \le 5$ or
            $\alpha_2 \equiv -q/2^{e_2} \mod 8$ if $e_2 \ge 6$.
    \end{enumerate}

    In this case, let $\psi = \prod_{p | q} \psi_p$ be a character \textup{(mod }$2q$\textup{)} such
    that $\psi^2 = \chi$. Then $S(\chi) = \Seven \, \Sodd$,
    where
    \begin{equation} \label{Splus}
        \Seven = \left\{\begin{array}{ll} 
        2^{e_2}  \,  \psi_2(q/2^{e_2})^2 \, e(-q/2^{2e_2 + 1}), & \mathrm{if}\ e_2 <6, \\
        4 \, \psi_2(-q/2^{e_2}) \, \tau(\psi_2), & \mathrm{if}\ e_2 \ge 6,
        \end{array}
        \right.
    \end{equation}
    and
    \begin{equation} \label{Sminus}
    \begin{split}
        \Sodd &= \prod_{\substack{p | q \\ p\ \mathrm{odd}} }
            \psi_p(-2q/p^{e_p})\big(\tau(\psi_p) + \epsilon_p(-2q/p^{e_p}) \, \tau(\epsilon_p \psi_p)\big) \\
             &= \prod_{\substack{p || q \\ p\ \mathrm{odd}} }
            \psi_p(-2q/p^{e_p})\big(\tau(\psi_p) + \epsilon_p(-2q/p^{e_p}) \, \tau(\epsilon_p \psi_p)\big)
                \prod_{\substack{p^2 | q \\ p\ \mathrm{odd}} } 2 \, \psi_p(-2q/p^{e_p}) \, \tau(\psi_p),
    \end{split}
    \end{equation}
    where $\epsilon_p$ is the quadratic character \textup{(mod }$p$\textup{)}.
\end{proposition}

{\sc Remarks}.

\smallskip

($i$) Note, for example, that there is no even primitive character (mod $3$), so $N(q) = 0$ if $q$ is divisible by $3$ but not $9$.

\smallskip

($ii$) From the above evaluation of $S(\chi)$, we see that when $S(\chi)\ne 0$, the value $|S(\chi)|$ is easy to compute whenever $q$ is square-full. For example, $|S(\chi)| = 2^{\omega(q)+3/2} \sqrt{q}$ if $q$ is square-full and $e_2 \ge 6$, where as usual $\omega(q)$ is the number of distinct prime divisors of $q$. On the other hand, for a character $\psi_p$ (mod $p$), the quantity
\[
 \tau(\psi_p) \pm \tau(\epsilon_p \psi_p)
\]
may be much smaller than $2\sqrt{p}$. We discuss this in more detail in Section
\ref{section:distribution}.
%In fact, by writing this in terms of Jacobi sums, we will see that, once this sum is squared, suitably normalized, and rotated, it should be distributed like the real part of a uniform random number on the unit circle. (See Remark with \eqref{eq:remark}.)
%\end{remark}

\smallskip

We conclude this subsection by describing a simple sub-family of characters $\chi$ for which $S(\chi) \ne 0$ which follows directly from Proposition \ref{prop:S1}. 

\begin{example}
Let $q= 8d$, where $d$ is an odd square-free number. Then the primitive characters $\chi$ \textup{(mod }$q$\textup{)} for which $S(\chi) \ne 0$ are precisely those characters of the form $\chi = \psi^2$, where
    $\psi$ is a primitive character \textup{(mod }$16d$\textup{)}. 
    
    \smallskip

To see why, note that part \textup{(i)} of Proposition \ref{prop:S1} implies that $\chi$ is a product of the even primitive character \textup{(mod }$8$\textup{)} times the product of even primitive characters \textup{(mod }$p$\textup{)} for each prime $p$ dividing $d$. When $p$ is odd, even primitive characters  \textup{(mod }$p$\textup{)} are precisely the squares of non-quadratic primitive characters  \textup{(mod }$p$\textup{)}. The even primitive character \textup{(mod }$8$\textup{)} is the square of a primitive character \textup{(mod }$16$\textup{)}. 
Therefore, $S(\chi) \ne 0$ if and only if $\chi$ is the square of a primitive character \textup{(mod }$16d$\textup{)}. 
\end{example}

\subsection{The \texorpdfstring{$Z$}{Z}-function of degree two \texorpdfstring{$L$}{L}-functions} It is known that $L$-functions associated to cuspidal automorphic representations of GL$(2)$ over the rationals correspond to primitive holomorphic cusp forms and primitive Hecke--Maass cusp forms. In contrast to the thin family of Dirichlet $L$-functions associated to the characters described in Proposition \ref{prop:S1}, we show that the integral of the $Z$-functions of these degree two $L$-functions {\it never} has a main term of size $T^{3/4}$. 

\smallskip

Let $f \in S_k(\Gamma_0(N),\chi)$ be a primitive holomorphic cusp form of weight $k \ge 1$, level $N$, and nebentypus $\chi$, and let
\[
f(z) = \sum_{n\ge 1} \lambda_f(n) \, n^{(k-1)/2} \, e (nz)
\]
be its normalized Fourier expansion at the cusp $\infty$. Then the $L$-function $L(s,f) = \sum_{n\ge 1} \lambda_f(n) n^{-s}$ extends to an entire function and satisfies the functional equation, in asymmetric form,
\[
L(s,f) = 
\mathcal{G}_f(s) \, L(1\!-\!s,\overline{f}),
\quad \textup{where}
\quad
\G_f(s) 
= \varepsilon_f \,N^{1/2-s} 
\frac{\Gamma_{\mathbb{R}}(1\!-\!s\!+\!\frac{k-1}{2})}{\Gamma_{\mathbb{R}}(s\!+\!\frac{k-1}{2})} \frac{\Gamma_{\mathbb{R}}(1\!-\!s\!+\!\frac{k+1}{2})}{\Gamma_{\mathbb{R}}(s\!+\!\frac{k+1}{2})},
%=\frac{\tau(\chi)}{q^{1/2}}q^{1/2-s}\frac{\pi^{(1-s)/2}\Gamma((1-s)/2)}{\pi^{s/2}\Gamma(s/2)}.
\]
where $\varepsilon_f$ is the root number of $L(s,f)$, and $\overline{f} \in
S_k(\Gamma_0(N),\overline\chi)$ is the dual form of $f$ with Fourier coefficients
satisfying $\lambda_{\overline{f}}(n) =\overline{\lambda_f(n)}$. Then we define
\[
Z_f(t) = \G_f(\tfrac{1}{2}\!+\!it)^{-1/2} \, L(\tfrac{1}{2}\!+\!it,f).
\]
As before, $Z_f(t)$ is real-valued for real $t$
and $|Z_f(t)| = |L(\tfrac{1}{2}+it,f)|$.

\begin{theorem}\label{thm:holo}
Let $f \in S_k(\Gamma_0(N),\chi)$ be a primitive holomorphic cusp form. Then
\[
\int_0^T Z_f(t) \ \mathrm{d}t \ll_f \sqrt{T} \log T,
\]
as $T\to \infty$.
%where the implied constant depends on $f$.
\end{theorem}

Let $\varphi$ be a primitive Hecke--Maass cusp form of level $N$ with nebentypus $\chi$ which is an eigenfunction of the Laplace operator with eigenvalue $\lambda = \frac{1}{4}+r^2$, where $r\in \mathbb{R}$ or $ir \in [0,\frac{1}{2})$. For $z=x+iy$, we write the Fourier expansion of $\varphi(z)$ at $\infty$ as
\[
\varphi(z) = \sqrt{y} \sum_{n\in \mathbb{Z}\setminus\{0\}} \lambda_{\varphi}(n) K_{ir}\big( 2\pi|n| y\big) \, e(nx).
\]
Then the $L$-function $L(s,\varphi) = \sum_{n\ge 1} \lambda_{\varphi}(n) \, n^{-s}$ extends to an entire function and satisfies the functional equation, in asymmetric form,
\[
L(s,\varphi) = 
\mathcal{G}_\varphi(s) \, L(1\!-\!s,\overline{\varphi}),
\quad \textup{where}
\quad
\G_\varphi(s) 
= \varepsilon_\varphi \,N^{1/2-s} 
\frac
{\Gamma_{\mathbb{R}}(1\!-\!s\!+\!\delta\!+\!ir)}{\Gamma_{\mathbb{R}}(s\!+\!\delta\!+\!ir)}
 \frac
{\Gamma_{\mathbb{R}}(1\!-\!s\!+\!\delta\!-\!ir)}{\Gamma_{\mathbb{R}}(s\!+\!\delta-ir)},
%=\frac{\tau(\chi)}{q^{1/2}}q^{1/2-s}\frac{\pi^{(1-s)/2}\Gamma((1-s)/2)}{\pi^{s/2}\Gamma(s/2)}.
\]
where $\varepsilon_\varphi$ is the root number of $\varphi$, $\delta=0$ if $\varphi$ is even, $\delta=1$ if $\varphi$ is odd, and $L(s,\overline{\varphi})$ is the $L$-function with coefficients $\lambda_{\overline{\varphi}}(n)=\overline{\lambda_{\varphi}(n)}$. Then we define
\[
Z_\varphi(t) = \G_\varphi(\tfrac{1}{2}\!+\!it)^{-1/2} \, L(\tfrac{1}{2}\!+\!it,\varphi).
\]

%$\varepsilon_f$ is the root number of $L(s,f)$, and $\overline{f} \in S_k(\Gamma_0(N),\chi)$ is the dual form of $f$ with Fourier coefficients satisfying $\lambda_{\overline{f}}(n) =\overline{\lambda_f(n)}$. 
%
%When $f$ is a Maass cusp form of level $N_f$ with eigenvalue $\frac{1}{4}+r^2$, the $Z$-function associated to $f$ is $Z_f(t) = \G_f(\tfrac{1}{2}+it)^{-1/2}L_f(\tfrac{1}{2}+it)$ with 
%\[
%\G_f(s) = \epsilon_{f} N_f^{\frac{1}{2}-s} \frac
%{\Gamma_{\mathbb{R}}(1-s+\delta+ir)}{\Gamma_{\mathbb{R}}(s+\delta+ir)}
% \frac
%{\Gamma_{\mathbb{R}}(1-s+\delta-ir)}{\Gamma_{\mathbb{R}}(s+\delta-ir)},
%\]
%where $|\epsilon_f|=1$, $\delta=1$ if $f$ is odd and $\delta=0$ if $f$ is even.

\begin{theorem}\label{thm:Maass}
Let $\varphi$ be a primitive Hecke--Maass cusp form of level $N$ with nebentypus  $\chi$. Then, as $T\to \infty$, we have
\[
\int_0^T Z_\varphi(t) \ \mathrm{d}t \ll_{\varphi, \varepsilon}  T^{\frac{11}{20}+\varepsilon},
\]
for any $\varepsilon>0$. 
\end{theorem}

\smallskip

{\sc Remarks}.

\smallskip

($i$) There is some hope to improve the upper bound on the right-hand side of Theorem \ref{thm:holo} from $\ll_f\sqrt{T} \log T$ to $\ll_{\varepsilon,f} T^{1/3+\varepsilon}$ when the level $N$ is square.  We plan to investigate this possibility in a future work.  

\smallskip

($ii$) The difference in order of magnitude in upper bounds in Theorems
\ref{thm:holo} and \ref{thm:Maass} comes from the current knowledge of
estimates for the coefficients of the $L$-functions in short intervals. See
Equations \eqref{short1} and \eqref{short2}. In particular, the assumption of the
Ramanujan--Petersson conjecture for the coefficients of $L(s,\varphi)$ would
allow us to replace the exponent $\frac{11}{20}+\varepsilon$ by
$\frac{1}{2}+\varepsilon$ in Theorem \ref{thm:Maass}.

%In particular, our proofs rely on the bounds
%\[
%\sum_{x< n \le x+\sqrt{x}} |\lambda_f(n)| \ll \sqrt{x} \, (\log x)^{-\delta} \quad \text{and} \quad \sum_{x< n \le x+\sqrt{x}} |\rho(n)|  \ll_\varepsilon x^{11/20+\varepsilon}. 
%\]
%for some $\delta>0$. The bound for Hecke-Maass cusp forms comes from the size of the error term in Using the Kim-Sarnak \cite{} bound $|\rho(n)| \ll_\varepsilon n^{7/64+\varepsilon}$ in place of this short interval estimate for the coefficients $|\rho(n)|$ would lead to an estimate of the form
%\[
%\int_0^T Z_\varphi(t) \ \mathrm{d}t \ll_{\varphi, \varepsilon}  T^{\frac{1}{2}+\frac{7}{64}+\varepsilon},
%\]
%as $T\to\infty$. The right-hand side here is still $o(T^{3/4})$.

\smallskip

($iii$) It seems difficult with known techniques to prove results analogous to
Theorems \ref{thm:holo} and \ref{thm:Maass} for degree 3 $L$-functions. A
bound of $o(T)$, as $T\to \infty$, for the integral of the $Z$-function
associated to an $L$-function on $[0,T]$ would imply an analogue of Hardy's theorem
that there are infinitely many zeros of that $L$-function on the critical line.
This is a notoriously difficult problem for principal $L$-functions of degree
greater than 2.

\section{Proof of Proposition \ref{prop:S1}}\label{sec:ExpSum}

In this section, we study the structure of the sum $S(\chi)$ and prove Proposition \ref{prop:S1}.

%%%%%%%%%%%%%%% Evaluation of the Sum %%%%%%%%%%%%%%%%%%%

%%%%%%%%%%%%%%% Lemmas %%%%%%%%%%%%%%%%%%%

\subsection{Three lemmas for simplifying Gauss sums}

\begin{lemma} \label{lem:gauss-sum}
    Let $\chi$ be a primitive character \textup{(mod }$p^a$\textup{)} for some $a > 1$, and $\psi$ a
    character \textup{(mod }$p^b$\textup{)} for some $b \le a/2$. For any integer $\ell$ in the range
    $a/2 \le \ell \le  a - b$ let $\alpha$ be an integer such that
    \[
        \chi(p^\ell \!+\! 1) = e\!\left(\frac{\alpha}{p^{a - \ell}}\right).
    \]
    Then
    \[
        \tau(\psi \chi) = \psi(-\alpha) \, \tau(\chi).
    \]
\end{lemma}
\begin{proof}
    We first note that $(mp^\ell + 1)(np^\ell + 1) = mnp^{2\ell} + (m+n)p^\ell + 1$, and
    $2\ell \ge a$ by hypothesis, so $\chi(np^\ell + 1)$ is an additive character
    as a function of $n$, and more specifically
    \[
        \chi(np^\ell \!+\! 1) = e(n\alpha/(p^{a-\ell})).
    \]
    We now split the Gauss sum $\tau(\psi\chi)$ into the various arithmetic
    progressions (mod $p^\ell$), writing
\[
\begin{split}
  \tau(\psi\chi) =&  \sum_{n=0}^{p^{a-\ell}-1} \sum_{m=0}^{p^\ell-1}
          \psi(np^\ell \!+\! m) \,
            \chi(np^\ell \!+\! m) \,
            e\!\left(\frac{np^\ell + m}{p^a}\right).
\end{split}
\]
    As we restrict to the case where $\ell \ge b$, $\psi(np^\ell + m)$ depends only
    on $m$, so we move it outside the inner sum and evaluate $\chi(np^\ell + m)$
    as $\chi(m)$ times an additive character, writing this sum as
\[
\begin{split}
    \tau(\psi\chi) =& \sum_{\substack{m=1 \\ p \nmid m}}^{
    p^{\ell}-1} \psi(m) \, \chi(m) 
    \sum_{n=1}^{p^{a-\ell}}
            \chi(\overline m np^\ell \!+\! 1) \,
            e\!\left(\frac{np^\ell \!+\! m}{p^a}\right)\\
=& 
\sum_{\substack{m=1 \\ p \nmid m}}^{p^{\ell}-1} \psi(m) \, \chi(m) \, e\!\left(m/p^a\right)
                \sum_{n=1}^{p^{a-\ell}}
                    e\!\left(\frac{\overline m n \alpha \!+\! n}{p^{a-\ell}}\right),
\end{split}
\]
where $\overline m$ denotes the inverse of $m$ modulo $p^{a}$. We now see that
the innermost sum vanishes unless $m \equiv -\alpha \mod{p^{a-\ell}}$, in which
case it is exactly $p^{a-\ell}$, so we find that the sum is equal to
\[
    p^{a-\ell} \sum_{\substack{m=1 \\ p \nmid m \\ m \equiv -\alpha \mod{p^{a-\ell}}}}^{p^{\ell}-1} \psi(m) \, \chi(m) \, e\!\left(m/p^a\right).
\]
Finally, we use the fact that $b \le a - \ell$, so $\psi(m) = \psi(-\alpha)$
when $m \equiv -\alpha \mod{p^{a-\ell}}$, so the whole sum is
\[
    \psi(-\alpha) p^{a-\ell} \sum_{\substack{m=1 \\ p \nmid m \\ m \equiv -\alpha \mod{p^{a-\ell}}}}^{p^{\ell}-1} \chi(m) \, e\!\left(m/p^a\right).
\]
Considering that this expression is equally valid when $\psi$ is the trivial character finishes the proof.
\end{proof}

We extract a few simpler variants of this lemma which allow us to consider
only the values $\chi(p+1)$ (when $p$ is odd) and $\chi(5)$ (when $p=2$).
\begin{lemma}\label{lem:gauss-sum-oddprime}
    Let $\chi$ be a primitive character \textup{(mod }$p^a$\textup{)} for $p > 2$ and some $a > 1$, and $\psi$ a
    character mod $p$. Let $\alpha$ be an integer such that
    \[
        \chi(p + 1) = e\!\left(\frac{\alpha}{p^{a - 1}}\right).
    \]
    Then
    \[
        \tau(\psi \chi) = \psi(-\alpha) \, \tau(\chi).
    \]
\end{lemma}
\begin{proof}
    From the previous lemma, we have
    \[
        \tau(\psi \chi) = \psi(-\alpha') \, \tau(\chi),
    \]
    where $\chi(p^{a-1} \!+\! 1) = e(\alpha'/p).$ We need only show that
    $\alpha \equiv \alpha' \mod p$. The congruence
    \[
       (p + 1)^{p^{a-2}} \equiv p^{a-1} \!+\! 1 \mod{p^a}
    \]
    follows easily by expanding using the binomial theorem, from which
    we see that in fact $\chi(p^{a-1} + 1) = e(\alpha/p)$.
\end{proof}

\begin{lemma}\label{lem:gauss-sum-evenprime}
    Let $\chi$ be a primitive character modulo $2^a$, for $a > 6$, and let
    $\epsilon$ be a character \textup{(mod }$8$\textup{)}. Let $\alpha$ be an integer such that
    \[
        \chi(5) = e(\alpha/2^{a - 2}).
    \]
    Then
    \[
        \tau(\epsilon \chi) = \epsilon(\alpha) \, \tau(\chi).
    \]
\end{lemma}
\begin{proof}
From Lemma \ref{lem:gauss-sum}, we have
    \[
        \tau(\epsilon \chi) = \epsilon(-\alpha') \, \tau(\chi),
    \]
    where $\chi(2^{a-3} \!+\! 1) = e(\alpha'/8)$. For $\ell \ge 4$, applying
    the binomial theorem to $(2^2 + 1)^{7 \cdot 2^{\ell-2}}$, we find the
    congruence $5^{7 \cdot 2^{\ell-2}} \equiv 2^\ell+ 1 \mod{2^{\ell+3}}$.
    As the multiplicative order of $5$ (mod $2^{\ell+3}$) is $2^{\ell+ 1}$, we find
    that the discrete logarithm of $2^\ell+ 1$ with respect to $5$ (mod $2^a$)
    must be congruent to $7 \cdot 2^{\ell-2}$ (mod $2^{\ell+1}$)
    for any $a \ge \ell + 3$. This means that
    \[
        \chi(2^\ell \!+ \!1) = \chi(5)^{7\cdot 2^{\ell-2} + k2^{\ell+1}}
    \]
    for some integer $k$. That is,
    \[
        \chi(2^\ell \!+ \!1) = e\!\left(\frac{7\alpha + 8k\alpha}{2^{a-\ell}}\right).
    \]
    When $\ell = a-3$, this is simply
    \[
        \chi(2^{a-3} \!+ \!1) = e\!\left(\frac{7\alpha}{8}\right),
    \]
    and now Lemma $\ref{lem:gauss-sum}$ finishes the proof.
\end{proof}

\subsection{An auxiliary sum}

We will evaluate $S(\chi)$ by relating it to the sum
\[
    S_2(\psi, r) = \sum_{n=1}^{q_2} \psi(n^2) \, e\!\left(\frac{rn^2}{q_2}\right),
\]
where $\psi$ is a primitive character (mod $q_2$) and $r$ is coprime to $q_2$.
Via a standard application of the Chinese Remainder Theorem, we can factor
$S_2(\psi, r)$ as
\begin{equation}\label{eq:S2-CRT}
    S_2(\psi, r) = \prod_p \psi_p (q_2/p^{e_p})^2 S_2(\psi_p, rq_2/p^{e_p}),
\end{equation}
where $q_2 = \prod_p p^{e_p}$ and $\psi = \prod_p \psi_p$. As one might
already guess from examining the formulae in Proposition \ref{prop:S1},
the relationship
between $S(\chi)$ and $S_2(\psi, r)$ comes from the fact that when $\chi$
is a primitive even character (mod $q$) and $q$ is even and there is a primitive
character $\psi$ (mod $2q$) such that $\psi^2 = \chi$, we will have
$S(\chi) = S_2(\psi, -1)$
for this character $\psi$ (mod $q_2$), where $q_2= 2q$. If $q$ is odd or $\chi$ does
not have a square root then we will see that it is easy to show that $S(\chi)
= 0$, so anytime $S(\chi) \ne 0$ we will be able to evaluate it this way.

\smallskip

It is straightforward to rewrite $S_2$ as a sum of Gauss sums. A sum over the
quadratic characters (mod $q_2$) will detect squares, and count them with the
correct multiplicity, and thus
\[
\begin{split}
    S_2(\psi, r)
    &= \sum_{n=1}^{q_2} \left(\sum_{\xi^2 = \chi_0} \xi(n) \right)
        \psi(n) \, e\!\left(\frac{rn}{q_2}\right) = \sum_{\xi^2 = \chi_0} \sum_{n=1}^{q_2} \xi\psi(n) \, e\!\left(\frac{rn}{q_2}\right) = \sum_{\xi^2 = \chi_0} \overline{\xi \psi}(r) \,\tau(\xi \psi).
\end{split}
\]
This formula does not seem particularly elucidating, however, and we find
it better to use the factorization and do this one prime at a time. Determining
when the sum vanishes still comes down to studying Gauss sums, but only for
prime power moduli.

\subsubsection{Even prime powers}
When $q = 2^a$ is an even prime power (and $a$ is large enough), the quadratic
characters are precisely the characters (mod 8). Studying Gauss sums in this case will
lead us to the following evaluation.

\begin{proposition}\label{prop:S2-evenprime}
    Let $a, r \in \mathbb{Z}^+$ with $r$ odd. Let $\psi_2$ be a primitive character \textup{(mod }$2^a$\textup{)} and $\alpha$ be an integer such that $
   \psi_2(5) = e(\alpha/2^{a-2})$.
    Then we have the following:
    \begin{itemize}
        \item[(a)] If $(i)$ $a < 5$, or $(ii)$ $a = 5$ and $r \equiv \alpha \pmod 4$, or $(iii)$ $a = 6$ and $r \equiv 5\alpha \pmod 8$, then $S_2(\psi_2, r) = 2^{a-1} e(r/2^a).$
    
        \item[(b)] If $a > 6$ and $r \equiv 
    \alpha \pmod 8$, then $S_2(\psi_2,r) = 4 \, \psibar_2(r) \, \tau(\psi_2)$.
    \end{itemize}
    In all other cases, $S_2(\psi_2, r) = 0$.
    %\Bnote{I wonder if there is a clean evaluation of the Gauss sum here when $a > 6$.} {\color{red} M: should this be a lemma?}
\end{proposition}

\begin{proof}
    There are only a small number of cases where $a \le 6$, so this can be done using a small computer computation which we do not include. Otherwise, as described above, this sum is
\[
\begin{split}
S_2(\psi_2,r) 
   =& \sum_{\xi \mod 8}\sum_{n=1}^{2^a} \xi \psi_2(n)\, e(rn/2^a) = \sum_{\xi \mod 8} \overline{\xi\psi_2}(r) \, \tau(\xi \psi_2).
\end{split}
\]
It now follows from Lemma \ref{lem:gauss-sum-evenprime} that
\[
    S_2(\psi_2, r) = \psibar_2(r) \, \tau(\psi_2) \sum_{\xi \mod 8} \overline{\xi}(r) \, \xi(
    \alpha).
\]
By the orthogonality of characters \textup{(mod $8$)}, this equals $4 \, \psibar_2(r) \, \tau(\psi_2)$ if $r \equiv \alpha \pmod 8$ and $0$ otherwise.
This completes the proof of the proposition.
\end{proof}

\subsubsection{Odd prime powers} \label{oddmod}
For any odd prime power, there is exactly one quadratic character, and thus the sum
for characters to such moduli is just
\[
    S_2(\psi_p, r) = \psibar_p(r)\big(\tau(\psi_p) + \epsilon_p(r) \, \tau(\epsilon_p \psi_p)\big).
\]
The next proposition says that this sum does not vanish when the modulus of the character
is a prime, and describes when it does vanish for higher powers.

\begin{proposition}\label{prop:S2-oddprime}
    Given $a\in \mathbb{Z}^+$, let $\psi_p$ be a primitive character modulo $p^a$, for a prime $p > 2$ and $(r,p)=1$. If $a = 1$, then $S_2(\psi_p, r) \ne 0$. Otherwise, let $\alpha$ be an integer such that $\psi_p(p + 1) = e(\alpha/p^{a-1})$. Then
    \[
        S_2(\psi_p, r) = \left(1 + \qr{-\alpha r}{p}\right)\psibar_p(r) \, \tau(\psi_p).
    \]
\end{proposition}
%\Mnote{We have two proofs for $S_2(\psi, r)\neq 0$ does not vanish. One direct proof is using ... One indirect proof is using a ..}
\begin{proof}
    We consider
    \begin{equation}\label{eq:prop5}
        S_2(\psi_p, r) = \psibar_p(r) \, \big(\tau(\psi_p) + \epsilon_p(r) \, \tau(\epsilon_p \psi_p)\big).
    \end{equation}
       In the case $a=1$, both $\psi_p$ and $\epsilon_p$ are primitive characters (mod $p$).
    In this case, we have two proofs that $\tau(\psi_p)\ne \pm \tau(\epsilon_p \psi_p)$ which implies that $S_2(\psi_p,r) \ne 0$. If $\psi_p=\epsilon_p$, then $|\tau(\psi_p)|=\sqrt{p}$ while $\tau(\psi_p \epsilon_p) = -1$, so it is clear that $S_2(\psi_p,r) \ne 0$. Therefore, we may assume that $\psi_p\ne\epsilon_p$.
    
   \smallskip 
    
    One short and indirect proof relies on a beautiful generalization of the Chowla--Mordell theorem due to Evans. The Chowla--Mordell theorem
    states that if $\chi$ is a primitive character (mod $p$), then  $\tau(\chi)/\sqrt{p}$ is a root of unity if and only if $\chi$ is quadratic. If $\chi_1$ and $\chi_2$ are primitive characters (mod $p$), then Evans
    \cite[Theorem 1]{evans} proved that $\tau(\chi_1)/\tau(\chi_2)$ is a root of unity if and only if $\chi_1=\chi_2$. If $\psi_p \ne \epsilon_p$, then $\psi_p$ and $\epsilon_p \psi_p$ are distinct primitive characters (mod $p$). Therefore, Evans' theorem implies that $\tau(\psi_p)/\tau(\epsilon_p \psi_p)\ne \pm 1$. Hence $S_2(\psi_p,r) \ne 0$.
    
   % Applying Theorem \ref{th:evans} with $\chi=\epsilon\psi$, we have $\tau(\psi)/\tau(\epsilon \psi)\neq \pm 1$ so that $|S_2(\psi, r)| = |\tau(\psi)+ \epsilon(r) \tau(\epsilon \psi)| \neq 0$.

\smallskip    
    
    We also provide a direct proof in the case $a=1$ using properties of Jacobi sums and cyclotomic fields. Again, since $\psi_p \ne \epsilon_p$, we know that $\psi_p$ and $\epsilon_p \psi_p$ are distinct primitive characters (mod $p$). We can express the Gauss sums $\tau(\psi_p)$ and $\tau(\epsilon_p \psi_p)$ in terms of the Jacobi sum
    \begin{equation}\label{js}
                J(\psi_p, \epsilon_p) = \sum_{n=1}^p \psi_p(n) \, \epsilon_p(1\!-\! n) = \frac{\tau(\psi_p) \, \tau(\epsilon_p)}{\tau(\epsilon_p \psi_p)}.
    \end{equation}
    Evaluating the Gauss sum for the quadratic character and
    rearranging, we find that
    \[
        \frac{\tau(\epsilon_p \psi_p)}{\tau(\psi_p)} = \frac{\sqrt{\pm p}}{J(\psi_p,\epsilon_p)}.
    \]
    The ratio on the right-hand side cannot be $\pm 1$ because the Jacobi sum is an element of the
    cyclotomic field $\Q(\zeta_{p-1})$ and the elements $\sqrt{\pm p}$ are not. Hence $\tau(\psi_p)/\tau(\epsilon_p \psi_p)\ne \pm 1$ and $S_2(\psi_p,r) \ne 0$.
    
    \smallskip

    In the case where $a > 1$, we use Lemma \ref{lem:gauss-sum-oddprime} to simplify $\tau(\epsilon_p \psi_p)$, giving
    \[
        S_2(\psi_p, r) = \psibar_p(r) \, \big(\tau(\psi_p) + \epsilon_p(-\alpha r)\, \tau(\psi_p)\big).
    \]
    This completes the proof of the proposition.
\end{proof}

\begin{remark}\label{remark:jacobi}% {\color{red} move this to the last section or keep it here?}
We can elaborate on the relationship between $S_2(\psi_p, r)$ and Jacobi sums. Consider
\begin{equation}\label{eq:remark}
S_{\pm} = \frac{\psi_p(4)}{\tau(\epsilon_p)} \cdot \frac{ ( \tau(\psi_p)\pm \tau(\epsilon_p \psi_p))^2}{\tau(\psi_p^2)}.
\end{equation}
By \cite[Theorem 2.1.4]{bew}, if $p>2$, for nontrivial, non-quadratic $\psi_p$ and
quadratic $\epsilon_p$ (mod $p$), we have
\[
    \psi_p(4) \, J(\psi_p,\psi_p)=J(\psi_p,\epsilon_p).
\]
Applying this identity and \eqref{js} in a short calculation, we obtain that
\begin{equation}\label{eq:S2jacobi}
\begin{split}
S_{\pm}
&= \frac{\psi_p(4)}{\tau(\epsilon_p)} \cdot \frac{\tau(\psi_p)^2}{\tau(\psi_p^2)}\Big(1 \pm \frac{\tau(\epsilon_p\psi_p)}{\tau(\psi_p)}\Big)^2
=
\frac{\psi_p(4)J(\psi_p,\psi_p)}{\tau(\epsilon_p)}
\Big(1\pm \frac{\tau(\epsilon_p)}{J(\psi_p,\epsilon_p)}\Big)^2\\
&=
\frac{J(\psi_p,\epsilon_p)}{\tau(\epsilon_p)}
\cdot
\Big(1\pm \frac{\tau(\epsilon_p)}{J(\psi_p,\epsilon_p)}\Big)^2
=\frac{J(\psi_p,\epsilon_p)}{\tau(\epsilon_p)} \pm 2 + \frac{\tau(\epsilon_p)}{J(\psi_p,\epsilon_p)}
= 2 \, \mathrm{Re}\Big( \frac{\tau(\psi_p)}{\tau(\epsilon_p \psi_p)}\Big) \pm 2.
\end{split}
\end{equation}
By work of Katz \cite[Chapter 9]{Katz}, this ratio of Gauss sums becomes uniformly distributed on the unit circle as $p\to\infty$, where $\psi_p$ runs over the non-principal, non-quadratic characters modulo $p$, which in turn suggests something about the distribution of the size of $S_2(\psi_p, r)$. We will return to this in Section \ref{section:distribution}.

%We expect that this ratio of Jacobi and Gauss sums becomes uniformly distributed on the unit circle for large $p$ as $\psi$ varies {\color{red}(see \cite{xi}, for example)}, which in turn suggests something about the distribution of the size of $S_2(\psi_p, r)$. We will return to this in Section \ref{section:distribution}.
\end{remark}

For reference, we summarize what the results of Proposition \ref{prop:S2-evenprime} and Proposition \ref{prop:S2-oddprime} imply about nonvanishing of $S_2(\psi,r)$ once we apply the Chinese Remainder Theorem factorization in \eqref{eq:S2-CRT}.

\begin{proposition}\label{prop:S2}
Let $\psi = \prod_p \psi_p$ be a primitive character \textup{(mod }$q_2$\textup{)}, where $q_2 = \prod_p p^{e_p}$ and each factor $\psi_p$ is a primitive character \textup{(mod }$p^{e_p}$\textup{)}, and let $r$ be an integer with $(r,q_2)=1$. Set $\beta_2$ to be an integer such that
$\psi_2(5) = e(\beta_2/2^{e_2 - 2})$, and for each odd prime $p$
set $\beta_p$ to be an integer such that
$\psi_p(p + 1) = e(\beta_p/p^{e_p - 1})$. Then $S_2(\psi, r) \ne 0$ if and
only if
\begin{itemize}
    \item[(a)] $(i)$ $e_2 <5$, or $(ii)$ $e_2 = 5$ and $\beta_2 \equiv rq_2/2^5 \pmod 4$, or $(iii)$ 
        $e_2 = 6$ and $\beta_2 \equiv 5rq_2/2^6 \pmod 8$, 
        or $(iv)$ $e_2 > 6$ and 
        $\beta_2 \equiv rq_2/2^{e_2} \pmod 8$; and
    \item[(b)] for each odd prime $p$ such that $p^2$ divides $q_2$, the residue $-\beta_p rq_2/p^{e_p}$
        is a square modulo $p$.
\end{itemize}
\end{proposition}
When the sum is nonzero, our proof will give an explicit evaluation in terms of
Gauss sums. It is tedious to write out all of the cases when $e_2$ is small,
but when $e_2 > 6$ we have
\begin{equation} \label{eq:S2-evaluation}
    S_2(\psi, r) = 2^{\nu(q_2) + 1}
        \prod_{p \mid q_2} \psi_p\!\left(\frac{q_2\overline{r}}{p^{e_p}}\right)
        \prod_{p^2 \mid q_2} \tau(\psi_p)
        \prod_{ p \exactdivide q_2}\left(\tau(\psi_p) + \epsilon_p\!\left(\frac{rq_2}{p}\right)\tau(\epsilon_p\psi_p)\right),
\end{equation}
where $\epsilon_p$ is the quadratic character (mod $p$) and $\nu(q_2)$ is the number of primes $p$ dividing $q_2$ such that $p^2$ divides
$q_2$.

\subsection{Evaluation of \texorpdfstring{$S(\chi)$}{S(chi)}} Let $\chi$  \textup{(mod }$q$\textup{)} be a primitive character. Recall that we have already seen that $S(\chi)=0$ if $q$ is odd or $\chi$ is odd. 

\begin{lemma} \label{square}
Let $q\in \mathbb{N}$ be even and let $\chi$  \textup{(mod }$q$\textup{)} be a primitive character. Then
\[
S(\chi)  = \sum_{\substack{\psi \textup{ (mod }2q) \\ \psi^2= \chi}}
 \overline{\psi}(-1) \, \tau(\psi) .
\]
In particular, if $S(\chi)\ne 0$, then there is a primitive character $\psi$  \textup{(mod }$2q$\textup{)} such that $\chi=\psi^2$. 
\end{lemma}

\begin{proof}
If $(a,b)=1$, then 
\[
e(a/b) = \frac{1}{\varphi(b)}\sum_{\xi \mod b} \tau(\xi) \, \xibar(a).
\]
Since $q$ is even, we have
\[
\begin{split}
S(\chi) 
&= \sum_{\substack{n=1 \\ (n,q)=1}}^{2q} \chi(n) \, e\Big(\frac{-n^2}{2q}\Big) = \sum_{n=1}^{2q}\chi(n) \Bigg( \frac{1}{\varphi(2q)}
\sum_{\psi \textup{ (mod }2q)}
\overline{\psi}(-n^2) \, \tau(\psi) \Bigg)\\
&=
\sum_{\psi \textup{ (mod }2q)}
\overline{\psi}(-1) \, \tau(\psi) 
 \Bigg(  \frac{1}{\varphi(2q)}
\sum_{n=1}^{2q}
\Big( \chi\cdot
\overline{\psi}^2\Big)(n) \Bigg).
\end{split}
\]
The inner sum over $n$ equals zero unless $\chi=\psi^2$, in which case it equals $\varphi(2q)$. Therefore
\[
S(\chi) = \sum_{\substack{\psi \textup{ (mod }2q) \\\psi^2= \chi}} \overline{\psi}(-1) \, 
\tau(\psi) .
\]
It is straightforward to show that if $\chi$ is primitive (mod $q$) with $q$ even and if $\psi^2= \chi$, then it must be the case that $\psi$ is primitive (mod $2q$). This proves the lemma. 
\end{proof}

We are now in a position to prove Proposition \ref{prop:S1}.

\begin{proof}[Proof of Proposition \ref{prop:S1}]
 Let $q \in \mathbb{N}$ be even and write $q = 2^{e_2} \prod_{p \mid q} p^{e_p}$, where the product runs over distinct odd primes $p$.  If $\chi$ (mod $q$) is a primitive character with $S(\chi) \ne 0$, then Lemma \ref{square} implies that there is a primitive character $\psi$ (mod $2q$) with $\chi=\psi^2$ and therefore
 \[
 S(\chi) = S(\psi^2) = S_2(\psi,-1). 
 \]
Conversely, suppose that (i) and (iii) hold. Then $\chi$ is the square of a character $\psi$ (mod $2q$), since for odd $p$ an even character (mod $p^{e_p}$) is the square of a character (mod $p^{e_p}$), and an even character (mod $2^{e_2}$) is the square of a character (mod $2^{e_2+1}$). As noted in the proof of Lemma \ref{square}, $\psi$ is primitive, so again $S(\chi)=S_2(\psi,-1)$.
Write $\psi= \prod_p \psi_p$ and $\chi = \prod_p \chi_p$, the products running over all primes $p$. By unique factorization of characters, we have $\chi_p=\psi_p^2$ for each $p$ dividing $q$. 

\smallskip

For odd primes $p$, when $e_p\ge 2$, we define
 \[
    \chi_p(p \!+\! 1) = e(\alpha_p/p^{e_p - 1}) \quad \text{and} \quad \psi_p(p \!+\! 1) = e(\beta_p/p^{e_p - 1}).
 \]
 As $\psi_p(p + 1)^2 = \chi_p(p + 1)$, we see that there is only a single choice
 $\alpha_p \equiv 2\beta_p \pmod p$ for odd $p$. In this case, part (b) of Proposition \ref{prop:S2} implies part (ii) of Proposition \ref{prop:S1}. When $e_p=1$, we note that $\chi_p$ is a square if and only if $\chi_p$ is even which implies part (i) of Proposition \ref{prop:S1}.

\smallskip

At the prime $2$, we note that $\psi_2$ is a primitive character (mod $2^{e_2+1}$) and we write
 \[
    \chi_2(5) = e(\alpha_2/2^{e_2 - 2}) \quad \text{and} \quad \psi_2(5) = e(\beta_2/2^{e_2 - 1}).
 \]
 Since $\chi_2 =\psi_2^2$, we have the
 possibilities $\beta_2 \equiv \alpha_2 \pmod{2^{e_2 - 1}}$ and
 $\beta_2 \equiv \alpha_2 + 2^{e_2 - 2} \pmod{2^{e_2 - 1}}$. If $e_2 = 3$, then
 we are not concerned about the value of $\beta_2$; if $e_2 = 4$, then the value
 of $\beta_2$ is determined (mod 4); if $e_2 > 4$ then the value of $\beta_2$
 is determined (mod 8). In every case, we may apply part (a) of Proposition \ref{prop:S2} to deduce part (iii) of Proposition \ref{prop:S1}.

 \smallskip
 
Finally, the expressions for $\Seven$ and $\Sodd$ in \eqref{Splus} and \eqref{Sminus}, respectively, follow immediately from Proposition \ref{prop:S2-evenprime} and Proposition \ref{prop:S2-oddprime}. This completes the proof. 
\end{proof}

\section{Proof of Theorem \ref{thm:counting_q} and Corollary \ref{sd}}\label{sec:counting}

In this section, we count the number of primitive characters $\chi$ for which
$S(\chi) \ne 0$. Recall that $N(q)$ is the number of primitive $\chi \mod q$ such that $S(\chi) \ne 0$.
This is not a multiplicative function\footnote{For example, $N(p) = 0$ but $N(8p) \ne 0$ for prime $p\geq 5$.},
but from the factorization of the sum using the Chinese Remainder Theorem in \eqref{eq:S2-CRT}, it
is closely related to one.

\smallskip

As we have evaluated $S(\chi)$ by relating it to $S_2(\psi, r)$, it will
similarly be convenient to first count the number of primitive $\psi$ (mod $q$)
such that $S_2(\psi, r) \ne 0$, which we label $N_2(q)$. As we will see in a
moment, this number does not depend on $r$ as long as $r$ is coprime to $q$,
and once it is clear that this number does not depend on $r$, it is also clear
that it is a multiplicative function from the factorization \eqref{eq:S2-CRT}.

\begin{proof}[Proof of Theorem \ref{thm:counting_q}]

When $q = 2^a$, a character $\psi \mod q$ is determined by
its values $\psi(-1) = \pm 1$ and $\psi(5) = e(\alpha/2^{a-2})$ for the $2^{a-2}$ 
%%%%%
%{\color{violet}[Z:$\alpha$ is odd, so $2^{a-3}$ choices of $\alpha$?]}
%%%%%
different possible values of $\alpha$. The primitive characters are exactly
those for which $\alpha$ is odd, and clearly every residue class modulo $8$
is covered the same number of times by the possible values of $\alpha$. As
$N_2(2^a)$ depends only on $\alpha$ (mod $8$) we see that it is independent
of $r$ and, examining Proposition \ref{prop:S2-evenprime}, we find that:
\begin{itemize}
    \item[($i$)] $N_2(2) = 0$ (there are no primitive characters);
    \item[($ii$)] for $2 \le a < 5$, we have $N_2(2^a) = 2^{a-2}$ (the sum is always nonzero);
    \item[($iii$)] $N_2(2^5) = 4$ (there are 8 primitive characters, and the sum vanishes for half of them);
    \item[($iv$)] for $a > 5$, we have $N_2(2^a) = 2^{a - 4}$ (there are $2^{a-2}$ primitive characters and the sum is nonzero one quarter of the time).
\end{itemize}

Now suppose that $q=p^a$ for an odd prime $p$. When $a=1$, by Proposition \ref{prop:S2-oddprime}, the sum $S_2(\psi, r)$ never vanishes, so $N_2(p) = p - 2$.
When $a \ge 2$, the vanishing or nonvanishing of $S_2(\psi, r)$ is
determined by the congruence class of $\alpha$ (mod $p$), where $\psi(p + 1) =
e(\alpha/p^{a-1})$. For half of the $\alpha$ (mod $p$) the sum will vanish and
for the other half it will not. The primitive characters are exactly those
for which $p \nmid \alpha$, and, as the character group is cyclic, each value
of $\alpha$ occurs the same number of times. There are $\varphi(p^a) - \varphi(p^{a-1})$
primitive characters, and for half of them the sum does not vanish, so
$N_2(p^a) = p^{a-2}(p-1)^2/2$ when $a \ge 2$.

\smallskip

Now, if $8$ divides $q$, then for every primitive character $\psi$ (mod $2q$)
such that no factor $\psi_p$ is quadratic, the
character $\psi^2$ has conductor $q$. So, we can compute $N(q)$ by computing
$N_2(2q)$ and adjusting for the multiplicity of squares and each quadratic
character (mod $p$). For each odd prime $p$
dividing $q$, this is a factor of $2$, and there is another factor of $4$ at the prime
$2$, so we find that
\[
    N(q) = \big\lceil  2^{m-5} \big\rceil  \prod_{p \exactdivide q} \left( \frac{p-3}{2} \right)
        \prod_{\substack{p^k \exactdivide q \\ p \textup{ odd},\ k > 1 }} \left( \frac{p^{k-2}(p-1)^2}{4} \right),
\]
which proves \eqref{eq:counting_q}, and completes the proof of the theorem.
\end{proof}

We now use the Selberg--Delange method to deduce Corollary \ref{sd} from Theorem \ref{thm:counting_q}.

\begin{proof}[Proof of Corollary \ref{sd}]
Using Theorem \ref{thm:counting_q}, 
%With this formula in hand, we now proceed to use the Selberg--Delange method to prove \eqref{eq:counting_Q} in Theorem  From Proposition \ref{prop:S2-evenprime}, we know that $N(2^a)=0$ if $a=1,2$ and $N(2^a)=\lceil 2^{a-5}\rceil$ if $a\ge 3$. We define a multiplicative function $N_0$ supported on odd positive integers by $N_0(1)=1$, $N_0(p)=\frac{p-3}{2}$, and $N_0(p^a)=\frac{p^{a-2}(p-1)^2}{4}$ for odd primes $p$ and $a\ge2$. Then, for odd positive integers $k$, \eqref{eq:counting_q} implies that $N(2^ak)=N(2^a) N_0(k)$ and 
we can formally expand the Dirichlet series
\begin{equation*}
\begin{split}
F(s)
:=\sum_{q=1}^{\infty}\frac{N(q)}{q^{s+1}}
&=\left(\frac{1}{2^{3(s+1)}}+\frac{1}{2^{4(s+1)}} + \sum_{a\geq 5}\frac{2^{a-5}}{2^{a(s+1)}}\right) \prod_{p\neq 2} \left( 1+\frac{p-3}{2p^{s+1}}+ \sum_{a\geq 2}\frac{p^{ a-2}(p-1)^2}{4p^{a(s+1)}}\right)\\
&=f(s) \, g(s) \, \prod_{p} \left( 1+\frac{1}{2p^s}\right),\\
\end{split}
\end{equation*}
where
\[
f(s) =\left(\frac{1}{2^{3(s+1)}}+\frac{1}{2^{4(s+1)}}+\frac{1}{2^{4s+5}(2^s-1)}\right)
\left(1-\frac{1}{2^{s+2}}+\frac{1}{2^{s+4}(2^{s}-1)}\right)^{-1}\\
\]
and
\[
g(s)
= \prod_{p}\left(1-\frac{3}{(2p^{s}+1)p}+\frac{(p-1)^2}{2p^2(p^s-1)(2p^s+1)}\right).
\]
%For any $\varepsilon>0$, we claim that $F(s)\zeta(s)^{-1/2}$ is analytic (and its Dirichlet series is uniformly bounded) in the half-plane $\mathrm{Re}(s)\geq 1/2+\varepsilon$. First 
Note that $f(s) \cdot g(s)$ is analytic, nonzero, and uniformly bounded for $\mathrm{Re}(s)\geq 1/2+\varepsilon$. If we let
\[
h(s)= \prod_{p}\left(1-\frac{1}{4p^{2s}}\right)\left(1+\frac{1}{2p^{2s}+p^s-1}\right), 
\]
then $h(s)$ is analytic, nonzero, and uniformly bounded in the half-plane $\mathrm{Re}(s)\geq 1/2+\varepsilon$, and satisfies
\[
\prod_{p}\left(1+\frac{1}{2p^s}\right)^2 = \frac{\zeta(s)}{\zeta(2s)} \, h(s). 
\]
Therefore, we have shown that 
\[
F(s)\, \zeta(s)^{-1/2} = f(s) \, g(s) \, h(s)^{1/2} \, \zeta(2s)^{-1/2},
\]
and hence $F(s)\zeta(s)^{-1/2}$ is analytic and uniformly bounded in the half-plane $\mathrm{Re}(s)\geq 1/2+\varepsilon$. Thus, for $Q\geq 3$, using the Selberg--Delange method as stated in \cite[Chapter II, Theorem 5.2]{tenenbaum}, we derive that 
\begin{equation*}
\sum_{q\leq Q} \frac{N(q)}{q}=  
f(1) \, g(1) \sqrt{\frac{h(1)}{ \pi \, \zeta(2)}} \frac{Q}{\sqrt{\log Q}}+ O\!\left(\frac{Q}{(\log Q)^{3/2}}\right)\!.
\end{equation*}
Summation by parts implies that
\begin{equation} \label{constantA}
\sum_{q\leq Q} N(q) = \frac{f(1) \, g(1)}{2}  \sqrt{\frac{h(1)}{\pi \, \zeta(2)}} \frac{Q^2}{\sqrt{\log Q}}+ O\!\left(\frac{Q^2}{(\log Q)^{3/2}}\right)\!,
\end{equation}
which proves \eqref{eq:counting_Q} with
\begin{equation}\label{Avalue}
\mathscr{A} =\frac{f(1) \, g(1)}{2} \sqrt{\frac{h(1)}{\pi \, \zeta(2)}}.
\end{equation}
This completes the proof of the corollary.
\end{proof}

%%%%%%%%%

%
%\section{Some curiosities}\label{sec:curiosity}
%
%TBA... \Znote{Do we need still need to rewrite the section of curiosities?}
%
%{\color{red} We previously had some discussion about the size of $|S(\chi)|$.}

\section{Proofs of Theorem \ref{thm:GL1}, Theorem \ref{thm:holo}, and Theorem \ref{thm:Maass}}\label{sec:integral}

In this section, we prove Theorem \ref{thm:GL1}, Theorem \ref{thm:holo}, and Theorem \ref{thm:Maass}. Instead of three separate proofs, we estimate the integral of a $Z$-function associated to a generic $L$-function and then specialize to the three cases covered in these theorems. 

\subsection{\texorpdfstring{$L$}{L}-functions}
For $d \in \mathbb{N}$, following \cite[Chapter 5]{IK}, we suppose that an $L$-function of degree $d$ is given by the absolutely convergent Dirichlet series and Euler product
\[
L(s,f)=\sum_{n\geq 1} \lambda_f(n) n^{-s} = \prod_p \big(1- \alpha_1(p) p^{-s} \big)^{-1} \cdots \big(1- \alpha_d(p) p^{-s} \big)^{-1}
\]
for $\mathrm{Re}(s)>1$, where $\lambda_f(1)=1$, $\lambda_f(n)\in \mathbb{C}$, and $\alpha_i(p) \in \mathbb{C}$. We assume that $L(s,f)$ extends to a meromorphic function on $\mathbb{C}$ with at most a pole at $s=1$. We further assume that $L(s,f)$ satisfies a functional equation, in asymmetric form, given by 
\[
L(s,f)= \G_f(s) \, L(1\!-\!s, \overline{f}) \quad \text{and} \quad \G_f(s)= \varepsilon_f \, N^{1/2-s} \prod_{j=1}^{d}\frac{\Gamma_{\mathbb{R}}(1\!-\!s\!+\!\kappa_j)}{\Gamma_{\mathbb{R}}(s\!+\!\kappa_j)},
\]
where $L(s, \overline{f}) = \overline{L(\bar{s}, f)}$.
Here $N$ is a positive integer (depending on the $L$-function), $|\varepsilon_f|=1$, and $\kappa_j\in \mathbb{C}$. We assume that the $\kappa_j$ are either real or come in conjugate pairs\footnote{Though this assumption does not hold in general, it is valid for the $L$-functions under consideration in this paper. We are able to use this assumption to define an analytic function which equals $Z_f(t)$ on the critical line.} and that $\mathrm{Re}(\kappa_j)>-\frac{1}{2}$.
%\Znote{Do we need to assume $\mathrm{Re}(\kappa_j)>-\frac{1}{2}$ for an analytic $G_f(s)$ on the half plane of $\mathrm{Re}(s)\geq 1/2$?}
With these assumptions, we define the $Z$-function, $Z_f(t)$, associated to $L(s,f)$ by
\[
Z_f(t) =  \G_f(\tfrac{1}{2}\!+\!it)^{-1/2} \, L(\tfrac{1}{2}\!+\!it,f).
\]
Then $Z_f(t)$ is real-valued for $t\in \mathbb{R}$, infinitely differentiable, and $|Z_f(t)|=|L(\tfrac{1}{2}+it,f)|$.

\subsection{Stationary phase}

Throughout the section, we assume $N$ is fixed and we allow the implied constant to depend on $N$ (and therefore on $L(s,f)$). %Adapting Gonek's stationary phase argument, we establish the following lemma. 

\begin{lemma}\label{gonek}
Let $\ell>0$ be fixed, and let $a$ lie in a fixed bounded interval. For large $x$ with $x<r\leq 2x$, we have, uniformly in $a$,
\[
\int_{x}^{2x} t^{a} \left(\frac{t}{r e}\right)^{it \ell}  \mathrm{d}t = \sqrt{\frac{2\pi }{\ell }} \, r^{a+\frac{1}{2}} \, e^{-ir\ell  + \frac{\pi i}{4}} + \mathcal{E}(a,\ell,r,x),
\]
where
\[
\mathcal{E}(a,\ell,r,x) = O\big((\ell x)^{a}\big) + O\!\left(\frac{(\ell x)^{a+1}}{\ell \, |x-r|+\sqrt{\ell x}}\right) +  O\!\left(\frac{(\ell x)^{a+1}}{\ell \, |2x-r|+\sqrt{\ell x}}\right).
\]
When $r\leq x$ or $r> 2x$, there is no main term.
\end{lemma}
\begin{proof}
This is a straightforward consequence of \cite[Lemma 2]{gonek}, after a suitable variable change.
\end{proof}

\begin{proposition}\label{mainlemma}
Let $T$ be large, $\mathcal{T}_f= N^{1/d} \, T/(2\pi)$, and $c=1+1/\log \mathcal{T}_f$. Assume that
\begin{equation} \label{L bounds}
|L(\tfrac{1}{2}+it,f)| \ll (|t|+1)^{d/4} \quad \text{and} \quad \sum_{n\ge 1} \frac{|\lambda_f(n)|}{n^c} \ll (\log T)^\beta
\end{equation}
for some $\beta \ge 0$. Then
\begin{equation*}\label{7/16}
\begin{split}
\int_T^{2T} Z_f(t) \, \mathrm{d}t &= A_f \sum_{\mathcal{T}_f < n^{2/d} \le 2\mathcal{T}_f } \lambda_f(n) \, n^{1/d-1/2} \, e\!\left( -\frac{d}{2} \left(\frac{n^2}{N} \right)^{\!1/d} \right) + O\!\left(T^{d/4} (\log T)^\beta \right)
\\
&\qquad\qquad+O\!\left( T^{1-d/4} \max_{\alpha \in \{1,2\}} \left\{ \sum_{\frac{3\alpha }{4}\mathcal{T}_f\le n^{2/d} \le \frac{5\alpha}{4} \mathcal{T}_f } \frac{|\lambda_f(n)|}{|n^{2/d}\!-\!\alpha \mathcal{T}_f|+\sqrt{T}} \right\}  \right),
\end{split}
\end{equation*}
where
\[
A_f = 2^{3/2}\pi\,  (\varepsilon_f \, d \, N^{1/d})^{-1/2} \ e\!\left( \frac{1}{8} \sum_{1\le j \le d} \mathrm{Re}(\kappa_j)  + \frac{1}{8} - \frac{d}{16}\right),
\]
and the implied constants depend on the $L$-function $L(s,f)$. 
\end{proposition}

{\sc Remark}. The $L$-functions considered in Theorem \ref{thm:GL1}, Theorem \ref{thm:holo}, and Theorem \ref{thm:Maass} all satisfy the bounds in \eqref{L bounds}.

\begin{proof}
Making the variable change $s=\tfrac{1}{2}+it$, we have
\[
\int_T^{2T} Z_f(t) \, \mathrm{d}t = \frac{1}{i} \int_{1/2+iT}^{1/2+2iT} \G_f(s)^{-1/2} \, L(s,f) \, \mathrm{d}s. 
\]
The function $\G_f(s)$ is analytic and nonzero for $\mathrm{Re}(s) \ge 1/2$ and $\mathrm{Im}(s) \ge T$, when $T$ is sufficiently large. Since this region is convex, we can define an analytic branch of $\G_f(s)^{-1/2}$ in it, and we choose the branch that agrees with the one defining $Z_f(t)$ on the segment from $\frac12+iT$ to $\frac12+2iT$. Cauchy's Theorem, applied to the positively oriented rectangle with vertices 
\[ 
\Big[\tfrac{1}{2}+iT,c+iT,c+2iT,\tfrac{1}{2}+2iT \Big],
\] 
implies that
\[
\int_T^{2T} Z_f(t) \, \mathrm{d}t =  \frac{1}{i} \int_{c+iT}^{c+2iT} \G_f(s)^{-1/2} \, L(s,f) \, \mathrm{d}s + H_1 -H_2,
\]
where $c=1+1/\log \mathcal{T}_f$ and
\[
H_k = \frac{1}{i} \int_{1/2+ k i T}^{c+ k i T} \G_f(s)^{-1/2} \, L(s,f) \, \mathrm{d}s
\]
for $k=1,2$. Note that trivially
\begin{equation} \label{PLmax}
|H_k| \le \big(c-\tfrac{1}{2}\big) \max_{\frac{1}{2}\le \sigma \le c} \big|\G_f(\sigma+kiT)^{-1/2} \, L(\sigma+kiT,f) \big|.
\end{equation}
Since $|\G_f(\tfrac{1}{2}+it)|=1$ for $t\in\mathbb{R}$, by \eqref{L bounds} we see that 
\begin{equation} \label{PL1}
\big|\G_f(\tfrac{1}{2}+it)^{-1/2} \, L(\tfrac{1}{2}+it,f) \big| \ll t^{d/4}
\end{equation}
as $t\to\infty$. Also, by Stirling's formula for the gamma function, we have
\begin{equation}\label{stirling}
\begin{split}
\G_f&(\sigma+it)^{-1/2} 
\\
&= 
\varepsilon_f^{-1/2}
\,e\!\left(\frac{1}{8}\sum_{j=1}^{d}\mathrm{Re}(\kappa_j)-\frac{d}{16}\right)
\left( N^{1/d} \frac{t}{2\pi}\right)^{d(\sigma-\frac{1}{2})/2}
\left( N^{1/d}\frac{t}{2\pi e}\right)^{itd/2}
\left\{1+O\left(\!\frac{1}{|t|}\!\right)\right\},
\end{split}
\end{equation}
uniformly for $\sigma$ in a compact interval and $t$ sufficiently large. So, again by \eqref{L bounds}, we deduce that
\begin{equation} \label{PL2}
\big|\G_f(c+it)^{-1/2} \, L(c +it,f) \big| \ll t^{d/4} \sum_{n\ge 1} \frac{|\lambda_f(n)|}{n^c} \ll t^{d/4} (\log T)^\beta
\end{equation}
as $t\to\infty$. Using \eqref{PL1} and \eqref{PL2}, an application of the Phragm\'en--Lindel\"of principle  now implies that
% the maximum in \eqref{PLmax} occurs at $\sigma =c$ and thus
\[
|H_1|+|H_2| \ll T^{d/4} (\log T)^\beta.
\]

Next we expand $L(s,f)$ as an absolutely convergent Dirichlet series, interchange the sum and the integral, and then make the variable change $s=c+it$, to see that
\[
\begin{split}
\frac{1}{i} &\int_{c+iT}^{c+2iT} \G_f(s)^{-1/2} \, L(s,f) \, \mathrm{d}s = \sum_{n\ge 1} \frac{\lambda_f(n)}{n^c} \int_T^{2T} \G_f(c+it)^{-1/2} \, n^{-it} \, \mathrm{d}t
\\
&= 
\varepsilon_f^{-1/2}\,e\!\left(\frac{1}{8}\sum_{j=1}^{d}\mathrm{Re}(\kappa_j)-\frac{d}{16}\right)
\sum_{n\ge 1} \frac{\lambda_f(n)}{n^c} \int_T^{2T} 
\left(N^{1/d}\frac{t}{2\pi}\right)^{d(c-\frac{1}{2})/2}
\left(N^{1/d}\frac{t}{2\pi n^{2/d}e}\right)^{itd/2} \mathrm{d}t
\\
& \qquad  + O\!\left(T^{d/4} (\log T)^\beta \right).
\end{split}
\]
Here, the second equality comes after applying Stirling's formula for $\G_f(c+it)^{-1/2}$ in \eqref{stirling} and then using \eqref{L bounds} to obtain the big-$O$ term. Making the variable change $N^{1/d} \,t/(2\pi) \mapsto t$ and then combining estimates, we have deduced that
\begin{equation*}\label{general_eq}
\begin{split}
\int_T^{2T} Z_f(t) \, \mathrm{d}t 
&=
\varepsilon_f^{-1/2}\,e\!\left(\frac{1}{8}\sum_{j=1}^{d}\mathrm{Re}(\kappa_j)-\frac{d}{16}\right)
\frac{2\pi}{N^{1/d}}
\sum_{n\ge 1} \frac{\lambda_f(n)}{n^c} \int_{\mathcal{T}_f}^{2\mathcal{T}_f} t^{d(c-\frac{1}{2})/2}\left( \frac{t}{n^{2/d}e}\right)^{ i t \pi d N^{-1/d}}\!\mathrm{d}t
\\
&\qquad\qquad+ O\!\left(T^{d/4} (\log T)^\beta \right).
\end{split}
\end{equation*}
We now use Lemma \ref{gonek}  to estimate the integral on the right-hand side. The main term is 
\[
A_f \sum_{\mathcal{T}_f < n^{2/d} \le 2\mathcal{T}_f } \lambda_f(n) \, n^{1/d-1/2} \, e\!\left( -\frac{d}{2} \left(\frac{n^2}{N} \right)^{\!1/d} \right),
\]
while the error term (using \eqref{L bounds} and the notation in Lemma \ref{gonek}) is
\[
\begin{split}
&\ll \sum_{ n \ge 1 } \frac{|\lambda_f(n)|}{n^c}\left\{ T^{d/4}+  \mathcal{E}\Big(\tfrac{d}{2}(c-\tfrac{1}{2}),\pi d N^{-1/d},n^{2/d},\mathcal{T}_f\Big)\right\}\\
&\ll T^{d/4} (\log T)^\beta + T^{1+d/4} \max_{\alpha \in \{1,2\}} \left\{ \sum_{n \ge 1} \frac{|\lambda_f(n)|}{n^c}\frac{1}{|n^{2/d}\!-\!\alpha \mathcal{T}_f|+\sqrt{T}} \right\} 
\\
&\ll T^{d/4} (\log T)^\beta + T^{1-d/4} \max_{\alpha \in \{1,2\}} \left\{ \sum_{\frac{3\alpha }{4}\mathcal{T}_f\le n^{2/d} \le \frac{5\alpha}{4} \mathcal{T}_f}\frac{|\lambda_f(n)|}{|n^{2/d}\!-\!\alpha \mathcal{T}_f|+\sqrt{T}} \right\}.
\end{split}
 \]
Combining estimates, the proposition now follows.
\end{proof}

\subsection{Dirichlet \texorpdfstring{$L$}{L}-functions: Proof of Theorem \ref{thm:GL1}}
For a Dirichlet $L$-function $L(s,\chi)$ associated to a primitive character $\chi$ (mod $q$), Proposition \ref{mainlemma} implies that
\begin{equation}\label{big-O}
\begin{split}
\int_{T}^{2T}Z_\chi(t)\,\mathrm{d}t
& \ = \ A_\chi
\sum_{ \sqrt{\mathcal{T}_\chi} < n \le \sqrt{2\mathcal{T}_\chi}} \!\!\!\!\!\!\!\! \chi(n) \, \sqrt{n} \, e\!\left(\frac{-n^2}{2q}\right)  +O\!\left(T^{1/4}\log T\right)\\
&\qquad \quad +O\left(T^{3/4}\underset{\alpha\in \{1,2\}}{\max}\left\{ \sum_{\frac{3\alpha}{4}\mathcal{T}_\chi \leq n^2\leq \frac{5\alpha}{4}\mathcal{T}_\chi}
\frac{1}{|n^2 - \alpha \mathcal{T}_\chi| +\sqrt{T}}\right\}\right),
\end{split}
\end{equation}
where
\[
A_\chi = \frac{2^{3/2}\, i^{\mathfrak{a}/2} \, \pi }{q^{1/4} \sqrt{\tau(\chi)}}  \, e\!\left(\frac{\mathfrak{a}}{8}+\frac{1}{16}\right)   \quad \text{and} \quad \mathcal{T}_\chi=\frac{qT}{2\pi}.
\]
Here and throughout this proof, all implied constants are allowed to depend on $q$. It makes the calculations cleaner if we first estimate the error terms and then sum over dyadic intervals (to recover the integral from $0$ to $T$) before evaluating the main term. Since the second big-$O$ term can be bounded in a similar manner when $\alpha=1$ or $\alpha=2$, we only give the details of how to estimate this term when $\alpha=1$. In this case, we partition the interval $[\sqrt{3\mathcal{T}_\chi/4}, \sqrt{5\mathcal{T}_\chi/4}]$ into three subintervals:
\[ 
\begin{split}
V_1&=\left[\sqrt{\mathcal{T}_\chi}-\log T, \sqrt{\mathcal{T}_\chi}+\log T\right],\\ 
V_2&=\left[\sqrt{3\mathcal{T}_\chi/4},\sqrt{\mathcal{T}_\chi} - \log T\right), \text{ and} \\
V_3&=\left(\sqrt{\mathcal{T}_\chi}+\log T, \sqrt{5\mathcal{T}_\chi /4}\right].
\end{split}
\] 
Accordingly, we define
\[
\Sigma_j=T^{3/4}\sum_{n\in V_j}\frac{1}{|\mathcal{T}_\chi\!-\!n^2|+\sqrt{T}},
\]
for $1\leq j\leq 3$, so that this error term is $O(\Sigma_1+\Sigma_2+\Sigma_3)$.  For $n\in V_1$, we have 
\[
\Sigma_1\ll T^{3/4} \sum_{n\in V_1} \frac{1}{\sqrt{T}} \ll T^{1/4} |V_1| \ll T^{1/4}\log T.
\]
For $n\in V_2$, by the integral test, we observe that
\[
\begin{split}
\Sigma_2
&\ll T^{1/4} + T^{3/4}\int_{V_2}\frac{\mathrm{d} u}{\mathcal{T}_\chi+\sqrt{\mathcal{T}_\chi}-u^2}. 
\end{split}
\]
By partial fractions, the integral on the right-hand side is equal to
\[
\frac{1 }{2\sqrt{\mathcal{T}_\chi+\sqrt{\mathcal{T}_\chi}}} \  \int_{V_2} \left( \frac{1}{\sqrt{\mathcal{T}_\chi+\sqrt{\mathcal{T}_\chi}}-u} + \frac{1}{\sqrt{\mathcal{T}_\chi+\sqrt{\mathcal{T}_\chi}}+u} \right) \mathrm{d}u \ll \frac{\log T}{\sqrt{T}}.
\]
Therefore, $\Sigma_2 \ll T^{1/4}\log T$. We can bound the contribution from the $n\in V_3$ similarly to $n \in V_2$ and conclude that $\Sigma_3\ll T^{1/4}\log T$, as well. Estimating similarly when $\alpha=2$, we conclude that
\[
\int_{T}^{2T}Z_\chi(t)\,\mathrm{d}t
 \ = \ A_\chi \sum_{ \sqrt{\mathcal{T}_\chi} < n \le \sqrt{2\mathcal{T}_\chi}}
\!\!\!\!\!\!\!\! \chi(n) \, \sqrt{n} \, e\!\left(\frac{-n^2}{2q}\right)  +O\!\left(T^{1/4}\log T\right).
\]
Summing this estimate dyadically over the intervals $[T/2,T], [T/4,T/2], [T/8,T/4], \ldots$ and then using the definition of $\mathcal{T}_\chi$ and $A_\chi$, we derive that
\[
\int_{0}^{T}Z_\chi(t)\,\mathrm{d}t
 \ = \ \frac{2^{3/2}\, i^{\mathfrak{a}/2} \, \pi}{q^{1/4} \sqrt{\tau(\chi)}}  \, e\!\left(\frac{\mathfrak{a}}{8}+\frac{1}{16}\right)  
\sum_{ n \le \sqrt{\frac{qT}{2\pi}}}  \chi(n) \, \sqrt{n} \, e\!\left(\frac{-n^2}{2q}\right)  +O\!\left(T^{1/4}\log T\right).
\]
Applying partial summation to the sum over $n$ on the right-hand side, we derive that
\[
\sum_{n\leq x} \chi(n) \sqrt{n}\,e\!\left(\frac{-n^2}{2q}\right)
=\frac{S(\chi)}{3q}\, x^{3/2} + O(x^{1/2}).
\]
The desired formula follows by putting $x=\sqrt{qT/(2\pi)}$.
Note that we may assume that $\mathfrak{a}=0$ as otherwise, if $\mathfrak{a}=1$, then $S(\chi)=0$. The final assertion of Theorem \ref{thm:GL1} follows from Theorem \ref{thm:counting_q}, since $N(8p)=\frac{p-3}{2}>0$ for every prime $p\ge5$. This completes the proof of Theorem \ref{thm:GL1}.

\subsection{Holomorphic cusp forms and Hecke--Maass cusp forms} 
\subsubsection{Averages of coefficients of \texorpdfstring{$L$}{L}-functions in short intervals}
In order to use Proposition \ref{mainlemma} to prove Theorem \ref{thm:holo} and Theorem \ref{thm:Maass}, we need to understand how the normalized coefficients of primitive holomorphic cusp forms and primitive Hecke--Maass cusp forms behave in short intervals. 

\begin{lemma}\label{shortsums}
Let $f \in S_k(\Gamma_0(N),\chi)$ be a primitive holomorphic cusp form. Then 
\begin{equation} \label{short1}
\sum_{x<n \le x+\sqrt{x}} |\lambda_f(n)| \ll_f \sqrt{x}, \quad \text{as } x \to \infty.
\end{equation}
Let $\varphi$ be a primitive Hecke--Maass cusp form of level $N$ with nebentypus  $\chi$. Then, for any $\varepsilon>0$, we have
\begin{equation} \label{short2}
\sum_{x<n \le x+\sqrt{x}} |\lambda_\varphi(n)| \ll_{\varphi,\varepsilon} x^{\frac{11}{20}+\varepsilon}, \quad \text{as } x \to \infty.
\end{equation}
\end{lemma}

\begin{proof}
If $f \in S_k(\Gamma_0(N),\chi)$ is a primitive holomorphic cusp form, then it is known that
\[
\sum_{p\le x} \frac{|\lambda_f(p)|^2}{p} = \log \log x + O(1).
\]
For instance, this follows from \cite[Proposition 2.3]{rs} and partial summation. By Cauchy's inequality, 
\begin{equation} \label{primes_est}
\sum_{p\le x} \frac{|\lambda_f(p)|}{p} \le \sqrt{\sum_{p\le x} \frac{1}{p}} \ \sqrt{\sum_{p\le x} \frac{|\lambda_f(p)|^2}{p}} \le \log \log x + O(1).
\end{equation}
%for some $\delta>0$. {\color{red} This follows with = from the proof of Sato-Tate conjecture... should we cite this Rankin/Kim/etc? I don't know the proper reference offhand.}  Rankin worked with primitive cusp forms on the full modular group, but his argument can be extended to all primitive $f \in S_k(\Gamma_0(N),\chi)$ using known estimates for symmetric square and symmetric fourth power $L$-functions (e.g.~\cite{kim} and \cite{kimshahidi}). 
The multiplicative function $|\lambda_f(n)|$ belongs to the class $M$ defined by Shiu in \cite{shiu}. Therefore, using \eqref{primes_est} in \cite[Theorem 1]{shiu}, we derive the bound in \eqref{short1}.

\smallskip

To derive \eqref{short2}, we cannot use Shiu's result since it is not known if the coefficients $\lambda_\varphi(n)$ satisfy the Ramanujan--Petersson conjecture. Instead, we use the Rankin--Selberg estimate
\begin{equation} \label{short3}
\sum_{n\le x } |\lambda_\varphi(n)|^2 = c_\varphi \, x + O(x^{3/5+\varepsilon}), 
\end{equation}
which holds for some constant $c_\varphi>0$. An estimate of this form is classical for primitive holomorphic cusp forms. This estimate can be proved in the same way using the Rankin--Selberg $L$-function $L(s,\varphi \otimes \overline{\varphi})$ and applying \cite[Theorem 4.1]{cn}. Using \eqref{short3} and Cauchy's inequality, we see that
\[
    \sum_{x<n\le x+\sqrt{x}}  |\lambda_\varphi(n)| \le x^\frac14 \sqrt{\sum_{x<n\le x+\sqrt{x}}  |\lambda_\varphi(n)|^2 } \ll x^{\frac{11}{20}+\varepsilon}.
\]
This completes the proof of the lemma. 
\end{proof}

%{\color{red}{\sc Remark.} Using the proof of the Sato-Tate conjecture \cite{} for holomorphic cusp forms (or an argument of Rankin \cite{rankin}) along with \cite[Theorem 1]{shiu}, we can show that
%\[
%\sum_{x<n \le x+\sqrt{x}} |\lambda_f(n)| \ll_f \sqrt{x} \, (\log x)^{-\delta}, \quad \text{as } x \to \infty,
%\]
%for some $\delta>0$. Such a bound is not necessary for our proof of Theorem \ref{thm:holo}}. 

\subsubsection{Proof of Theorem \ref{thm:holo}}

Let $f \in S_k(\Gamma_0(N),\chi)$ be a primitive holomorphic cusp form. Then
 Proposition \ref{mainlemma} implies that
\begin{equation}\label{gl_2}
\begin{split}
\int_T^{2T}Z_f(t)\,\mathrm{d}t
&=\frac{2\pi}{\varepsilon_f^{1/2} N^{1/4}} \,e\!\left(\frac{k}{8}\right) 
\sum_{\mathcal{T}_f<n\leq 2\mathcal{T}_f}
\lambda_f(n) \,e\!\left(-\frac{n}{\sqrt{N}}\right) + O(T^{1/2}\log T)\\
&\qquad \qquad +\,O\!\left( T^{1/2}\max_{\alpha\in \{1,2\}}\left\{\sum_{\frac{3\alpha}{4}\mathcal{T}_f \leq n\leq \frac{5\alpha}{4}\mathcal{T}_f}
\frac{|\lambda_f(n)|}{|n\!-\!\alpha \mathcal{T}_f|+\sqrt{T}}\right\}\right).
\end{split}
\end{equation}
%\[
%A_f=\frac{2\pi}{\varepsilon_f^{1/2} N^{1/4}} \,e\!\left(\frac{k}{8}\right).
%\]
%Throughout this section, all big-$O$ terms are allowed to depend on $f$. 
As in the proof of Theorem \ref{thm:GL1}, before estimating the main term, we first estimate the second big-$O$ term in \eqref{gl_2} and then we sum over dyadic intervals to recover the integral from $0$ to $T$.

\smallskip
 
When $\alpha=1$, we partition the interval $[3\mathcal{T}_f/4, 5\mathcal{T}_f/4]$ into
\[
\begin{split}
V_1&=\big[\mathcal{T}_f -\sqrt{T}, \mathcal{T}_f +\sqrt{T}\big], \quad
V_2=\big[3\mathcal{T}_f/4, \mathcal{T}_f-\sqrt{T}\big), \quad \text{and} \quad
V_3=\big(\mathcal{T}_f+\sqrt{T},5\mathcal{T}_f/4\big].
\end{split}
\]
Accordingly, we let 
\[
\Sigma_j=T^{1/2}\sum_{n\in V_j}\frac{|\lambda_f(n)|}{|n-\mathcal{T}_f|+\sqrt{T}},
\]
for $j=1,2,3$, and the second error term of \eqref{gl_2} is $O(\Sigma_1+\Sigma_2+\Sigma_3)$. When $n\in V_1$, we simply note that $\frac{1}{|n-\mathcal{T}_f|+\sqrt{T}}\leq \frac{1}{\sqrt{T}}$, and then the first estimate in Lemma \ref{shortsums} implies  that
\[
\Sigma_1\leq \sum_{n\in V_1}|\lambda_f(n)|\ll T^{1/2}.
\]
For $n\in V_2$, we further partition $V_2$ into
\[
V_2 = \bigcup_{j=1}^{J}U_j: = \bigcup_{j=1}^{J}\left\{n\in V_2:  \mathcal{T}_f-2^{j}\sqrt{T}\leq n<\mathcal{T}_f -2^{j-1}\sqrt{T}\right\},
\]
where $J=\big\lceil \log\big(\frac{\mathcal{T}_f}{4\sqrt{T}}\big)/\log(2)\big\rceil$. On each interval $U_j$, we have $|n-\mathcal{T}_f|\geq 2^{j-1}\sqrt{T}$. We partition $U_j$ into $2^{j-1}$ intervals of length $\sqrt{T}$ and then apply the first estimate in Lemma \ref{shortsums} to each of these $2^{j-1}$ intervals. It follows that
\[
T^{1/2}\sum_{n\in U_j}\frac{|\lambda_f(n)|}{|n-\mathcal{T}_f|+\sqrt{T}}\ll 2^{-j} \sum_{n\in U_j}|\lambda_f(n)|\ll T^{1/2}.
\]
Since $J=O(\log T)$, we deduce that $\Sigma_2\ll T^{1/2}\log T$. A similar calculation for $n\in V_3$ implies that $\Sigma_3\ll T^{1/2}\log T$. Thus, when $\alpha=1$, the second big-$O$ term in \eqref{gl_2} is $O(T^{1/2}\log T)$. Similarly, when $\alpha=2$, a straightforward adaptation of our calculation shows that this error term is again $O(T^{1/2}\log T)$, so
\[
\int_T^{2T}Z_f(t)\,\mathrm{d}t
=\frac{2\pi}{\varepsilon_f^{1/2} N^{1/4}} \,e\!\left(\frac{k}{8}\right) 
\sum_{\mathcal{T}_f<n\leq 2\mathcal{T}_f}
\lambda_f(n) \,e\!\left(-\frac{n}{\sqrt{N}}\right) + O(T^{1/2}\log T).
\]
Summing these estimates over the dyadic intervals $\{[2^{-j-1}T,2^{-j}T]\}_{j\geq 0}$, we obtain that
\begin{equation}\label{modint}
\int_{0}^{T}Z_f (t) \,\mathrm{d}t 
= \frac{2\pi}{\varepsilon_f^{1/2} N^{1/4}} \,e\!\left(\frac{k}{8}\right)  \sum_{n\leq \mathcal{T}_f} \lambda_f(n) \,e\!\left(-\frac{n}{\sqrt{N}}\right)+O(T^{1/2}\log T).
\end{equation}
A classical result of Wilton \cite{wilton} (see \cite[Theorem 5.3]{iwaniec1} for full generality) states that
\[
\sum_{n\leq x}\lambda_f(n) \, e(n\alpha) \ll_f x^{1/2}\log x
\]
uniformly for $\alpha \in \mathbb{R}$. Therefore, the sum over $n$ in \eqref{modint} is $O(T^{1/2}\log T).$
This completes the proof of Theorem \ref{thm:holo}. 

\subsubsection{Proof of Theorem \ref{thm:Maass}} Since the proof of this theorem is very similar to the proof of Theorem \ref{thm:holo}, we simply indicate the necessary changes. If $\varphi$ is a primitive Hecke--Maass cusp form of level $N$ with nebentypus  $\chi$, then Proposition \ref{mainlemma} implies that
\begin{equation*}
\begin{split}
\int_T^{2T}Z_\varphi(t)\,\mathrm{d}t
&=\frac{2\pi}{\varepsilon_\varphi^{1/2} N^{1/4}}  \,e\!\left(\frac{\delta}{4}\right) 
\sum_{\mathcal{T}_\varphi<n\leq 2\mathcal{T}_\varphi}
\lambda_\varphi(n) \,e\!\left(-\frac{n}{\sqrt{N}}\right) + O(T^{1/2}\log T)\\
&\qquad \qquad +\,O\!\left( T^{1/2}\max_{\alpha\in \{1,2\}}\left\{\sum_{\frac{3\alpha}{4}\mathcal{T}_\varphi \leq n\leq \frac{5\alpha}{4}\mathcal{T}_\varphi}
\frac{|\lambda_\varphi(n)|}{|n\!-\!\alpha \mathcal{T}_\varphi|+\sqrt{T}}\right\}\right),
\end{split}
\end{equation*}
where $\mathcal{T}_\varphi = \sqrt{N}T/(2\pi)$. Estimating the second big-$O$ term as we did in Theorem \ref{thm:holo} but instead using the second estimate in Lemma \ref{shortsums}, we deduce that
\[
\int_T^{2T}Z_\varphi(t)\,\mathrm{d}t
=\frac{2\pi}{\varepsilon_\varphi^{1/2} N^{1/4}}  \,e\!\left(\frac{\delta}{4}\right) 
\sum_{\mathcal{T}_\varphi<n\leq 2\mathcal{T}_\varphi}
\lambda_\varphi(n) \,e\!\left(-\frac{n}{\sqrt{N}}\right) + O(T^{\frac{11}{20}+\varepsilon}).
\]
Summing these estimates over the dyadic intervals $\{[2^{-j-1}T,2^{-j}T]\}_{j\geq 0}$, we obtain that
\begin{equation}\label{maassint}
\int_0^{T}Z_\varphi(t)\,\mathrm{d}t
=\frac{2\pi}{\varepsilon_\varphi^{1/2} N^{1/4}}  \,e\!\left(\frac{\delta}{4}\right) 
\sum_{ n\leq \mathcal{T}_\varphi}
\lambda_\varphi(n) \,e\!\left(-\frac{n}{\sqrt{N}}\right) + O(T^{\frac{11}{20}+\varepsilon}).
\end{equation}
By \cite[Theorem 3]{hafner} (see also \cite{ehs} or \cite[Theorem 8.1]{iwaniec2}), we have
\[
\sum_{n\leq x}\lambda_\varphi(n)\, e(n\alpha) \ll_\varphi x^{1/2}\log x
\]
uniformly for $\alpha \in \mathbb{R}$. Therefore, the sum over $n$ in \eqref{maassint} is $O(T^{1/2}\log T).$
This completes the proof of Theorem \ref{thm:Maass}.

\section{Some observations on the distribution of \texorpdfstring{$C_\chi$}{C}}\label{section:distribution}

Recall that
\[
    C_\chi = \frac{2^\frac{3}{4}\pi^\frac{1}{4}e\left(\frac{1}{16}\right)}{3q^{\frac{1}{2}}\tau(\chi)^\frac{1}{2}} S(\chi).
\]
From our evaluation of $S(\chi)$, we see that
\[
    \abs{C_\chi} \le A_q,
\]
where $A_q = \frac{2^{3/4}\pi^{1/4} 2^a}{3 q^{1/4}}$ with
\[
    a = \begin{cases}
        \omega(q) + 1/2, & \text{ if $8 || q$}, \\
        \omega(q) + 1, & \text{ if $16 || q$}, \\
        \omega(q) + 3/2, & \text{ if $32 | q$}, \\
    \end{cases}
\]
and equality is achieved when $q$ is square-full (and $C_\chi \ne 0$).
For square-full $q$, then, the only question about the distribution of $C_\chi$
is its sign.

\smallskip

%The proof of Theorem \ref{thm:GL1} involves a choice of square root of $\G_\chi(s)$ in the region $\{\mathrm{Re}(s) \ge 1/2\} \setminus [1,\infty)$. The argument of $\G_\chi$ is the same as the argument of $\tau(\chi)$ on the segment $[1/2,1)$, so this is equivalent to a choice of square root of the Gauss sum, and we do not see a way of making this choice that is not somewhat arbitrary. A somewhat natural choice, perhaps (which is made by the LMFDB, for example), is to choose the square root so that $Z_\chi(0) > 0$, supposing that it never vanishes. As the argument of $L(1/2, \chi)^2$ is the argument of $\tau(\chi)$, this is the same as setting the argument of $\tau(\chi)^{1/2}$ to the argument of $L(1/2, \chi)$.

The proof of Theorem \ref{thm:GL1} involves a choice of square root of $\G_\chi(s)$ in the
region $\{\mathrm{Re}(s) \ge 1/2\} \setminus [1,\infty)$. The argument of
$\G_\chi$ is the same as the argument of $\tau(\chi)$ on the segment $[1/2,1)$,
so this is equivalent to a choice of square root of the Gauss sum, and we do not see a
way of making this choice that is not somewhat arbitrary. As stated after Theorem \ref{thm:GL1}, we use the square root for which $Z_\chi(0) > 0$ when $L(\frac12,\chi)\ne0$, which is also the choice made by the LMFDB. As the argument of $L(1/2, \chi)^2$ is the argument of $\tau(\chi)$,
this is the same as setting the argument of $\tau(\chi)^{1/2}$ to the
argument of $L(1/2, \chi)$.

\smallskip

Somewhat curiously, it seems that with this choice,
$C_\chi$ has a slight (presumably asymptotically negligible) positivity bias.
For example, up to modulus 500,000, there are 226,470,976 characters for which
$C_\chi$ is nonzero. We find that $C_\chi > 0 $ for 117,548,676
characters, which is about 51.9\% of the time. Up to modulus one million there
are 880,161,637 characters for which $C_\chi \ne 0$, and 51.53\% of
them are positive. Table \ref{table:Cchi-pos-neg} lists the number
of positive and negative values of $C_\chi$ for some square-full numbers with
small prime factors.

\begin{table}
\begin{tabular}{c|c|c|c||c|c|c|c}
    $q$ & \# pos. & \# neg. & \% pos. &
    $q$ & \# pos. & \# neg. & \% pos. \\
    \hline
    $2^3$ & 1 & 0 & 100              & $8\cdot5^{2}$ & 3 & 1 &            75    \\
    $2^4$ & 1 & 0 & 100              & $8\cdot5^{3}$ & 13 & 7 &           65    \\
    $2^5$ & 1 & 0 & 100              & $8\cdot5^{4}$ & 55 & 45 &          55    \\
    $2^6$ & 2 & 0 & 100              & $8\cdot5^{5}$ & 271 & 229 &        54.2  \\
    $2^7$ & 3 & 1 & 75               & $8\cdot5^{6}$ & 1331 & 1169 &      53.24 \\
    $2^8$ & 5 & 3 & 62.5             & $8\cdot5^{7}$ & 6425 & 6075 &      51.40 \\
    $2^9$ & 12 & 4 & 75              & $8\cdot5^{8}$ & 31975 & 30525 &    51.16 \\
    $2^{10}$ & 21 & 11 & 65.625      & $8\cdot5^{9}$ & 158315 & 154185 &  50.66 \\
    $2^{11}$ & 40 & 24 & 62.5        & $8\cdot5^{10}$ & 788433 & 774067 & 50.46 \\
    $2^{12}$ & 75 & 53 & 58.59       &      $8\cdot7^{2} $& 7 & 2               & 77.78     \\
    $2^{13}$ & 150 & 106 & 58.59     &      $8\cdot7^{3} $ & 40 & 23            & 63.49     \\
    $2^{14}$ & 283 & 229 & 55.27     &      $8\cdot7^{4} $ & 241 & 200          & 54.65     \\
    $2^{15}$ & 574 & 450 & 56.05     &      $8\cdot7^{5} $ & 1645 & 1442        & 53.29     \\
    $2^{16}$ & 1100 & 948 & 53.71    &      $8\cdot7^{6} $ & 11157 & 10452      & 51.63     \\
    $2^{17}$ & 2129 & 1967 & 51.98        & $8\cdot7^{7} $ & 77040 & 74223      & 50.93     \\
    $2^{18}$ & 4258 & 3934 & 51.98        & $8\cdot7^{8} $ & 534064 & 524777    & 50.44     \\
    $2^{19}$ & 8436 & 7948 & 51.49        &     $8\cdot11^{2}$ & 16 & 9             & 64    \\
    $2^{20}$ & 16787 & 15981 &      51.23 &     $8\cdot11^{3}$ & 159 & 116          & 57.82 \\
    $2^{21}$ & 33658 & 31878 &      51.36 &     $8\cdot11^{4}$ & 1599 & 1426        & 52.86 \\
    $2^{22}$ & 66585 & 64487 &      50.80 &     $8\cdot11^{5}$ & 17145 & 16130      & 51.53 \\
    $2^{23}$ & 133075 & 129069 &    50.76 &     $8\cdot11^{6}$ & 185397 & 180628    & 50.65 \\
    $2^{24}$ & 265419 & 258869 &    50.62 &     $8\cdot13^{2} $& 23 & 13 &          63.89   \\
    $2^{25}$ & 529188 & 519388 &    50.48 &     $8\cdot13^{3} $& 269 & 199 &        57.48   \\
    $2^{26}$ & 1056879 & 1040273 &  50.40    &  $8\cdot13^{4} $& 3167 & 2917 &      52.05   \\
    $2^{27}$ & 2111356 & 2082948 &  50.34    &  $8\cdot13^{5} $& 40351 & 38741 &    51.02   \\
    $8\cdot3^2$    & 1 & 0 & 100             &  $8\cdot13^{6} $& 519303 & 508893 &  50.51   \\
    $8\cdot3^3$    & 2 & 1 & 66.67            &  $8\cdot17^{2} $& 40 & 24 &          62.5    \\
    $8\cdot3^4$    & 7 & 2 & 77.78            &  $8\cdot17^{3} $& 584 & 504 &        53.68   \\
    $8\cdot3^5$    & 17 & 10 &          62.96&  $8\cdot17^{4} $& 9532 & 8964 &      51.54   \\
    $8\cdot3^6$    & 48 & 33 &          59.26&  $8\cdot17^{5} $& 159078 & 155354 &  50.59   \\
    $8\cdot3^7$    & 135 & 108 &        55.56&   $8\cdot3^2\cdot5^2$ & 3 & 1 &       75                      \\
    $8\cdot3^{8}$  & 381 & 348 &   52.26& $8\cdot3^2\cdot5^2\cdot7^2$ & 28 & 8 & 77.78  \\
    $8\cdot3^{9}$  & 1171 & 1016 & 53.54& $8\cdot3^2\cdot5^2\cdot7^2\cdot11^2$ & 539 & 361 & 59.89 \\
    $8\cdot3^{10}$ & 3436 & 3125 & 52.37& $8\cdot5^2\cdot7^2\cdot11^2$ & 507 & 393 & 56.33    \\
    $8\cdot3^{11}$ & 10185 & 9498 &51.74& $8\cdot7^2\cdot11^2\cdot13^2$ & 4289 & 3811 & 52.95 \\
    $8\cdot3^{12}$ & 30295 & 28754 & 51.30& $8\cdot11^2\cdot13^2\cdot17^2$ & 29815 & 27785 & 51.76 \\
    $8\cdot3^{13}$ & 89907 & 87240 & 50.75 & $16\cdot11^5$ & 17128 & 16147 & 51.47 \\
    $8\cdot3^{14}$ & 268801 & 262640 & 50.58 &  $16\cdot13^5$ & 40337 & 38755 & 51.00 \\
    $8\cdot3^{15}$ & 803747 & 790576 & 50.41 & $16\cdot17^5$ & 159825 & 154607 & 50.83 \\
    \hline
\end{tabular}
\\
\medskip
\caption{Some statistics on the positivity and negativity of $C_\chi$,
where the square root of $\tau(\chi)$ is chosen so that $Z_\chi(0) > 0$.}
\label{table:Cchi-pos-neg}
\end{table}

\smallskip

We now consider more generally the distribution of $C_\chi$ when
$q$ is not square-full. First we observe that from the fact that
$C_\chi$ must be a real number, one can make some deductions about signs
of ratios of Gauss sums.

\begin{proposition}\label{prop:gauss-sums-mod-2}
    Suppose that $\psi_2$ is a primitive Dirichlet character \textup{(mod $2^a$)}
    with $a \ge 4$
    and that $\chi_2$ is a primitive character \textup{(mod $2^{a-1}$)} such that
    $\psi_2^2 = \chi_2$. Let $\alpha$ be an integer such that
    $\psi_2(5) = e(\alpha/2^{a-2})$ (equivalently, $\chi_2(5) = e(\alpha/2^{a-3})$).
    If $a > 5$, and $a$ is even or $\alpha \equiv \pm 1 \mod 8$, then
    \[
        \frac{\tau(\psi_2)^2}{\tau(\chi_2)} = 2^{a/2 + 1/2} \, e\!\left(\qr{-4}{\alpha}\frac{1}{8}\right).
    \]
    If $a > 5$ is odd and $\alpha \equiv \pm 3 \mod 8$, then
    \[
    \frac{\tau(\psi_2)^2}{\tau(\chi_2)} = 2^{a/2 + 1/2} \, e\!\left(\frac{\alpha}{8}\right).
    \]
    If $a = 5$, then
    \[
        \frac{\tau(\psi_2)^2}{\tau(\chi_2)} = 8\,e\!\left(\frac{1}{2} - \frac{\alpha}{8}\right).
    \]
    If $a = 4$, then
    \[
        \frac{\tau(\psi_2)^2}{\tau(\chi_2)} = 4\sqrt{2}\,e\!\left(-\qr{-4}{\alpha}\frac{1}{8}\right).
    \]
\end{proposition}
\begin{proof}
    The small cases can be computed directly, so suppose that $a > 6$.
    Choose a prime $p>5$ such that $-p \equiv \alpha \mod 8$. Let $\psi_p$
    be a primitive even non-quadratic character (mod $p$) and $\chi_p = \psi_p^2$, and
    put $\chi = \chi_2\chi_p$ and $\psi = \psi_2 \psi_p$, hence $\chi$ is a primitive character modulo $q=2^{a-1}p$. Then
    $S(\chi) = S_2(\psi, -1) \ne 0$ and we have the factorization
    \[
        S_2(\psi, -1) = \psi_2(p)^2 S_2(\psi_2, -p) \psi_p(2^a)^2 S_2(\psi_p, -2^a).
    \]
    If $a$ is even or if $2$ is a square (mod $p$), then
    \[
        S_2(\psi, -1) = \psi_2(p)^2 S_2(\psi_2, -p) \psi_p(2^a) S_2(\psi_p, -1)
    \]
    and from our evaluation of $S_2$ we find that
    \[
    \begin{split}
    S_2(\psi, -1)^2
        &= \psi_2(p)^4 S_2(\psi_2, -p)^2 \psi_p(2^a)^2 S_2(\psi_p, -1)^2 \\
        &=\psi_2(p)^4 (4\psibar_2(-p) \tau(\psi_2))^2 \psi_p(2^a)^2 (\tau(\psi_p) + \epsilon_p(-1) \tau(\epsilon_p \psi_p))^2  \\
        &=16 \psi_2(p)^2 \tau(\psi_2)^2 \psi_p(2^a)^2 (\tau(\psi_p) + \epsilon_p(-1) \tau(\epsilon_p \psi_p))^2.
    \end{split}
    \]
    Now, 
    \[
        C_\chi^2 = \frac{2^{3/2}\pi^{1/2}e(1/8)}{9q \tau(\chi)}S_2(\psi,-1)^2.
    \]
    Factoring the Gauss sum as $\tau(\chi) = \chi_2(p)\chi_p(2^{a-1})\tau(\chi_2) \tau(\chi_p)
    = \psi_2(p)^2\psi_p(2^{a-1})^2\tau(\chi_2) \tau(\chi_p)$ and putting in our evaluation
    of $S_2$, we have
    %\[
    %    \begin{split}
    %        \tau(\chi) &= \chi_2(p)\chi_p(2^{a-1})\tau(\chi_2) \tau(\chi_p) \\
    %                   &= \psi_2(p)^2\psi_p(2^{a-1})^2\tau(\chi_2) \tau(\chi_p)
    %    \end{split}
    %\]
    %and putting in our evaluation of $S_2$, we have
    \[
        C_\chi^2 = 16 \frac{2^{3/2}\pi^{1/2}e(1/8)}{9q}
        \frac{\tau(\psi_2)^2}{\tau(\chi_2)} \psi_p(4)
            \frac{(\tau(\psi_p) + \epsilon_p(-1) \tau(\epsilon_p \psi_p))^2}{\tau(\chi_p)}.
    \]
    Now, upon inserting $q=2^{a-1}p$, multiplying and dividing by a
    quadratic Gauss sum, and recalling (\ref{eq:S2jacobi}) we see that
    \[
    \begin{split}
        C_\chi^2 &= 16 \, e\!\left(\frac{1 - \epsilon_p(-1)}{8}\right)\frac{2^{3/2}\pi^{1/2}e(1/8)}{9\cdot2^{a-1} \sqrt{p} }
        \frac{\tau(\psi_2)^2}{\tau(\chi_2)}
            \left(\frac{\psi_p(4)}{\tau(\epsilon_p)}
            \frac{(\tau(\psi_p) + \epsilon_p(-1) \tau(\epsilon_p \psi_p))^2}{\tau(\chi_p)}\right)
                \\
            &= 32 \, e\!\left(\frac{2 - \epsilon_p(-1)}{8}\right)\frac{2^{3/2}\pi^{1/2}}{9\cdot2^{a-1} \sqrt{p} }
        \frac{\tau(\psi_2)^2}{\tau(\chi_2)}
            \left(\mathrm{Re}\left(\frac{J(\psi_p, \epsilon_p)}{\tau(\epsilon_p)}\right) + \epsilon_p(-1)\right).
                \\
    \end{split}
    \]
    We know that this quantity is a positive real number. Since $|J(\psi_p,\epsilon_p)/\tau(\epsilon_p)|=1$ and $C_\chi\ne0$, the last factor is nonzero and has the sign of $\epsilon_p(-1)$. So if
    $\epsilon_p(-1) = 1$, then
    \[
        e(1/8) \, \tau(\psi_2)^2/\tau(\chi_2)
    \]
    is positive. If $\epsilon_p(-1) = -1$, then
    \[
        e(-1/8) \, \tau(\psi_2)^2/\tau(\chi_2)
    \]
    is positive.

\smallskip

This covers the case when $a$ is even or 2 is a square (mod
$p$). When $a$ is odd and $2$ is not a square, we also need to
consider $\epsilon_p(-2)$, which is determined by $p$ (mod 8), which
is determined by $\alpha$ (mod 8), and similar calculations lead to the
second case of the proposition.
\end{proof}

Having evaluated the sign of $\tau(\psi_2)^2/\tau(\chi_2)$, we can
examine $C_\chi$ for characters modulo a power of 2 times an
odd prime squared to do the same for odd primes. We omit the
details of the following, which can also be deduced from
work of Pigno and Pinner \cite[Corollary 1.2]{pp-binomialprime}
and the observation that
$\frac{\tau(\psi_p)^2}{\tau(\psi_p^2)} = J(\psi_p,\psi_p)$.
\begin{proposition}
    Suppose that $\psi_p$ is a primitive Dirichlet character \textup{(mod $p^a$)}, where
    $p$ is an odd prime and $a > 1$. Let $\alpha$ be an integer such that
    $\psi_p(p\!+\!1) = e(\alpha/(p^{a-1}))$. Then
    \[
        \frac{\psi_p(4)\,\tau(\psi_p)^2}{\tau(\psi_p^2)} =
        \begin{cases}
            p^{a/2}, & \textrm{ if $a$ is even}, \\
            \epsilon_p(\alpha) \, p^{a/2}, & \textrm{ if $a$ is odd and $p \equiv 1 \mod 4$}, \\
            -i \, \epsilon_p(\alpha) \, p^{a/2}, & \textrm{ if $a$ is odd and $p \equiv 3 \mod 4$}. \\
        \end{cases}
    \]
    %or maybe we want to write
    %\[
    %    \frac{\psi_p(4)\tau(\psi_p)^2}{\tau(\psi_p^2)} =
    %    \begin{cases}
    %        p^{a/2} & \textrm{ if $a$ is even} \\
    %        \qr{-4}{p}\tau(\epsilon_p)\epsilon_p(\alpha)p^{(a-1)/2} & \textrm{ if $a$ is odd}. \\
    %    \end{cases}
    %\]
\end{proposition}
\begin{proof}
    If $\alpha$ is a square (mod $p$) or $2$ is not a square (mod $p$), then we
    can choose $q = 2^{a_2}p^a$ for some value of $a_2 \ge 6$ such that $2^{a_2 + 1} \alpha$
    is a square (mod $p$).
    If $\alpha$ is not a square and $2$ is, then we can choose
    some $q = 2^{a_2}p^ap'$ for some other prime $p'\ge 5$ which is not a square (mod $p$), so
    that $2^{a_2 + 1} \alpha p'$ is a square (mod $p$).

\smallskip

    Then for a primitive even $\chi_2$ (mod $2^{a_2}$) with $\chi_2(5) =
    e(\alpha_2/2^{a_2-2})$, $\alpha_2 \equiv -p^ap' \mod 8$ (with $p'$
    possibly equal to 1) and any even primitive character $\chi_{p'}$ (mod $p'$), the
    product $\chi = \chi_2 \psi_p^2 \chi_{p'}$ will satisfy the conditions of
    Proposition \ref{prop:S1} and $C_\chi \ne 0$. Then we can proceed as in
    the proof of the previous proposition and use the fact that $C_\chi^2$ is
    positive to get the result.
\end{proof}

%\Bnote{Confession: I have neither done the calculation described in the proof
%nor tried to deduce this from Pigno and Pinner. I think in Pigno and Pinner's corollary,
%is it not obvious that there exists such a simple statement. In our case, it is clear
%that one can get something like this; it seems like it might have depended on
%$p$ mod 8, though.}

%\Znote{The proofs of Prop 6 and 7 are beautiful!}

%\Znote{At first glance, I thought Proposition 7 looked routine, since it follows naturally from the analysis of $C_\chi$. But after you pointed out Pigno and Pinner’s corollary, I now feel that our Proposition 7 is {\bf truly elegant}. I completely agree that PP's corollary is far less effective than our version, as it involves a third primitive character and the relative multiplicative orders.}

Finally, putting these propositions together, we give an alternate expression for $C_\chi$.
\begin{proposition}\label{prop:Cchi-evaluation}
    Suppose that $C_\chi \ne 0$ for a primitive character $\chi \mod q$ with $q
    = \prod_{p|q} p^{e_p}$. Let $\psi = \prod_{p|q} \psi_p$ be a primitive
    character \textup{(mod $2q$)} such that $\psi^2 = \chi$. Write $q_2 =
    \prod_{p || q} p$. Then
\[
C_\chi^2 = \frac{2^{\frac{3}{2} + 2a} \pi^{\frac{1}{2}}}{9\, q^{\frac{1}{2}}}
    \prod_{p|q_2}
    \epsilon_p(-2q/p)
    \prod_{p|q_2}
        \left(2\mathrm{Re}\left(\frac{J(\psi_p,\epsilon_p)}{\tau(\epsilon_p)}\right) + 2\epsilon_p(-2q/p) \right),
\]
%where $A_q = \frac{2^{3/4}\pi^{1/4} 2^a}{3 q^{1/4}}$ with
where
\[
    a = \begin{cases}
        \omega(q/q_2) + 1/2, & \text{ if $8 || q$,} \\
        \omega(q/q_2) + 1, & \text{ if $16 || q$,} \\
        \omega(q/q_2) + 3/2, & \text{ if $32 | q$}. \\
    \end{cases}
\]
\end{proposition}
\begin{proof}
We work as in the proofs of the previous propositions, and as in Equation
(\ref{eq:S2-evaluation}), factoring $S_2(\psi,-1)$ and the Gauss sum $\tau(\chi)$.
We will find that
\[
C_\chi^2 = B \frac{2^{\frac{3}{2} + 2a} \pi^{\frac{1}{2}}}{9\, q^{\frac{1}{2}}}
    \prod_{p|q_2}
        \left(2\mathrm{Re}\left(\frac{J(\psi_p,\epsilon_p)}{\tau(\epsilon_p)}\right) + 2\epsilon_p(-2q/p) \right)
\]
for some root of unity $B$. Since $C_\chi^2>0$ and each factor in the product is nonzero and has the sign of $\epsilon_p(-2q/p)$, we
know that $B = \prod_{p|q_2} \epsilon_p(-2q/p)$.
\end{proof}

\begin{figure}
    \includegraphics[width=\textwidth]{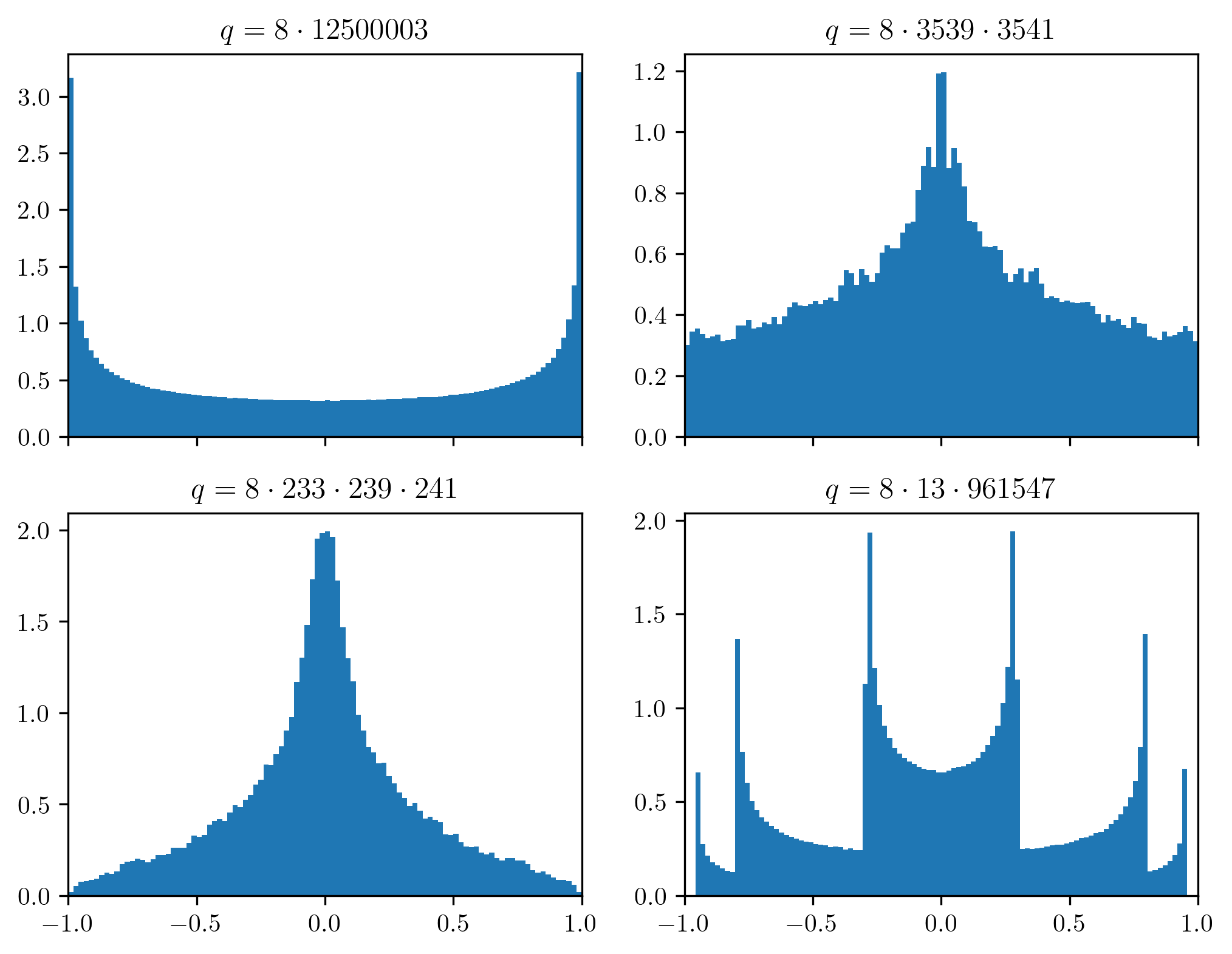}
    \caption{Normalized histograms of 
        $C_\chi/A_q$
        for the nonzero values of $C_\chi$ for primitive characters,
        mod $q$ for various $q$. 
    }
    \label{fig:Cchi-distribution}
\end{figure}

Now suppose that $q = 2^m p$ for an odd prime $p$.
As $\psi_p$ varies, the ratio
$J(\psi_p, \epsilon_p)/\tau(\epsilon_p)$ becomes equidistributed on the unit
circle as $p\to\infty$ (see the remark following Proposition \ref{prop:S2-oddprime}), and we do not see any reason why the vanishing or nonvanishing
of $S(\chi)$ should interact with this equidistribution.
%It is natural to expect that as $\psi_p$ varies, the ratio $\frac{J(\psi_p, \epsilon_p)}{\tau(\epsilon_p)}$ is equidistributed on the unit circle, and we do not see any reason why the vanishing or nonvanishingof $S(\chi)$ should interact with this supposed equidistribution.
Thus it seems that, after shifting and scaling, $C_\chi^2$ should be distributed
like the real part of a random variable uniformly distributed on the unit
circle. Provided that the choice of square root behaves as if it is independent
of everything else, we should then also expect that
$C_\chi/A_q$
behaves like the real part of a uniformly distributed random variable on the
unit circle (using the identity $\cos 2x = 2(\cos x)^2 - 1$). Specifically,
this suggests the following conjecture.
%\begin{conjecture}
%Fix $a\ge 3$, and for a prime $p$ let $q=2^a p$. Then for every interval $I\subseteq[-1,1]$,
%    as $p \rightarrow \infty$ over the primes,
%    \[
%        \frac{
 %       \#\left\{\textrm{primitive }\chi \mod{2^a p} \textrm{ such that }  \frac{C_\chi}{A_q} \in [\alpha,\beta]\setminus\{0\} \right\}}
  %      {\#\left\{\textrm{primitive }\chi \mod{2^a p} \textrm{ such that } C_\chi \ne 0 \right\}}
   %         \longrightarrow
    %        \frac{1}{\pi}
    %        \int_\alpha^\beta \frac{1}{\sqrt{1 - x^2}}\mathrm{d}x.
    %\]
%\end{conjecture}

\begin{conjecture}
Fix $m\ge 3$, and for a prime $p$ let $q=2^m p$. For any subinterval $[\alpha,\beta]\subseteq[-1,1]$,
    as $p \rightarrow \infty$ over the primes, we have
    \[
        \frac{
        \#\left\{\textrm{primitive }\chi \mod{q} \textrm{ such that }  \frac{C_\chi}{A_q} \in [\alpha,\beta]\setminus\{0\} \right\}}
        {\#\left\{\textrm{primitive }\chi \mod{q} \textrm{ such that } C_\chi \ne 0 \right\}}
            \longrightarrow
            \frac{1}{\pi}
            \int_\alpha^\beta \frac{1}{\sqrt{1 - x^2}}\,\mathrm{d}x.
    \]
\end{conjecture}

More generally, if $q = q_1 q_2$, where $q_1$ is square-full (and divisible by 8)
and $q_2$ is square-free, then for $\omega(q_2)$ fixed, we should expect
that as long as the smallest prime factor of $q_2$ tends to infinity,
$C_\chi/A_q$ should be asymptotically distributed like the product of $\omega(q_2)$
independent copies of the real part of a uniform random variable on
the unit circle. That is, we guess
\[
    \lim_{\substack{P^-(q_2) \rightarrow \infty \\ \omega(q_2) = k}}
    \frac{
        \#\left\{\textrm{primitive }\chi \mod{q} \textrm{ such that } C_\chi/A_q \in [\alpha,\beta]\setminus\{0\} \right\}}
        {\#\left\{\textrm{primitive }\chi \mod{q} \textrm{ such that } C_\chi \ne 0 \right\}}
            =
            \int_\alpha^\beta \phi_k(x) \mathrm{d}x,
\]
where $P^-(q_2)$ denotes the smallest prime factor of $q_2$ and $\phi_k(x)$ is
the density function of the random variable $\prod_{j=1}^k \cos(2 \pi X_j)$,
where the $X_j$ are independent and uniformly distributed on $[0,1]$.
%\Bnote{There may not be a nice expression for $\phi_k(x)$. I asked Marton
%Balasz and he said ``Number theorists always ask the most unnatural
%questions.'' Also, it might not be obvious that this is what the answer had to
%be. The product and the real part don't commute, of course, and taking
%the product first might have been more natural. It might be worth
%not being lazy and work out the computations which obviously will lead
%to this conjecture, but I haven't done that.}

\smallskip

If $q_2$ has a fixed small prime factor, then the distribution will be
distorted by the small number of values taken by $S_2(\psi_p,r)$ for this
small prime factor, and the distribution may tend to something else.
Figure \ref{fig:Cchi-distribution} plots the distribution of
$C_\chi$ for a few $q$ around one hundred million with differing
factorization patterns.
\begin{figure}
    \includegraphics[width=\textwidth]{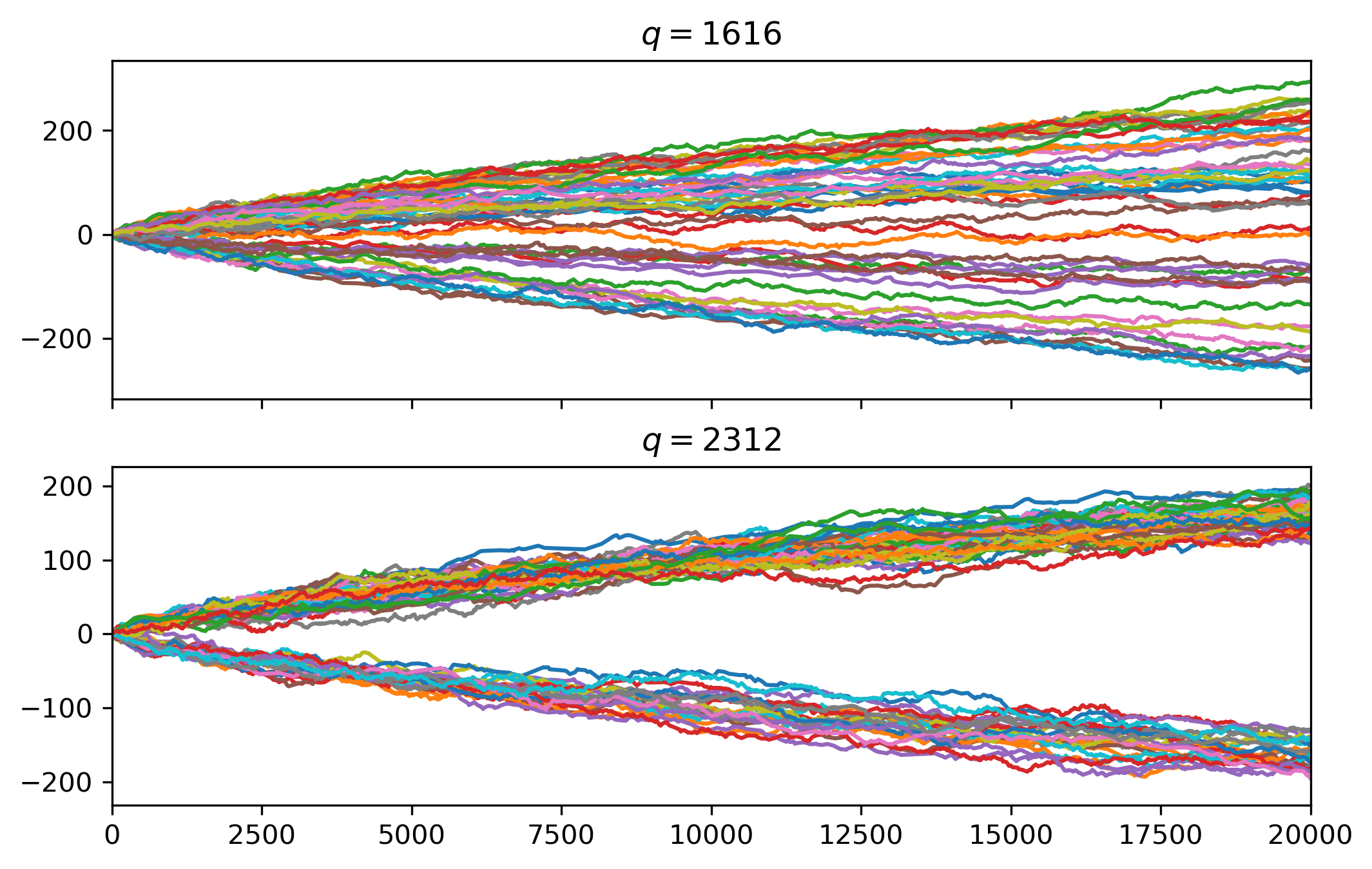}
    \caption{Plots of $M_\chi^+(T)-M_\chi^-(T)$, the difference between the amount of time
        that $Z_\chi(t)$ spends positive and negative,
        for the 49 characters (mod 1616) with $C_\chi \ne 0$
        and the 64 characters (mod 2312) with $C_\chi \ne 0$. 
    }
    \label{fig:Mplots}
\end{figure}

\begin{figure}
    \includegraphics[width=3in]{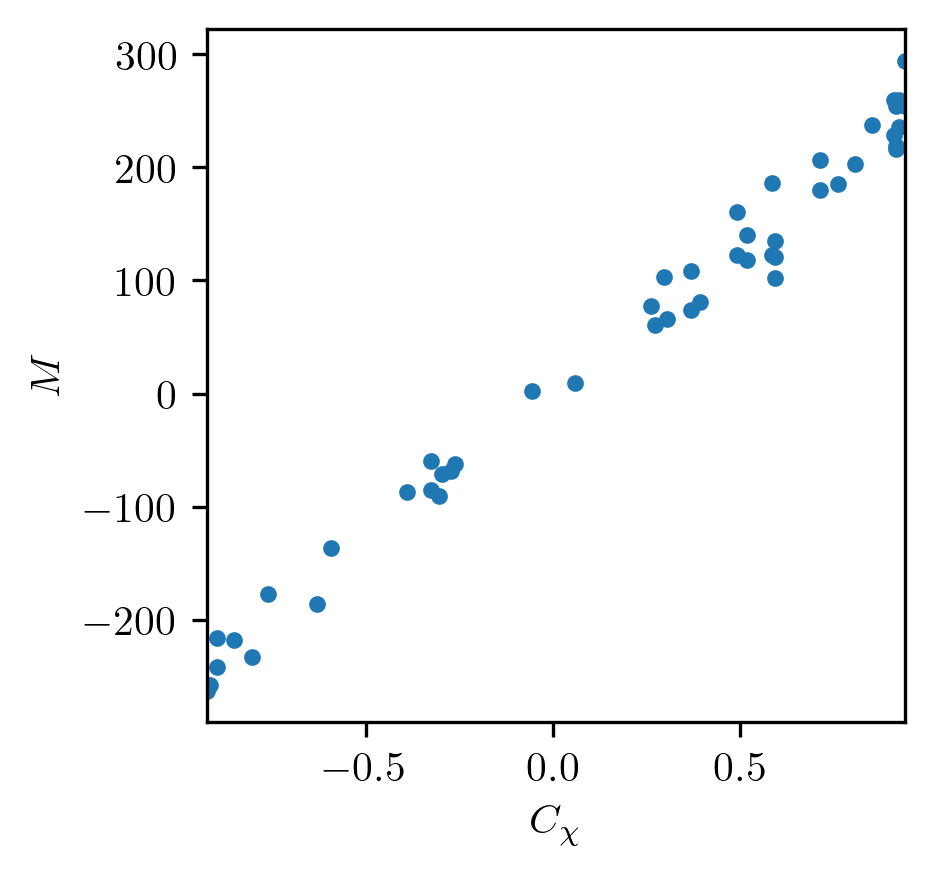}
    \caption{
        $C_\chi$ (horizontal axis) compared to
        $M_\chi = M^+_\chi(20000) - M^-_\chi(20000)$ (vertical axis)
        for the 49 characters mod 1616 for which $C_\chi \ne 0$.
    }
    \label{fig:MvsC}
\end{figure}

\smallskip

Finally, we briefly remark on the apparent relationship between $C_\chi$
and low order biases in the spacing between zeros of $L(s, \chi)$. Recall
from \eqref{eq:Mplusminus} that we define $M^+_\chi(T)$ and $M^-_\chi(T)$
respectively
to be the measures of the sets of $t \in [0,T]$ for which $Z_\chi(t)$
is positive and $Z_\chi(t)$ is negative. Equivalently, $M_\chi(T) := M_\chi^+(T) - M_\chi^-(T)$ measures the difference between the sum of the even-indexed gaps $g_{2n}$ and the sum of the odd-indexed gaps $g_{2n-1}$ between zeros of $L(s, \chi)$ on the critical line up to height $T$.

\smallskip

Figure \ref{fig:Mplots} presents plots of $M_\chi(T)$ for characters modulo
$1616 = 2^4\cdot101$ and $2312 = 2^3 \cdot17^2$, where $C_\chi$ does not
vanish. In the latter case, $\abs{C_\chi}$ is constant when it does not
vanish, as $2312$ is square-full. For 1616 we see a range of nonzero values
of $C_\chi$, and in Figure \ref{fig:MvsC} we see what seems to be a clear
linear relationship between $C_\chi$ and $M_\chi(T)$. A natural guess is
that there is some constant $A$ such that when $C_\chi \ne 0$ we have
\[
    M_\chi(T) \sim A \cdot C_\chi T^{3/4},
\]
though some logarithmic factors could appear as well.
%{\color{red}M: can we move all figures a few pages earlier to appear before the references start?}

%\clearpage

\appendix
\section{\texorpdfstring{A direct proof that $C_\chi$ is real}{A direct proof that C is real}}
\label{appendix}

In this appendix, we give a direct proof that the constant $C_\chi$ defined in Theorem \ref{thm:GL1} is real. 

\begin{lemma}
For every primitive character $\chi$ modulo $q$, we have
\[
S(\chi)=\frac{\tau(\chi)}{\sqrt{q}}\,
e(-1/8)\,\overline{S(\chi)}.
\]
In particular, $C_\chi$ is real.
\end{lemma}

\begin{proof}
When $q$ is odd, the formula holds because $S(\chi)=0$.
Hence, we may assume that $q$ is even. For primitive characters
$\chi$ modulo $q$, we have
\[
\chi(n)=\frac{\tau(\chi)}{q}
\sum_{a=1}^q\overline{\chi}(a) \, e(-an/q),
\]
and thus
\begin{equation} \label{aa}
S(\chi)
=\frac{\tau(\chi)}{q}
\sum_{a=1}^q  \overline{\chi}(a)
\sum_{n=1}^{2q}
e\!\left(-\frac{n^2\!+\!2an}{2q}\right)=\frac{\tau(\chi)}{q}
\sum_{a=1}^q\overline{\chi}(a) \,
e\!\left(\frac{a^2}{2q}\right)
\sum_{n=1}^{2q}e\!\left(-\frac{(n\!+\!a)^2}{2q}\right)
\end{equation}
upon completing the square. The quadratic Gauss sum can be evaluated as
\[
\sum_{n=1}^{2q}e\!\left(-\frac{(n\!+\!a)^2}{2q}\right) =
\sum_{n=1}^{2q} e\!\left(-\frac{n^2}{2q}\right)
=\begin{cases}
0, & q\text{ odd},\\
2\sqrt{q}\,e(-1/8), & q\text{ even}.
\end{cases}
\]
Since $q$ is even, we have
\[
\sum_{a=1}^{q} \overline{\chi}(a) \,
e\!\left(\frac{a^2}{2q}\right)
=\frac12 \, \overline{S(\chi)},
\]
and therefore, inserting the previous two identities into \eqref{aa}, we conclude that
\[
S(\chi)=\frac{\tau(\chi)}{\sqrt{q}}\,
e(-1/8)\,\overline{S(\chi)}.
\]
Consequently, for either choice of square root, \[
\frac{e(1/16) \, S(\chi)}{\sqrt{\tau(\chi)}}\in\mathbb{R},
\]
and hence $C_\chi$ is real. This completes the proof of the lemma.
\end{proof}

\section*{Acknowledgments} 
\noindent JB is supported by the Heilbronn Institute for Mathematical Research.
ZG was supported in part by a fellowship from the Fields Institute.
MBM was supported in part by the NSF grant DMS 2401461. The authors would like
to thank Valentin Blomer, Andy Booker, Brian Conrey, Jerzy Kaczorowski, and the referees for a
number of useful comments that helped to improve the paper.

\hbadness=10000
\bibliographystyle{amsplain}

\begin{thebibliography}{10}

%\bibitem{aggarwal}{\sc K.~Aggarwal}, 
%\emph{Weyl bound for $\mathrm{GL}(2)$ in $t$-aspect via a trivial delta method},  
%J. Number Theory {\bf 208} (2020), 72--100.

\bibitem{bew}{\sc B.~C.~Berndt, R.~J.~Evans, and K.~S.~Williams},
\emph{Gauss and Jacobi Sums}, Canadian Mathematical Society series of monographs and advanced texts, Wiley, 1998.

%\bibitem{bmn}{\sc A.~R.~Booker, M.~B.~Milinovich, and N.~Ng},
%\emph{Subconvexity for modular form $L$-functions in the $t$ aspect}, 
%Adv.~Math. {\bf 341} (2019), 299--335.

\bibitem{cech}{\sc M.~\v{C}ech},
\emph{First moment of quadratic Dirichlet $L$-functions with secondary terms}, preprint, {\tt \href{https://doi.org/10.48550/arXiv.2509.03197}{arXiv:2509.03197v2}[math.NT]}, 2025.


\bibitem{cn}{\sc K.~Chandrasekharan and R.~Narasimhan},
\emph{Functional equations with multiple gamma factors and the average order of arithmetical functions}, Ann. of Math. {\bf 76(1)} (1962), 93--136.


\bibitem{chowla}{\sc S.~Chowla},
\emph{The Riemann hypothesis and Hilbert's tenth problem}, Norske Vid. Selsk. Forh. (Trondheim) {\bf 38} (1965), 62--64.

\bibitem{cfkrs}{\sc J.~B.~Conrey, D.~W.~Farmer, J.~P.~Keating, M.~O.~Rubinstein, and N.~C.~Snaith}, 
\emph{Integral Moments of $L$-functions}, Proc.~London Math.~Soc.~(3) {\bf 91} (2005), 33--104.

\bibitem{diaconu}{\sc A.~Diaconu}, 
\emph{On the third moment of $L(\tfrac{1}{2},\chi_d)$ I: The rational function field case}, J. Number Theory {\bf 198} (2019), 1--42.

\bibitem{dw}{\sc A.~Diaconu and I.~Whitehead}, 
\emph{On the third moment of $L(\tfrac{1}{2},\chi_d)$ II: the number field case}, J. Eur. Math. Soc. (JEMS) {\bf 23} (2021), no. 6, 2051--2070. 

\bibitem{ehs}{\sc C.~Epstein, J.~L.~Hafner, and P.~Sarnak},
\emph{Zeros of $L$-functions attached to Maass forms}, Math. Z. {\bf 190} (1985), 113--128.

\bibitem{evans}{\sc R.~J.~Evans}, 
\emph{Generalizations of a theorem of Chowla on Gaussian sums}, Houston J. Math. {\bf 3} (1977), no. 3, 343--349.

\bibitem{florea}{\sc A.~M.~Florea}, 
\emph{Improving the error term in the mean value of $L(\tfrac{1}{2},\chi)$ in the hyperelliptic ensemble}, Int. Math. Res. Not. IMRN (2017), no. 20, 6119--6148.

\bibitem{gz} {\sc P.~Gao and L.~Zhao},
\emph{First moment of quadratic Hecke $L$-functions with lower order term}, J. Number Theory {\bf 289} (2026), 160--185.

\bibitem{gonek}{\sc S.~M.~Gonek},
\emph{Mean values of the Riemann zeta-function and its derivatives},
Invent. Math. {\bf 75} (1984), 123--141. 

\bibitem{gi}{\sc S.~M.~Gonek and A.~Ivi\'c}, 
\emph{On the distribution of positive and negative values of Hardy's $Z$-function}, J. Number Theory {\bf 174} (2017), 189--201.

\bibitem{hafner}{\sc J.~L.~Hafner},
\emph{Some Remarks on Odd Maass Wave Forms}, Math. Z. {\bf 196} (1987), 129--132.

\bibitem{hamdar}{\sc M.~H.~Hamdar},
\emph{Hecke $L$-functions away from the central line}, Math. Ann. {\bf 394} (2026), no. 3, Paper No.~74.

%\bibitem{hb1} {\sc D. R. Heath-Brown}, \textit{Hybrid bounds for Dirichlet $L$-functions},
%Invent. Math., \textbf{47} (1978), no. 2, 149--170.

%\bibitem{hb2} {\sc D.~R.~Heath-Brown}, 
%\emph{Hybrid bounds for Dirichlet $L$-functions. II}, 
%Quart. J. Math. Oxford Ser. (2) {\bf 31} (1980), no. 122, 157--167.

\bibitem{ivic}{\sc A.~Ivi\'{c}},
\textit{On the integral of Hardy's function}, Arch. Math. {\bf 83(1)} (2004), 41--47.

%\bibitem{ivic}{\sc A. Ivi\'{c}},
%\textit{On some problems involving Hardy's function},
%Cent. Eur. J. Math. \textbf{8(6)} (2010), 1029--1040.


%\bibitem{huxley}{\sc H.~N.~Huxley},
%\emph{Area, Lattice Points and Exponential Sums},
%Oxford University Press, 1996.

\bibitem{iwaniec1}{\sc H.~Iwaniec},
\emph{Topics in classical automorphic forms}, 
American Mathematical Society, 1997.

\bibitem{iwaniec2}{\sc H.~Iwaniec},
\emph{Spectral methods of automorphic forms}, 
second edition, American Mathematical Society, 2002.



\bibitem{IK} {\sc H. Iwaniec and E. Kowalski}, 
\emph{Analytic Number Theory}, American Mathematical Society Colloquium Publications, 53. American Mathematical Society, Providence, RI, 2004.

%\bibitem{iwaniec3}{\sc H.~Iwaniec},
%\emph{Lectures on the Riemann Zeta-function}, 
%American Mathematical Society (2014).

%\bibitem{jls}
%{\sc J.~J\"a\"asaari, S.~Lester, A.~Saha},
%\emph{On fundamental Fourier coefficients of Siegel cusp forms of degree $2$}, J. Inst. Math. Jussieu, {\bf 22}(4) (2023), 1819--1869.

\bibitem{jutila1}{\sc M.~Jutila},
\emph{Atkinson's formula for Hardy's function},
J. Number Theory \textbf{129(11)} (2009), 2853--2878.

\bibitem{jutila2}{\sc M.~Jutila},
\emph{An asymptotic formula for the primitive of Hardy's function},
Ark. Mat. \textbf{49(1)} (2011), 97--107.

\bibitem{Katz}{\sc N.~M.~Katz}, \emph{Gauss sums, Kloosterman sums, and monodromy groups}, Annals of Mathematics Studies, 116, Princeton Univ.~Press, Princeton, NJ, 1988.

\bibitem{ks1}{\sc J.~P.~Keating and N.~C.~Snaith}, 
\emph{Random matrix theory and $\zeta(\tfrac{1}{2}+it)$}, Comm. Math. Phys. {\bf 214} (2000), 57--89.

\bibitem{ks2}{\sc J.~P.~Keating and N.~C.~Snaith}, 
\emph{Random matrix theory and $L$-functions at $s=\frac{1}{2}$}, Comm. Math. Phys. {\bf 214} (2000), 91--110.

%\bibitem{kim}
%{\sc H. H, Kim}, \emph{Functoriality for the exterior square of $GL_4$ and the symmetric fourth of $GL_2$}, J. Amer. Math. Soc. {\bf 16} (2003), no. 1, 139--183. 

%\bibitem{kimshahidi}
%{\sc H. H. Kim and F. Shahidi}, \emph{Cuspidality of symmetric powers with applications}, Duke Math. J. {\bf 112} (2002), no. 1, 177--197.

%\bibitem{kis}{\sc H.~Kim and P.~Sarnak},
%\emph{Refined estimates towards to Ramanujan and Selberg conjectures}, J. Amer. Math. Soc. {\bf 16}, (2003), 175--181.



\bibitem{korolev1}{\sc M.~A.~Korolev},
\emph{On the primitive of the Hardy function $Z(t)$}, Dokl. Math. {\bf 75(2)} (2007), 295--298.

\bibitem{korolev2}{\sc M. A. Korolev}, 
\textit{On the integral of the Hardy function $Z(t)$},
Izv. Math. \textbf{72(3)} (2008), 429--478.

\bibitem{lmfdb:character-labels}{\sc The {LMFDB} Collaboration},
\emph{The {L}-functions and modular forms database, Dirichlet Character Labels},
\url{https://www.lmfdb.org/Character/Dirichlet/Labels}, (2024), Online, accessed 2 September 2024.

%\bibitem{mv}{\sc P. Michel and A. Venkatesh},
%\textit{The subconvexity problems for $\mathrm{GL}_2$},
%Publ. Math. Inst. Hautes \'{E}tudes Sci. \textbf{111} (2010), 171--271.


%\bibitem{montgomery}{\sc H.~L.~Montgomery},
%\emph{Ten Lectures on the interface between analytic number theory and harmonic analysis}, CBMS Regional Conference Series in Mathematics, {\bf 84}., American Mathematical Society, (1994).


%\bibitem{mv}{\sc H.~L.~Montgomery and R.~C.~Vaughan},
%\emph{Multiplicative Number Theory I: Classical Theory},
%Cambridge University Press, 2006.

%\bibitem{ir}{\sc K.~Ireland and M.~Rosen},
%\emph{A classical introduction to modern number theory},
%Second Edition, Springer--Verlag New York, 1998.


%\bibitem{pp}{\sc V.~Pigno and C.~Pinner},
%\emph{Twisted monomial Gauss sums modulo prime powers},
%Funct.~Approx.~Comment.~Math. \textbf{51(2)} (2014), 285--301.
\bibitem{pp-binomialprime}{\sc V.~Pigno and C.~G.~Pinner},
\emph{Binomial character sums modulo prime powers},
J. Th\'eor. Nombres Bordeaux {\bf 28} (2016), no.~1, 39--53.

\bibitem{platt}
{\sc D.~J. Platt},
\emph{Isolating some non-trivial zeros of zeta}, Math. Comp. {\bf 86} (2017), no.~307, 2449--2467.
%\bibitem{postnikov}{\sc A.~G.~Postnikov}, 
%\emph{On the sum of characters with respect to a modulus equal to a power of a prime number},
%Izv.~Aka.~Nauk SSSR.~Ser.~Mat. {\bf 19} (1955), 11--16.

%\bibitem{rankin}
%{\sc R. A. Rankin}, 
%\emph{Sums of powers of cusp form coefficients. II}, Math. Ann. {\bf 272} (1985), no. 4, 593--600.

\bibitem{rs}{\sc Z.~Rudnick and P.~Sarnak}, 
\emph{Zeros of principal $L$-functions and random matrix theory}, Duke Math. J. {\bf 81} (1996), 269--322.

\bibitem{shiu}{\sc P.~Shiu}, 
\emph{A Brun--Titchmarsh theorem for multiplicative functions}, J. Reine Angew. Math. {\bf 313} (1980), 161--170.


\bibitem{tenenbaum}{\sc G.~Tenenbaum},
\emph{Introduction to Analytic and Probabilistic Number Theory},
Third Edition, American Mathematical Society, 2015.

\bibitem{titchmarsh}{\sc E.~C.~Titchmarsh}, 
\emph{The theory of the Riemann zeta-function}, 
Oxford University Press, 1988.

%\bibitem{washington}{\sc L.~C.~Washington},
%\emph{Introduction to Cyclotomic Fields}, 
%Second Edition, Springer--Verlag New York, 1996.

\bibitem{wilton}{\sc J.~R.~Wilton},
\emph{A note on Ramanujan's arithmetical function $\tau(n)$},
Math. Proc. Cambridge Philos. Soc. {\bf 25} (1929), 121--129.

%\bibitem{xi}{\sc P.~Xi}, {\color{red} Do we cite this?}\emph{Equidistributions of Jacobi Sums}, preprint, {\tt \href{https://doi.org/10.48550/arXiv.1809.04286}{arXiv:1809.04286v1}[math.NT]}, 2018.

%%%%%
%%%%%%
%%%%%%%%%%%
\end{thebibliography}

\end{document}